\documentclass[a4paper]{amsart}
\usepackage[english]{babel}
\usepackage[utf8]{inputenc}
\usepackage[T1]{fontenc}
\usepackage{lmodern}

\usepackage[shortlabels]{enumitem}
\usepackage{booktabs}
\setlist[enumerate]{label={\arabic*.}}
\setlist[description]{font=\normalfont\itshape}

\usepackage[dvipsnames]{xcolor}
\usepackage[backref,hyperfootnotes=false]{hyperref}
\hypersetup{
  colorlinks=true,
  linkcolor=Maroon,
  citecolor=OliveGreen
}

\usepackage{todonotes}

\usepackage{cleveref}
\usepackage[nopatch=footnote]{microtype}

\usepackage{setspace}
\usepackage{amsmath,amsfonts,amsthm,mathtools,amssymb}
\usepackage{mleftright}
\mleftright % \left/\right without the spurious \mathinner spacing
\usepackage{aliascnt}
\usepackage{tikz}
\usetikzlibrary{arrows.meta}

\def\Z{\mathbf{Z}}

\def\C{\mathbf{C}}
\def\F{\mathbf{F}}
\def\PP{\mathbf{P}}

\DeclareMathOperator{\Irr}{Irr}
\DeclareMathOperator{\Rep}{rep}
\DeclareMathOperator{\ab}{ab}
\DeclareMathOperator{\Ind}{Ind}
\DeclareMathOperator{\Res}{Res}
\DeclareMathOperator{\Aut}{Aut}
\DeclareMathOperator{\Out}{Out}
\DeclareMathOperator{\Inn}{Inn}
\DeclareMathOperator{\Hom}{Hom}
\DeclareMathOperator{\Epi}{Epi}
\DeclareMathOperator{\End}{End}
\DeclareMathOperator{\supp}{supp}
\DeclareMathOperator{\St}{St}
\DeclareMathOperator{\Sch}{Sch}

\newcommand{\GL}{\mathrm{GL}}
\newcommand{\SL}{\mathrm{SL}}
\newcommand{\PGL}{\mathrm{PGL}}
\newcommand{\PSL}{\mathrm{PSL}}
\newcommand{\PSU}{\mathrm{PSU}}

\newcommand{\Sp}{\mathrm{Sp}}
\newcommand{\Alt}{\mathrm{Alt}}
\newcommand{\Sym}{\mathrm{Sym}}
\newcommand{\Cay}{\mathrm{Cay}}
\newcommand{\EL}{\mathrm{EL}}
\newcommand{\mon}{\mathrm{mon}}
\newcommand{\qp}{\mathrm{qp}}
\newcommand{\QP}{\mathrm{QP}}
\newcommand{\qs}{\mathrm{qs}}
\newcommand{\QS}{\mathrm{QS}}

\makeatletter
\newcommand{\fact}[1]{\item[(#1)]\phantomsection\def\@currentlabel{(#1)}\label{fact:#1}}
\makeatother
\newcommand{\factref}[1]{\ref{fact:#1}}

\newtheorem{theorem}{Theorem}[section]
\newtheorem*{theorem*}{Theorem}
\newaliascnt{lemma}{theorem}
\newtheorem{lemma}[lemma]{Lemma}
\aliascntresetthe{lemma}
\newaliascnt{proposition}{theorem}
\newtheorem{proposition}[proposition]{Proposition}
\aliascntresetthe{proposition}
\newaliascnt{corollary}{theorem}
\newtheorem{corollary}[corollary]{Corollary}
\aliascntresetthe{corollary}
\newaliascnt{conjecture}{theorem}

\aliascntresetthe{conjecture}

\theoremstyle{definition}
\newaliascnt{question}{theorem}
\newtheorem{question}[question]{Question}
\aliascntresetthe{question}
\newaliascnt{remark}{theorem}
\newtheorem{remark}[remark]{Remark}
\aliascntresetthe{remark}
\newaliascnt{definition}{theorem}

\aliascntresetthe{definition}
\newaliascnt{example}{theorem}
\newtheorem{example}[example]{Example}
\aliascntresetthe{example}

\crefname{theorem}{theorem}{theorems}
\Crefname{theorem}{Theorem}{Theorems}
\crefname{lemma}{lemma}{lemmas}
\Crefname{lemma}{Lemma}{Lemmas}
\crefname{proposition}{proposition}{propositions}
\Crefname{proposition}{Proposition}{Propositions}
\crefname{corollary}{corollary}{corollaries}
\Crefname{corollary}{Corollary}{Corollaries}
\crefname{question}{question}{questions}
\Crefname{question}{Question}{Questions}
\crefname{remark}{remark}{remarks}
\Crefname{remark}{Remark}{Remarks}
\crefname{definition}{definition}{definitions}
\Crefname{definition}{Definition}{Definitions}
\crefname{example}{example}{examples}
\Crefname{example}{Example}{Examples}
\crefname{conjecture}{conjecture}{conjectures}
\Crefname{conjecture}{Conjecture}{Conjectures}

\title[Representation and abelianization growth]{Representation and abelianization growth, particularly along expanding sequences}
\author{Urban Jezernik}
\thanks{This work has been supported by the Slovenian Research Agency program P1-0222 and grants J1-50001, J1-70033.}

\begin{document}
\baselineskip=13pt % comfortable reading

\begin{abstract}
We study two obstructions to a finite group $G$ having a generating $d$-tuple whose Cayley graph has spectral gap at least $\epsilon$. 
Lubotzky and Weiss proved that such a gap forces $\ab_k(G)$, the largest order of the abelianization of a subgroup of index at most $k$, to be $\exp(O_\epsilon(dk))$,
while
de la Harpe, Robertson and Valette proved that it forces $\Rep_k(G)$, the number of irreducible complex representations of degree at most $k$, to be $\exp(O_\epsilon(dk^2))$.
We show that the first bound implies the second, and in fact implies it with $k^2$ replaced by $k\log k$.
For this, we prove that every finite $d$-generated group satisfies the inequality
$\Rep_k(G)\leq\exp((d-1)k\log k+O(dk))\cdot\ab_k(G)$,
regardless of its spectral gap. 
The leading term $(d-1)k\log k$ is attained by powers of alternating groups, and it disappears for groups without large alternating sections. 
The proof reduces to counting the finite primitive subgroups of $\GL_m(\C)$, and we show that there are at most $\exp(O(\log^4m))$ of them up to conjugacy and scalars.
Wigderson asked whether $k^2$ in the second bound can be lowered to $k$, so our result answers his question up to a logarithm.
We do not know whether the logarithm is actually needed.
For powers of alternating groups, the natural candidates, we show that a spectral gap keeps the number of factors a fixed power below the maximum allowed by $d$-generation, and that random generators have spectral gap tending to $0$ already for superpolynomially many factors.
\end{abstract}

\maketitle

\section{Introduction}\label{sec:intro}

\subsection{Expanding sequences}\label{sub:expanding}

Expander graphs are sparse graphs that are nevertheless highly connected. More precisely, a sequence of finite regular graphs of fixed degree, with the number of vertices tending to infinity, is a family of \emph{expanders} if the \emph{spectral gap} $1-\lambda_2$ is bounded away from $0$ along the sequence, where $\lambda_2$ is the second largest eigenvalue of the normalized adjacency operator. Expanders have many applications in mathematics and computer science \cite{HLW06,Kow19,Lub12}.

Many explicit families of expanders arise as Cayley graphs of finite groups \cite{Lub94,Tao15}.
For a finite group $G$ and a tuple $\Sigma=(s_1,\dots,s_d)$ of elements of $G$, the \emph{Cayley graph} $\Cay(G,\Sigma)$ has vertex set $G$ and an undirected edge joining $g$ and $gs_i$ for each $g\in G$ and each $i$. It is $2d$-regular, with loops counted twice. Its normalized adjacency operator acts on functions $f\colon G\to\C$ by
\[
    (Af)(g) = \frac{1}{2d}\sum_{i=1}^d\left(f(gs_i)+f(gs_i^{-1})\right) .
\]
The spectral gap of $\Cay(G,\Sigma)$ is $1-\lambda_2$, where $\lambda_2$ is the second largest eigenvalue of $A$. Equivalently, it is the smallest eigenvalue of the normalized Laplacian $L=1-A$ on the functions orthogonal to the constants. It is positive exactly when $\Sigma$ generates $G$.
We call a sequence $(G_n)_{n\geq1}$ of finite groups with $|G_n|\to\infty$ \emph{$d$-expanding} if there is $\epsilon>0$ such that every $G_n$ has a generating $d$-tuple whose Cayley graph has spectral gap at least $\epsilon$. It is an \emph{expanding sequence} if it is $d$-expanding for some $d$. Some strong positive results are known. All nonabelian finite simple groups together form an expanding sequence \cite{KLN06,BGT10}. The groups $\PSL_2(\F_p)$ for primes $p$ also expand with their standard two generators, by Selberg's theorem \cite[Th.~4.4.2]{Lub94}, and even with two random generators with high probability, by the work of Bourgain and Gamburd \cite{BG08}.

By contrast, few tools are available to show that a given sequence of groups is \emph{not} expanding. Besides bounded generation, the main known obstructions bound the sizes of the abelian sections, due to Lubotzky and Weiss \cite{LW93}, and the number of small representations of the groups, due to de la Harpe, Robertson and Valette \cite{HRV93}. Both have since been studied by Meshulam and Wigderson \cite{MW04}, Pisier \cite{Pis17} and Sabatini \cite{Sab23,Sab25}. In this paper, we further explore these two obstructions and the relation between them. In particular, answering \cite[Question~4.1]{Sab25}, we show that for $d$-generated groups the obstruction of Lubotzky and Weiss is stronger than that of de la Harpe, Robertson and Valette.

\subsection{Representation and abelianization growth}\label{sub:obstructions}

Both obstructions concern two functions of a finite group $G$ and a parameter $k\geq1$. The first is the \emph{representation growth} of $G$, given by the number $\Rep_k(G)$ of irreducible complex representations of $G$ of degree at most $k$. It has been studied extensively for infinite groups, see the survey of Klopsch \cite{Klo13}. The second is the \emph{abelianization growth} of $G$, defined by Lubotzky and Martin \cite[Def.~2.4]{LM04} and studied in detail by Sabatini \cite{Sab23}, given by
\[
    \ab_k(G) = \sup_{|G:H|\leq k} |H:H'| .
\]
Here $H'=[H,H]$ is the commutator subgroup, and $|H:H'|$ is the size of the corresponding abelian quotient $H/H'$ of the subgroup $H$. Quotients of subgroups are called \emph{sections} of $G$, so $\ab_k(G)$ measures the largest abelian section of $G$ coming from a subgroup of small index.
We are interested in how these two functions grow along a sequence of finite groups $(G_n)_{n\geq1}$, for instance in the asymptotic behaviour of $\Rep_k(G_n)$ as a function of $k$ when $n\to\infty$, particularly when the sequence is expanding.

\subsubsection*{The two growths along expanding sequences}

We follow Pisier \cite[Sec.~2]{Pis17} and write $\mathcal{C}_{d,\epsilon}$ for the class of finite groups that have a generating $d$-tuple whose Cayley graph has spectral gap at least $\epsilon$. So a sequence is $d$-expanding exactly when it lies in $\mathcal{C}_{d,\epsilon}$ for some $\epsilon>0$. The best known bounds on the two functions for a group $G \in \mathcal{C}_{d,\epsilon}$ are due to Lubotzky--Weiss and de la Harpe--Robertson--Valette:\footnote{Implied constants in $O(\cdot)$ and $\Omega(\cdot)$ are absolute unless they carry a subscript, such as $O_\epsilon$.}
\begin{alignat*}{2}
  \ab_k(G) &\leq \exp(O_\epsilon(dk)) &\quad& \text{\cite[Th.~3.6]{LW93}}, \\
  \Rep_k(G) &\leq \exp(O_\epsilon(dk^2)) && \text{\cite[Prop.~IV]{HRV93}}.
\end{alignat*}
Both bounds come from packing arguments, and the exponents count dimensions (we recall them in \Cref{prop:LW} and \Cref{rem:methods}).
The first bound is sharp up to the constant in the exponent, as remarked by Lubotzky and \.Zuk \cite[Th.~5.13(1)]{LZ03}, see also Sabatini \cite[Sec.~2.3]{Sab25}.

\begin{example}[Lubotzky--\.Zuk, see {\Cref{ex:lower}} for details]\label{ex:lz}
The groups $\F_5[T]\rtimes T$ with $T=\PSL_2(\F_p)$ for primes $p\geq7$ with $p\equiv2\pmod5$ form an expanding sequence with respect to generating tuples of some bounded length $d_0$. The abelian normal subgroup $\F_5[T]$ has index $k=|T|$, so $\ab_k\geq5^k$. Each of its $5^k$ linear characters induces to a representation of degree $k$, so by Frobenius reciprocity it occurs in the restriction of an irreducible representation of degree at most $k$. Such a representation restricts to at most $k$ linear characters, so $\Rep_k\geq5^k/k$.
\end{example}

The argument in the example works in any finite group $G$ with a subgroup $H$ of index at most $k$. Applied to one with $|H:H'|=\ab_k(G)$, it gives the inequality of Lubotzky and Martin \cite[Lem.~2.6(a)]{LM04}, namely
\begin{equation}\label{eq:lower}
    \Rep_k(G) \geq \frac{1}{k} \cdot \ab_k(G).
\end{equation}
So representation growth is always at least abelianization growth, up to a factor $k$. The question is how much larger it can be.

\subsubsection*{Wigderson's question}

Along expanding sequences, the example above shows that $\log\Rep_k$ can grow linearly, and by the de la Harpe--Robertson--Valette bound, it grows at most quadratically. The actual order of growth is the subject of the following question.

\begin{question}[Wigderson, see {\cite[Sec.~1]{Pis17}} and {\cite[Question~4.3]{Sab25}}]\label{q:wigderson}
Let $d\geq2$ and $\epsilon>0$. What is the order of growth of
\[
    \sup_{G\in\mathcal{C}_{d,\epsilon}}\log\Rep_k(G)
\]
as $k\to\infty$? In particular, does it grow like $k$ rather than like $k^2$?
\end{question}

Meshulam and Wigderson \cite[Th.~1.4]{MW04} showed that the growth is linear for groups all of whose irreducible representations are induced from representations of bounded degree. This includes the monomial groups, but these are solvable by Taketa's theorem, so it does not reach the general case. More recently, Pisier \cite{Pis17} placed the question in a more general setting. There $G$ is an arbitrary (not necessarily finite) group with a fixed generating $d$-tuple, the regular representation of $G$ is replaced by a finite-dimensional unitary representation $\pi$ such that $\pi\otimes\bar\pi$ has a spectral gap, and one counts the irreducible constituents of $\pi$ of degree at most $k$.\footnote{For a finite group $G$ and its regular representation $\lambda$, this recovers $\Rep_k(G)$ and the spectral gap of the Cayley graph, since $\lambda\otimes\bar\lambda$ is a multiple of $\lambda$.} He showed that in this setting the growth is at least quadratic for all large $d$ \cite[Th.~2.11]{Pis17}. That is, for every $k$ there are a group $G$, possibly infinite, and a representation $\pi$ as above with $\exp(\Omega_\epsilon(dk^2))$ irreducible constituents of degree $k$. One might read this as evidence that $k^2$ is also the right order in Wigderson's question. We show, however, that this is not the case.

\subsection{Representation growth of expanders}\label{sub:expanders-main}

We answer \Cref{q:wigderson} up to a logarithm. Along expanding sequences, $\log\Rep_k$ grows at most like $k\log k$, and not like $k^2$.

\begin{theorem}\label{cor:quasilinear}
Let $d\geq2$ and $\epsilon>0$. Every $G\in\mathcal{C}_{d,\epsilon}$ satisfies
\[
    \Rep_k(G) \leq \exp\left((d-1)k\log k+O_\epsilon(dk)\right)
    \quad \text{for all } k\geq1 .
\]
\end{theorem}

It follows from the example above that for $d\geq d_0$ and $\epsilon$ small enough, we can sandwich the values of $\log\Rep_k$ along $\mathcal{C}_{d,\epsilon}$ between $\Omega(k)$ and $O(k\log k)$, so the growth is at least linear and at most quasilinear. We do not know whether the logarithm is really needed, and we return to this question in \Cref{sub:altfree-intro}.

\subsubsection*{An inequality between the two growths}
\Cref{cor:quasilinear} follows from an inequality between abelianization and representation growth in the direction opposite to \eqref{eq:lower}. Along expanding sequences, $\ab_k$ grows at most exponentially in $k$ by the bound of Lubotzky and Weiss. So it suffices to show that $\Rep_k$ cannot exceed $\ab_k$ by much more than that. We show that this indeed holds in every $d$-generated finite group, without any expansion assumption.

\begin{theorem}\label{thm:main}
Every finite $d$-generated group $G$ satisfies
\[
    \Rep_k(G) \leq \exp\left((d-1)k\log k+O(dk)\right)\cdot \ab_k(G)
    \quad \text{for all } k\geq1 .
\]
\end{theorem}

Multiplying the bound of \Cref{thm:main} by the bound $\ab_k(G)\leq\exp(O_\epsilon(dk))$ of Lubotzky and Weiss gives \Cref{cor:quasilinear}. More generally, \Cref{thm:main} settles how the two obstructions are related \cite[Question~4.1]{Sab25}. For a sequence of $d$-generated groups, the condition $\ab_k\leq\exp(O(k))$ of Lubotzky and Weiss implies $\Rep_k\leq\exp(O(k\log k))$, while the converse fails (\Cref{ex:window}).

\Cref{thm:main} improves a bound of Sabatini \cite[Prop.~2]{Sab23}, who proved that every $d$-generated finite group satisfies
\begin{equation}\label{eq:upper}
    \Rep_k(G) \leq \left( k!\cdot \ab_{(k+1)!}(G) \right)^{O(d)\cdot(\log \ab_{(k+1)!}(G)+k\log k)} .
\end{equation}
The latter requires control of the abelianization growth at index $(k+1)!$ rather than at $k$, and it depends on it through $\ab^{O(d\log\ab)}$. Improving \eqref{eq:upper} is \cite[Question~4.2]{Sab25}, and \Cref{thm:main} does so in both respects. The abelianization growth is evaluated at index $k$, and the dependence on it is linear.
Moreover, \Cref{thm:main} is sharp in three ways. First, the linear dependence on $\ab_k(G)$ cannot be improved, by \eqref{eq:lower}. Second, the index $k$ in $\ab_k(G)$ cannot be lowered to $k-1$ (\Cref{ex:sharpindex}). Third, the leading term $(d-1)k\log k$ in the exponent is attained, as we explain in the following subsection.

\subsection{Two sources of small representations}\label{sub:method}

The lower bound \eqref{eq:lower} comes from inducing the linear characters of a single subgroup of small index. The factor $\exp((d-1)k\log k)$ in \Cref{thm:main} measures how far $\Rep_k$ can go beyond this, and it comes from two independent sources, namely the fact that there might be many subgroups of small index and that there might be representations that abelianization growth does not see. The proof of \Cref{thm:main} combines these two. We first look at each source on its own.

\subsubsection*{Monomial representations}
A representation is \emph{monomial} if it is induced from a linear character of a subgroup. A monomial representation of degree at most $k$ is determined by a subgroup $H$ of index at most $k$ and one of the at most $\ab_k(G)$ linear characters of $H$. A $d$-generated group has at most $j(j!)^{d-1}$ subgroups of index $j$ \cite[Cor.~1.1.2]{LS03}, so the number of irreducible monomial representations of degree at most $k$ is at most
\[
    k^2(k!)^{d-1}\cdot\ab_k(G) = \exp\left((d-1)k\log k+O(dk)\right)\cdot\ab_k(G)
\]
(\Cref{thm:monomial}). Here the whole factor comes from counting subgroups, and the following example shows that it is attained up to the constant in the exponent.

\begin{example}[see {\Cref{ex:monomialsharp}} for details]\label{ex:monomial-intro}
Let $W=\Alt(5)\wr\Alt(m)$, with $\Alt(m)$ acting naturally on $m$ points, and put $k=5m$. The irreducible representation of $\Alt(5)$ of degree $5$ is induced from a nontrivial linear character of $\Alt(4)$, and it gives a monomial irreducible representation of $W$ of degree $k$. List the generating $d$-tuples of $W$ up to automorphisms, and let the \emph{free subdirect power} $\Pi_d(W)$ be the subgroup of the corresponding direct power of $W$ generated by the $d$ elements whose coordinates are the entries of these tuples. It is $d$-generated and maps onto $W$ in each coordinate, so each of its $\exp((1/5-o(1))(d-1)k\log k)$ coordinates pulls back its own monomial representation of degree $k$. On the other hand, $\Pi_d(W)$ has no abelian composition factors for $m \geq 5$, so $\ab_k(\Pi_d(W))\leq k^{O(1)}$ by \cite[Th.~3]{Sab23}.
\end{example}

\subsubsection*{Primitive representations}
On the other hand, abelianization growth does not see all representations. It misses the \emph{primitive} representations, those not induced from a representation of a proper subgroup. A primitive $\rho\in\Irr(G)$ of degree $m$ can be recovered from three pieces of data. 
The first is the image $Q$ of $\rho(G)$ in $\PGL_m(\C)$, up to conjugacy, which we call a \emph{quasiprimitive} group. The second is the surjection of $G$ onto $Q$, and the third is the lift of the resulting projective representation to an ordinary one. 
We count these in reverse order.

Two lifts differ by a linear character of $G$, so there are at most $|G:G'|\leq\ab_k(G)$ of them. 
There are at most $|Q|^{d-1}$ choices for the surjection that give different $\rho$ (\Cref{lem:orbits}), and $\log|Q|\leq m\log m+O(m)$ by Jordan's theorem in the sharp form of Collins \cite{Col07,Col08}. For $m\leq k$, the surjections thus cost at most $\exp((d-1)k\log k+O(dk))$. 
The following example shows that this factor is attained.

\begin{example}[see {\Cref{ex:altpower}} for details]\label{ex:altpower-intro}
For $k\geq6$, the standard representation of $\Alt(k+1)$ of degree $k$ is primitive and $\ab_k(\Alt(k+1))=1$, since every proper subgroup of $\Alt(k+1)$ has index at least $k+1$. Take the free subdirect power of $\Alt(k+1)$ as in \Cref{ex:monomial-intro}. By a theorem of Hall \cite{Hal36}, it is the full direct power, with one factor for each generating $d$-tuple up to automorphisms, so it still has $\ab_k=1$. Each factor pulls back its own standard representation, and by a theorem of Dixon \cite{Dix69} there are $\exp((d-1)k\log k+O(dk))$ factors. This attains the leading term of \Cref{thm:main}.
\end{example}

It remains to count the groups $Q$, which do not depend on $G$. Write $\QP(m)$ for the number of conjugacy classes of quasiprimitive subgroups of $\PGL_{m'}(\C)$ with $m'\leq m$. For \Cref{thm:main}, the bound $\log\QP(m)=O(m)$ would be enough. A soft count does not suffice for this, since there are $\exp(\Theta(m^3\log^3m))$ groups of order at most $(m+1)!$ \cite[Th.~4.5 and Cor.~16.21]{BNV07}, so the count has to use the structure of primitive linear groups. We prove a much stronger bound than we need, which may be of independent interest.

\begin{theorem}\label{thm:count}
Let $m\geq2$. The quasiprimitive subgroups of $\PGL_{m'}(\C)$ with $m'\leq m$ fall into at most $\exp\left(O(\log^4m)\right)$ conjugacy classes.
\end{theorem}

Equivalently, up to conjugacy and scalars, there are at most $\exp(O(\log^4m))$ primitive finite subgroups of $\GL_m(\C)$. This is the complex analogue of the counts of primitive permutation groups by Pyber and Shalev \cite{PS97} and of primitive linear groups over $\F_p$ by Jaikin-Zapirain and Pyber \cite{JZP11}, and the proof adapts theirs. We are not aware of a previous bound on $\QP(m)$ in the literature, and we do not expect our bound to be sharp.

\subsubsection*{Combining the two}
Every $\rho\in\Irr(G)$ of degree at most $k$ is induced from a primitive representation $\sigma$ of a subgroup $H\leq G$. Writing $j=|G:H|$ and $m=\sigma(1)$, we have $jm\leq k$. We recover $\rho$ from $H$ and $\sigma$, and count the representations $\sigma$ as in the primitive case, applied to $H$. As in the monomial case, there are at most $j(j!)^{d-1}$ choices for $H$. The lifts now differ by linear characters of $H$, so there are at most $|H:H'|\leq\ab_k(G)$ of them, and this is where abelianization growth enters. The surjections onto $Q$ now cost $|Q|^{d(H)-1}$, where $d(H)$ is the minimal number of generators of $H$, and $d(H)-1\leq(d-1)j$ by Schreier's index formula. Taken separately at their worst values, these bounds give $\exp(O(dk^2\log k))$, which is no better than the bound of de la Harpe, Robertson and Valette. But the constraint $jm\leq k$ makes the costs complementary, and the exponents add up to
\[
    (d-1)\left(j\log j+jm\log m\right)+O(dk) \leq (d-1)k\log k+O(dk) .
\]
The worst cases are the two ends, which are precisely the two cases above. At $j=1$ the representation $\rho$ is primitive, and the whole cost is Jordan's factorial, which \Cref{ex:altpower-intro} attains. At $j=k$ the representation $\rho$ is monomial, and the whole cost is the factorial in the subgroup count, which \Cref{ex:monomial-intro} attains up to a constant.

\begin{remark}
Our count is modelled on Sabatini's proof of \eqref{eq:upper}. He applies Jordan's theorem to $\rho(G)$ itself, and then has to identify $\rho(G)$ among all groups of order up to about $(k+1)!\cdot\ab_{(k+1)!}(G)$. Applying it to the primitive $\sigma$ instead, we only have to identify $Q$, whose order does not involve $G$ (\Cref{rem:sabatini} for more details).
\end{remark}

\subsection{Alternating sections and the logarithm}\label{sub:altfree-intro}

The examples in the previous subsection attain the logarithm in \Cref{thm:main}, and both are built from large alternating groups. This is not an accident, since both factorials in the proof come from symmetric groups. The subgroup count bounds the action of $G$ on the cosets of $H$, which has its image in $\Sym(j)$, and Jordan's bound is attained by the symmetric group of degree $m+1$. We look at the two extreme cases, groups without large alternating sections and groups with many of them.

\subsubsection*{Groups without large alternating sections}
For these groups, the logarithm disappears. We say that $G$ is \emph{$\Alt(c)$-free} if no section of $G$ is isomorphic to the alternating group $\Alt(c)$.

\begin{theorem}\label{thm:altfree}
Let $c\geq4$ and $d\geq2$, and let $G$ be a finite $d$-generated $\Alt(c+1)$-free group. Then
\[
    \Rep_k(G) \leq \exp\left(O_{c,d}(k)\right)\cdot \ab_k(G)
    \quad \text{for all } k\geq1 .
\]
\end{theorem}

The proof is the same count, with the two factorials replaced by exponentials. For the subgroups, this is the subgroup growth bound of Lubotzky and Segal \cite[Th.~3.3, Cor.~3.4]{LS03} for groups without large alternating sections. For the quasiprimitive groups, we show that an $\Alt(c+1)$-free quasiprimitive $Q\leq\PGL_m(\C)$ has order $m^{O_c(1)}$ (\Cref{prop:qaltfree}). This is a linear analogue of a theorem of Babai, Cameron and P\'alfy \cite{BCP82} on permutation groups.

The expanding groups in \Cref{ex:lz} are $\Alt(5)$-free. Together with \Cref{thm:altfree}, they answer Wigderson's question in the absence of large alternating sections.

\begin{corollary}\label{cor:sharp}
Let $c\geq4$, $d\geq2$ and $\epsilon>0$. Then every $\Alt(c+1)$-free $G\in\mathcal{C}_{d,\epsilon}$ satisfies $\Rep_k(G)= \exp(O_{c,d,\epsilon}(k))$ for all $k\geq1$. Conversely, if $d\geq d_0$ and $\epsilon$ is small enough, then for all large $k$ some $\Alt(5)$-free $G\in\mathcal{C}_{d,\epsilon}$ satisfies $\Rep_k(G)\geq \exp(\Omega(k))$.
\end{corollary}

So along $\Alt(c+1)$-free expanding sequences, both $\ab_k$ and $\Rep_k$ grow at most exponentially in $k$. For $\Alt(c+1)$-free sequences, these two conditions are thus necessary for $d$-expansion, but they are not sufficient, as the following example shows.

\begin{example}[see {\Cref{prop:notsufficient}} for details]\label{ex:segal-intro}
Fix a prime $p\geq7$ with $p\equiv\pm2\pmod5$, and let $W_n$ be the $n$-fold iterated wreath product of $\PSL_2(\F_p)$ in its action on the projective line. Since $5$ does not divide $|\PSL_2(\F_p)|$, these groups are $\Alt(5)$-free. For every $d\geq4$, their free subdirect powers $\Pi_d(W_n)$, as in \Cref{ex:monomial-intro}, satisfy $\ab_k\leq k^{O(1)}$ and $\Rep_k\leq\exp(O_d(k))$, but are not $d$-expanding. The reason is that $W_n$ is a quotient of a $4$-generated group of automorphisms of a rooted tree constructed by Segal \cite{Seg01} that is amenable \cite{BKN10}. The images of its generators are generating tuples of $W_n$ whose spectral gaps tend to $0$, and hence the spectral gap of \emph{every} generating $d$-tuple of $\Pi_d(W_n)$ tends to $0$ as well.
\end{example}

Sabatini asked whether $d$-generated groups without abelian composition factors, or more generally $d$-generated groups satisfying the bounds of Lubotzky and Weiss and of de la Harpe, Robertson and Valette, always form an expanding sequence \cite[Questions~1.1 and 4.4]{Sab25}. For $d$-expansion, the powers of alternating groups in \Cref{ex:altpower-intro} already give a negative answer (\Cref{ex:altpower-notexp}). The groups in the preceding example show more, namely that not even the bound $\Rep_k\leq\exp(O(k))$, which expansion forces in the $\Alt(5)$-free class, gives $d$-expansion. The original questions allow any bounded number of generators, and they remain open. For $p$ large enough, the groups $\Pi_d(W_n)$ also show that \Cref{thm:altfree} is sharp up to the constant in the exponent, for infinitely many $k$ (\Cref{prop:altfreegap}).

\subsubsection*{Groups with many alternating sections}
At the opposite end are groups with plenty of alternating sections. Here the logarithm in \Cref{thm:main} is attained, by the powers of alternating groups in \Cref{ex:altpower-intro}, and the question is whether it survives along expanding sequences.
The natural candidates are powers $\Alt(n)^{f(n)}$. At $k=n-1$, such a power has $\ab_k=1$ and $\Rep_k=1+f(n)$, so everything depends on how large the number $f(n)$ of factors can be along an expanding sequence.

If the powers are $d$-generated, then $f(n)$ is at most the largest $t$ for which $\Alt(n)^t$ is $d$-generated. We call this number \emph{Hall's ceiling}. By the theorem of Hall used in \Cref{ex:altpower-intro}, it is $\exp((1+o(1))(d-1)n\log n)$, and the groups in that example, with $n=k+1$, sit exactly at the ceiling. The powers at the ceiling are not $d$-expanding, since every generating $d$-tuple of $\Alt(n)$ occurs in some coordinate, including tuples with small spectral gap (\Cref{ex:altpower-notexp}).

The logarithm would therefore be present if $f(n)\geq|\Alt(n)|^\delta$ for a fixed $\delta>0$ along an expanding sequence, and we do not know whether this can happen.

\begin{question}\label{q:alt}
Is there $\delta>0$ such that, for some $f(n)\geq|\Alt(n)|^{\delta}$, the powers $\Alt(n)^{f(n)}$ form an expanding sequence?
\end{question}

Bartholdi and Kassabov \cite{BK25} give a \emph{lower} bound on how large $f(n)$ can be. For every $e\geq2$, they construct $(e+4)$-expanding powers of $\Alt(n)$ with $\log f(n)$ of order $(\log n)^{e+1}$, as quotients of a group with Kazhdan's property (T). This is far below Hall's ceiling, where $\log f(n) \sim (d-1)n\log n$. Our contribution here is an \emph{upper} bound, showing that expanding powers cannot come too close to the ceiling.

\begin{theorem}\label{thm:ceiling}
Let $d\geq2$ and $\epsilon>0$, and let $\Sigma_n$ be a generating $d$-tuple of $\Alt(n)^{f(n)}$ such that the Cayley graphs $\Cay(\Alt(n)^{f(n)},\Sigma_n)$ have spectral gap at least $\epsilon$. Then
\[
    \log f(n) \leq \left(d-1-\frac{\epsilon}{160}+o(1)\right)n\log n .
\]
\end{theorem}

The proof rests on an almost intertwiner, an idea of de la Harpe, Robertson and Valette \cite{HRV93}. Suppose that two coordinates of a generating $d$-tuple of the power agree after a conjugation, up to a permutation of small support in one entry. Pulling back the standard representation of $\Sym(n)$ along these two coordinates gives two representations with no common constituent, and yet the identity matrix is almost an intertwiner between them. This forces a small spectral gap. A counting argument then shows that two such coordinates must exist as soon as the number of factors is close to Hall's ceiling.

A natural place to look for a positive answer to \Cref{q:alt} is random generators, which give expanders in many other settings, for instance for $\PSL_2(\F_p)$ \cite{BG08}. This turns out to fail long before Hall's ceiling. For $d\geq3$, a uniformly random generating $d$-tuple of $\Alt(n)^{f(n)}$ has spectral gap tending to $0$ in probability as soon as $f(n)$ is superpolynomial in $n$ (\Cref{prop:random}). The reason is that $\Alt(n)$ has generating tuples with small spectral gap that are rare, but not too rare, and a power with many factors contains one of them in some coordinate. Even the powers of Bartholdi and Kassabov expand only with respect to their special generators and not with respect to random ones (\Cref{ex:bkrandom}).

For groups of Lie type in place of $\Alt(n)$, on the other hand, the analogue of \Cref{q:alt} has a positive answer, again through property (T). The powers $\PSL_3(\F_p)^{p^m}$ are quotients of a group with property (T) by Ershov and Jaikin-Zapirain \cite{EJ10}, and their number of factors is a fixed positive power of $|\PSL_3(\F_p)|$ (\Cref{ex:sl3}). The same holds in unbounded rank, for the groups $\SL_{3n}(\F_2)$, by an example of Kassabov (\Cref{ex:kassabov}).

\Cref{fig:landscape} summarizes how the conditions in this paper are ordered by \eqref{eq:lower}, \Cref{thm:main} and the bound of Lubotzky and Weiss, and where the examples are.

\begin{figure}[t]
\centering
\begin{tikzpicture}[font=\scriptsize]
    \def\E{(0.45,0.65) rectangle (5.6,2.75)}
    \def\F{(0.3,1.25) rectangle (8.45,3.35)}
    \useasboundingbox (0,0.35) rectangle (8.75,4.55);
    \draw[rounded corners=6pt] (0,0.35) rectangle (8.75,4.55);
    \draw[rounded corners=5pt] (0.15,0.5) rectangle (8.6,3.95);
    \fill[gray!7,rounded corners=4pt] \E;
    \draw[rounded corners=4pt] \F;
    \begin{scope}
        \clip[rounded corners=4pt] \F;
        \draw[rounded corners=4pt] \E;
    \end{scope}
    \begin{scope}[even odd rule]
        \clip (-1,-1) rectangle (11,5) [rounded corners=4pt] \F;
        \draw[dashed,rounded corners=4pt] \E;
    \end{scope}
    % regions
    \node[anchor=west] at (0.1,4.25) {$\Rep_k\leq\exp(O(k\log k))$};
    \node[anchor=west] at (0.25,3.65) {$\ab_k\leq\exp(O(k))$};
    \node[anchor=west] at (0.4,3.05) {$\Rep_k\leq\exp(O(k))$};
    \node[anchor=west] at (0.55,2.45) {$d$-expanding};
    % d-expanding, with small representations
    \foreach \x/\y/\t in {0.7/2.05/\Cref{ex:lower}, 3.2/2.05/\Cref{ex:sl3}, 0.7/1.65/\Cref{ex:kassabov}, 3.2/1.65/\cite{BK25}}{
        \fill (\x,\y) circle (1.5pt);
        \node[right] at (\x+0.05,\y) {\t};
    }
    % open
    \draw[fill=white] (0.7,0.95) circle (1.5pt);
    \node[right] at (0.75,0.95) {\Cref{q:alt}};
    % not d-expanding
    \fill (6.05,3.05) circle (1.5pt);
    \node[right] at (6.1,3.05) {\Cref{prop:notsufficient}};
    \fill (6.05,3.65) circle (1.5pt);
    \node[right] at (6.1,3.65) {\Cref{ex:altpower}};
    \fill (6.05,4.25) circle (1.5pt);
    \node[right] at (6.1,4.25) {\Cref{ex:window}};
\end{tikzpicture}
\caption{Sequences of $d$-generated groups, ordered by the conditions in this paper. The dashed boundary marks the open question whether some expanding sequence has representation growth faster than exponential in $k$.}
\label{fig:landscape}
\end{figure}
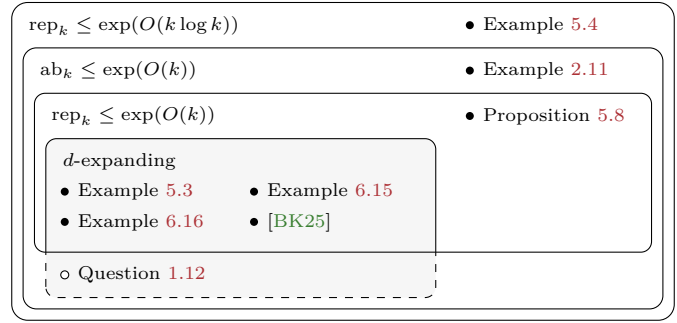

\subsection{Reader's guide}

We first prove \Cref{thm:main}, in \Cref{sec:count,sec:qp}. In \Cref{sec:count} we run the Clifford count, starting with monomial representations, and reduce the theorem to counting quasiprimitive groups. We also give the examples showing that the index $k$ and the leading term $(d-1)k\log k$ are optimal, among them the powers of alternating groups of \Cref{ex:altpower-intro}. The count of quasiprimitive groups, \Cref{thm:count}, takes up \Cref{sec:qp}, and it completes the proof. In \Cref{sec:altfree} we adapt the argument to groups without large alternating sections and prove \Cref{thm:altfree}. Here we also introduce free subdirect powers. Applied to iterated wreath products they show that \Cref{thm:altfree} is sharp, and applied to $\Alt(5)\wr\Alt(m)$ they give the groups of \Cref{ex:monomial-intro}. We then turn to expansion. In \Cref{sec:reduction} we recall the proof of the theorem of Lubotzky and Weiss and deduce \Cref{cor:quasilinear} and \Cref{cor:sharp}, with the groups of \Cref{ex:lz} for the lower bounds. We also show there that slow representation growth does not imply $d$-expansion, with the groups of \Cref{ex:segal-intro}. Finally, in \Cref{sec:powers} we study expanding powers of alternating groups and \Cref{q:alt}. We first show that random generators do not give expanding powers, then prove \Cref{thm:ceiling}, and end with the examples of Lie type.

\subsection{AI disclosure}\label{sub:disclosure}

This paper is the outcome of a long conversation with Anthropic's language models Claude Opus 5 and Claude Opus 5.5. We learned of Wigderson's question from Sabatini \cite{Sab23,Sab25}. When we asked the model about the upper bound \eqref{eq:upper} relating $\Rep_k$ and $\ab_k$, it suggested how one might do better. In the conversation that followed, it reduced the problem to quasiprimitive subgroups of $\PGL_m(\C)$ and to counting them. Its first argument was far more complicated than necessary, as were almost all of its arguments in this project. The versions presented here are simpler, partly through our own simplifications and partly through the model's, once asked to look for them. We also asked the model about examples that would clarify the relationship between $\ab_k$ and $\Rep_k$, and most of them are now in the paper. The text was also written in this back and forth, with the model drafting passages that we revised, and revising passages that we drafted.

The proofs and mathematical ideas in this paper were essentially all discovered by the model. The human contribution was to ask questions, digest the proofs, guide the model by deciding which paths to pursue, simplify the arguments or tell the model to do so, encourage the model when it believed an argument would not work (which often led to one after all), set up the narrative of the paper, and sharpen the exposition by putting explanations in a human-friendly format and order. We take full responsibility for the correctness of the paper.

\subsection{Acknowledgements}

We thank Luca Sabatini for the questions that started this work, and for helpful comments on an earlier version of this manuscript.

\section{Counting small representations}\label{sec:count}

In this section we reduce \Cref{thm:main} to a count of quasiprimitive groups. We first treat monomial representations, where the whole argument takes a few lines, then recall the Clifford theory we need, run the general count, and close with examples showing that the result is sharp.

\subsection{The monomial count}

We write $d(G)$ for the minimal number of generators of $G$, and $s_k(G)$ for the number of subgroups of $G$ of index at most $k$. For a subgroup $H\leq G$, we write $\widehat{H}$ for the group of linear characters of $H$, which has order $|H:H'|$.

\begin{proposition}\label{thm:monomial}
Let $G$ be a finite $d$-generated group and let $\Rep_k^{\mon}(G)$ denote the number of irreducible monomial representations of $G$ of degree at most $k$. Then
\[
    \Rep_k^{\mon}(G) \leq k^2 (k!)^{d-1}\cdot \ab_k(G) = \exp\left((d-1)k\log k+O(dk)\right)\cdot \ab_k(G) .
\]
\end{proposition}

\begin{proof}
Let $\rho\in\Irr(G)$ be monomial of degree at most $k$, so that $\rho=\Ind_H^G\lambda$ for some $H\leq G$ and some $\lambda\in\widehat{H}$. Then $|G:H|=\rho(1)\leq k$. The pair $(H,\lambda)$ determines $\rho$, so it suffices to count such pairs.
Since $G$ is $d$-generated, the number of subgroups of index exactly $j$ is at most $j(j!)^{d-1}$ by \cite[Cor.~1.1.2]{LS03}, and hence the number of subgroups of index at most $k$ is at most $k^2(k!)^{d-1}$.
Given such a subgroup $H$, the number of its linear characters is $|H:H'|\leq\ab_k(G)$, since $|G:H|\leq k$. The bound follows.
\end{proof}

The factor $k^2(k!)^{d-1}$ is of the right order, up to the constant in the exponent. In \Cref{ex:monomialsharp} we construct groups with $\Rep_k^{\mon}/\ab_k\geq\exp((1/5-o(1))(d-1)k\log k)$.

To pass from the monomial case to the general one, we must replace the linear characters of $H$ by \emph{primitive} representations of $H$, and count these. We first set up the necessary background, and then run the count in general, reducing \Cref{thm:main} to counting quasiprimitive subgroups of $\PGL_m(\C)$.

\subsection{Primitive groups and Clifford theory}\label{sub:primitive}

\subsubsection{Primitive groups}

A representation that is not induced from a proper subgroup has a very restricted shape, and this shape is a property of its image.
A finite subgroup $\Gamma\leq\GL(V)$ is \emph{imprimitive} if $V$ decomposes as $V=V_1\oplus\cdots\oplus V_r$ with $r\geq2$ in such a way that $\Gamma$ permutes the $V_i$. In other words, the matrix of each group element is block-monomial with respect to this decomposition.
If no such decomposition exists, $\Gamma$ is \emph{primitive} \cite[Sec.~15]{Sup76}.
A representation $\sigma$ of a finite group $H$ on $V$ is \emph{primitive} if the subgroup $\sigma(H)\leq\GL(V)$ is primitive. For irreducible $\sigma$, this says exactly that $\sigma$ is not induced from a representation of a proper subgroup of $H$.
A subgroup $Q\leq\PGL_m(\C)$ is \emph{quasiprimitive} if it is the image in $\PGL_m(\C)$ of a finite primitive subgroup $\Gamma\leq\GL_m(\C)$.
In this case $Q\cong\Gamma/Z(\Gamma)$, since by Schur's lemma $Z(\Gamma)$ is exactly the intersection of $\Gamma$ with the scalar matrices.

\subsubsection{Clifford theory}

We recall some basic facts from Clifford theory on the representations of a finite group and those of its subgroups, see \cite[Ch.~6]{Isa76} or \cite[Sec.~16]{Sup76}.

First, every $\rho\in\Irr(G)$ is induced from a primitive representation of a subgroup of $G$. Indeed, if $H\leq G$ is minimal such that $\rho=\Ind_H^G\sigma$ for some $\sigma\in\Irr(H)$, then $\sigma$ is primitive, and $\rho(1)=|G:H|\sigma(1)$.
Second, the restriction of a primitive representation to a normal subgroup $N$ is homogeneous. This means that the underlying space $V$ is a direct sum of pairwise isomorphic irreducible representations of $N$ \cite[Cor.~16.2.1]{Sup76}.
Third, if $\sigma$ is primitive, then every abelian normal subgroup of $\sigma(H)$ is cyclic and central \cite[Lem.~1]{Col08}, and hence consists of scalars by Schur's lemma.
This last fact is the key to the general count, since it ties $\rho$ tightly to the quasiprimitive quotient of $H$ in $\PGL_m(\C)$.

\begin{remark}\label{rem:names}
The terminology around quasiprimitivity is not uniform in the literature. We will use results from \cite{Col08,Wei12,LMM01}, where a finite subgroup of $\GL(V)$, for a vector space $V$ over a field $K$, is called quasiprimitive if its restriction to every normal subgroup is homogeneous. By the second fact above, every primitive $\Gamma\leq\GL_m(\C)$ has this property, so their theorems apply to the groups we meet here.
\end{remark}

\subsubsection{Lifting projective representations}

In the monomial case, the argument reduced to counting linear characters of subgroups of $G$. In the general case, we have to count primitive representations of subgroups of $G$, and it is easier to pass to their projective images along the projection $\pi\colon\GL_m(\C)\to\PGL_m(\C)$. We record two basic properties of this passage. The first is what ultimately brings the abelianization growth into the picture.

\begin{lemma}\label{lem:lifts}
Let $H$ be a finite group. If $\sigma_1,\sigma_2\colon H\to\GL_m(\C)$ are representations with $\pi\circ\sigma_1=\pi\circ\sigma_2$, then $\sigma_2=\nu\cdot\sigma_1$ for some $\nu\in\widehat{H}$. 
Consequently, any projective representation of $H$ has at most $|H:H'|$ lifts to a representation of $H$.
\end{lemma}

\begin{proof}
For each $h\in H$, the matrices $\sigma_1(h)$ and $\sigma_2(h)$ have the same image in $\PGL_m(\C)$, so $\nu(h)=\sigma_2(h)\sigma_1(h)^{-1}$ is a scalar. Scalars are central, and hence
\[
    \nu(hh') = \sigma_2(hh')\sigma_1(hh')^{-1} = \sigma_2(h)\sigma_2(h')\sigma_1(h')^{-1}\sigma_1(h)^{-1} = \nu(h)\nu(h')
\]
for all $h,h'\in H$.
So $\nu$ is a linear character of $H$.
\end{proof}

\begin{lemma}\label{lem:orbits}
Let $H$ be a finite group and let $Q\leq\PGL_m(\C)$. Let $\mathcal{P}$ be the set of pairs $(\varphi,\sigma)$ in which $\varphi\colon H\twoheadrightarrow Q$ is a surjection and $\sigma\colon H\to\GL_m(\C)$ is an irreducible representation with $\pi\circ\sigma=\varphi$. Let $c_x$ denote conjugation by $x$ in $Q$.
Then $Q$ acts freely on $\mathcal{P}$ by
\[
    x\cdot(\varphi,\sigma) = \left(c_x\circ\varphi,\ X\sigma X^{-1}\right),
    \quad \text{with } X\in\GL_m(\C) \text{ any lift of } x,
\]
and two pairs in the same orbit have isomorphic second coordinates.
\end{lemma}

\begin{proof}
First, the formula is well defined, since conjugation by $X$ does not depend on the choice of the lift. The new pair lies in $\mathcal{P}$, because a conjugate of an irreducible representation is irreducible and $\pi\circ(X\sigma X^{-1})=c_x\circ\pi\circ\sigma=c_x\circ\varphi$. It is an action because lifting is multiplicative up to scalars, which conjugation ignores. Now suppose that $x$ fixes a pair $(\varphi,\sigma)$. Then $X\sigma(h)X^{-1}=\sigma(h)$ for every $h\in H$, so $X$ centralizes $\sigma(H)$. Since $\sigma$ is irreducible, $\sigma(H)$ spans $\mathrm{M}_m(\C)$ by Burnside's theorem, so $X$ is a scalar and $x=1$. Thus the action is free. Finally, the representations $X\sigma X^{-1}$ and $\sigma$ are clearly isomorphic.
\end{proof}

\subsection{The general count}\label{sub:general-count}

We now turn to the general case, following the same plan as in the monomial one.
Let $\rho$ be an irreducible representation of $G$, and view it as induced from a primitive representation $\sigma$ of degree $m$ of a subgroup $H\leq G$.
In the monomial case, the pair $(H,\lambda)$ determined $\rho$, and counting it was a matter of counting subgroups and linear characters. Now $\sigma$ need not be linear, so we record it in two steps, first its image in $\PGL_m(\C)$ and then the lift back to $\GL_m(\C)$. We thus count the representations of $G$ of degree at most $k$ through the following four ingredients:
\begin{enumerate}
    \item the subgroup $H$,
    \item the projective image of $\sigma(H)$ in $\PGL_m(\C)$ up to conjugacy,
    \item the surjection of $H$ onto that image, and
    \item the lift of the projective representation to an ordinary one.
\end{enumerate}
We go through them in turn.

\subsubsection{The four ingredients}

The subgroup $H$ is handled exactly as in the monomial case. Its index $j=|G:H|$ satisfies $jm=\rho(1)\leq k$, so $j\leq k$, and there are at most $s_k(G)\leq k^2(k!)^{d-1}$ choices for $H$.

The projective image does not depend on $G$, so we may count all groups that could possibly occur. Set
\[
    \QP(m) = \#\left\{ \text{quasiprimitive } Q\leq\PGL_{m'}(\C) \ \mid \ m' \leq m \right\}\big/\text{conjugacy} .
\]
This is monotone in $m$ by construction, and it is finite.\footnote{Such a $Q$ has order at most $J(m)$ by \factref{J}, its preimage in $\SL_m(\C)$ is a finite group of order at most $mJ(m)$ acting irreducibly, there are finitely many groups of a given order, and each of them has finitely many irreducible representations of degree $m$ up to isomorphism.}
Every projective image we meet lies in degree $m=\sigma(1)\leq k$, so this ingredient costs at most $\QP(k)$.

The surjection is the expensive ingredient. A homomorphism from $H$ onto a group $Q$ is determined by the images of $d(H)$ generators, so there are at most $|Q|^{d(H)}$ of them. We therefore need to bound both the order of $Q$ and the number of generators of $H$, and the following two standard facts do this.

\begin{itemize}
\fact{J} \emph{Jordan's theorem.} A finite subgroup $\Gamma\leq\GL_m(\C)$ has an abelian normal subgroup of index at most $J(m)$, where $J(m)=(m+1)!$ for $m\geq71$ by Collins \cite[Th.~A]{Col07}. In all cases $\log J(m)=O(m\log m)$, and we take $J$ to be monotone in $m$. If $\Gamma$ is primitive, its abelian normal subgroups consist of scalars by the Clifford facts above, so $|\Gamma:Z(\Gamma)|\leq J(m)$.
A quasiprimitive subgroup $Q\leq\PGL_m(\C)$ therefore satisfies $|Q|\leq J(m)$. In fact $|Q|\leq(m+1)!$ for $m>12$ by \cite[Th.~A]{Col08}, with equality for the symmetric group of degree $m+1$.

\fact{S} \emph{Schreier's index formula.} If $|G:H|=j$ and $G$ is $d$-generated, then $d(H)\leq 1+j(d-1)$.
\end{itemize}

Neither bound is of any use on its own. At their worst values, $\log J(k)=O(k\log k)$ and $d(H)-1\leq k(d-1)$ together give $\exp(O(dk^2\log k))$. The point is that in the proof below they are never both large. This interplay between $j$ and $m$ is what brings the exponent down from $k^2\log k$ to $k\log k$.

The lift is handled by \Cref{lem:lifts}, since a given surjection has at most $|H:H'|\leq\ab_k(G)$ lifts.

\subsubsection{Combining the ingredients}

\begin{theorem}\label{prop:reduction}
Every finite $d$-generated group $G$ satisfies
\[
    \Rep_k(G) \leq \exp\left((d-1)k\log k+O(dk)\right)\cdot \QP(k)\cdot \ab_k(G)
    \quad \text{for all } k\geq1 .
\]
\end{theorem}

\begin{proof}
If $d=1$, then $G$ is cyclic, so $\Rep_k(G)=|G:G'|=\ab_k(G)$ and the bound holds trivially. Assume from now on that $d\geq2$.

Fix once and for all a representative of each conjugacy class of quasiprimitive subgroups of $\PGL_{m}(\C)$ for $m\leq k$.
To each $\rho\in\Irr(G)$ of degree at most $k$ we attach the four ingredients above, a quadruple $(H,Q,\varphi,\sigma)$ with $\sigma$ a representation of $H$ lying over $\varphi$, and we count the possibilities for each in turn.

By the Clifford facts above, there are $H\leq G$ and a primitive $\sigma\in\Irr(H)$ with $\rho=\Ind_H^G\sigma$. Set $j=|G:H|$ and $m=\sigma(1)$, so that $jm=\rho(1)\leq k$. We hold $j$ and $m$ fixed while counting, and sum over the pairs $(j,m)$ at the end. Since $\sigma$ is irreducible, the scalars in $\sigma(H)$ form exactly its centre $Z(\sigma(H))$ by Schur's lemma.
Thus the image $\pi(\sigma(H))=\sigma(H)/Z(\sigma(H))$ is a quasiprimitive subgroup of $\PGL_m(\C)$.
Replacing $\sigma$ by a conjugate representation, which changes neither its equivalence class nor that of $\rho$, we may assume that this image is the chosen representative $Q$ of its conjugacy class. Let $\varphi=\pi\circ\sigma\colon H\twoheadrightarrow Q$. This attaches to $\rho$ a pair $(H,Q)$ and a pair $(\varphi,\sigma)$ in the set $\mathcal{P}_{H,Q}$ of \Cref{lem:orbits}.

The representation $\sigma$ determines $\rho=\Ind_H^G\sigma$. By \Cref{lem:orbits}, all pairs in the $Q$-orbit of $(\varphi,\sigma)$ have isomorphic second coordinates, so they all determine the same $\rho$. Hence $\rho$ can be recovered from $(H,Q)$ together with this orbit. Since the action is free, every orbit has $|Q|$ elements, and so
\[
    \Rep_k(G) \leq \sum_{(H,Q)}\frac{|\mathcal{P}_{H,Q}|}{|Q|} ,
\]
where the sum runs over the pairs that arise. Here $|\mathcal{P}_{H,Q}|$ is at most the number of surjections $H\twoheadrightarrow Q$ times the largest number of lifts of one of them.

Now fix a pair $(j,m)$ with $jm\leq k$, and count the pairs $(H,Q)$ with $|G:H|=j$ and $Q\leq\PGL_m(\C)$.
Write $a_j$ for the number of subgroups of $G$ of index exactly $j$.
There are at most $a_j\leq j(j!)^{d-1}$ choices for $H$ by \cite[Cor.~1.1.2]{LS03}, at most $\QP(m)$ choices for $Q$, and at most $|H:H'|\leq\ab_j(G)$ lifts of a given $\varphi$ by \Cref{lem:lifts}. As for the surjection, the count above gives at most $|Q|^{d(H)}$ homomorphisms, and the division by $|Q|$ leaves $|Q|^{d(H)-1}$.
Here $|Q|\leq J(m)$ by \factref{J}, since $\sigma(H)$ is primitive, and so \factref{S} gives $|Q|^{d(H)-1}\leq J(m)^{j(d-1)}$. Multiplying the four counts and summing over the pairs, we obtain
\[
    \Rep_k(G)  \leq  \sum_{jm\leq k} a_j\cdot \QP(m)\cdot J(m)^{j(d-1)}\cdot \ab_j(G) ,
\]
where the sum runs over the at most $k^2$ pairs of positive integers with $jm\leq k$.
By \factref{J} and Stirling's formula we have $\log J(m)\leq m\log m+O(m)$, while $\log a_j\leq\log j+(d-1)j\log j$. Using $\log j\leq k$ and $j(d-1)\cdot O(m)=O((d-1)k)$, we see that the subgroup and the surjection together contribute at most
\[
    (d-1)\left(j\log j+jm\log m\right)+O(dk)
\]
to the logarithm.
Since $jm\leq k$, the expression in brackets is bounded by
\[
    j\log j+jm\log m \leq j\log j+k\log(k/j) \leq k\log k .
\]
Finally, replace $\QP(m)$ by $\QP(k)$ and $\ab_j(G)$ by $\ab_k(G)$, and absorb the $k^2$ pairs into the error term. The logarithm of the prefactor is then $(d-1)k\log k+O(dk)$, with an absolute implied constant.
\end{proof}

The bound $k\log k$ on the expression in brackets is attained only at the two ends $j=1$ and $j=k$, which are the primitive and the monomial case. In between it is smaller, and near $j=k/\log k$ it drops to $O(k\log\log k)$.

\begin{remark}\label{rem:sabatini}
The count above is modelled on Sabatini's proof of \eqref{eq:upper} \cite[Prop.~2]{Sab23}.
Let us explain where the two arguments differ.

Sabatini starts with an irreducible representation $\rho$ of $G$ of degree at most $k$, and applies Jordan's theorem directly to its image $\rho(G)$. This gives an abelian normal subgroup of $\rho(G)$ of index at most $J(k)$. Its preimage $A_\rho\leq G$ has index at most $J(k)$ and abelian image, so
\[
    |\rho(G)| \leq J(k)\cdot |A_\rho:A_\rho' | \leq J(k)\cdot\ab_{J(k)}(G) =: N .
\]
The representation is then recorded in three pieces, as follows.
\begin{itemize}
    \item the abstract group $\rho(G)$, one of at most $N^{3d\log N}$ $d$-generated groups of order at most $N$ by \cite{Lub01},
    \item one of its at most $N$ irreducible representations, which fixes its embedding into $\GL_k(\C)$ up to conjugacy, and
    \item the homomorphism $G\to\rho(G)$, at a cost of $N^d$.
\end{itemize}
Multiplying gives $\Rep_k(G)\leq N^{4d\log N}$, and \eqref{eq:upper} follows.

So both arguments record a representation by factoring it through a fixed group, a surjection onto it and some abelian data. They differ in four respects.

\emph{Where the structure theory is applied.} Jordan's theorem applied to $\rho(G)$ produces a genuinely large abelian normal subgroup, and pulling it back forces the index $J(k)\approx(k+1)!$. We apply the structure theory not to $\rho$ but to a primitive $\sigma$ inducing it, and the abelian normal subgroups of a primitive group are scalar. There is then nothing large to pull back. The only subgroup involved is $H$, which has index $j\leq k$ because $jm=\rho(1)\leq k$. This is the improvement from $(k+1)!$ to $k$.

\emph{How the abelian data enters.} Sabatini includes the abelian part $\ab_{J(k)}(G)$ in the order $N$ of the group to be identified, so $\ab$ appears both in the base and in the exponent. This is why \eqref{eq:upper} has the shape $\ab^{O(d\log\ab)}$. In our argument this does not happen. The group to be identified is $Q\leq\PGL_m(\C)$, of order at most $J(m)$, while the abelian part comes from a linear character of $H$ and costs a factor $|H:H'|\leq\ab_k(G)$ by \Cref{lem:lifts}. The dependence on $\ab$ thus becomes linear.

\emph{What has to be identified.} Naming an abstract group of order $N$ among all of them costs $N^{O(d\log N)}$. This is the dominant term above, and since $N\geq J(k)$, it is already $\exp(\Omega(dk^2\log^2k))$. In our argument, we only have to locate a fixed representative of a conjugacy class of quasiprimitive subgroups of $\PGL_m(\C)$. By \Cref{thm:count} there are at most $\exp(O(\log^4m))$ of these, and $\exp(O(k))$ would already suffice.

\emph{The homomorphism.} In Sabatini's argument it costs $N^d$ with $N\geq J(k)$, while in ours it is restricted to $H$ and costs $|Q|^{d(H)-1}$, which the trade-off between $j$ and $m$ keeps within the prefactor.
\end{remark}

\subsection{Sharpness of the bound}\label{sub:two-sharp-examples}

\subsubsection{Wreath products}

The first example shows that $\Rep_k$ cannot be bounded by abelian data from subgroups of index smaller than $k$.

\begin{example}\label{ex:sharpindex}
Let $k$ and $p>k$ be primes, and let $G=\F_p[C_k]\rtimes C_k$ be the wreath product $\F_p\wr C_k$. The $C_k$-module $\F_p[C_k]$ is cyclic, so $G$ is $2$-generated, and $|G|=p^kk$. The index of a subgroup divides $p^kk$, and since $p>k$, the only such divisor smaller than $k$ is $1$. So every proper subgroup of $G$ has index at least $k$, and consequently
\[
    \ab_{k-1}(G) = |G:G'| = pk ,
\]
since the commutator subgroup $G'$ is the augmentation ideal of $\F_p[C_k]$.
On the other hand, $C_k$ fixes only $p$ of the characters of the abelian group $\F_p[C_k]$, namely the scalar multiples of the augmentation map $\F_p[C_k]\to\F_p$.
Since $k$ is prime, the orbits of $C_k$ on the remaining $p^k-p$ characters are all regular. So each of the $(p^k-p)/k$ orbits induces to an irreducible representation of $G$ of degree $k$, and hence $\Rep_k(G)\geq(p^k-p)/k$. Therefore
\[
    \frac{\Rep_k(G)}{\ab_{k-1}(G)} \geq \frac{p^k-p}{pk^2} \longrightarrow \infty
    \quad \text{as } p\to\infty \text{ with } k \text{ fixed} .
\]
So no bound of the form $\Rep_k(G)\leq F(k,d)\cdot\ab_{k-1}(G)$ can hold for all $d$-generated finite groups, whatever the function $F$.
Note that $\ab_k(G)\geq|\F_p[C_k]|=p^k$, so \Cref{thm:main} is consistent with this example.
\end{example}

\subsubsection{Powers of simple groups}\label{subsub:powers}

The second example shows that the leading term $(d-1)k\log k$ of the prefactor is sharp.\footnote{We will see in the next section that $\QP(k)=\exp(O(\log^4k))$, so the factor $\QP(k)$ in \Cref{prop:reduction} does not affect the sharpness of the prefactor.}
It is the example predicted by the proof of \Cref{prop:reduction} in the extreme case $j=1$. There the cost of the surjection is $|Q|^{d-1}$, where $Q$ is a quasiprimitive subgroup of $\PGL_k(\C)$, and by Collins's theorem $|Q|\leq(k+1)!$ with equality for $\Sym(k+1)$.
A group attaining the bound should thus have many small primitive representations with projective image close to $\Sym(k+1)$. The construction below takes a large direct power of $\Alt(k+1)$, whose small representations come from the projections onto the factors.

The construction works for any nonabelian simple group $S$, and we record it in this generality, since we use it again later. Write $\ell(S)$ for the smallest degree of a nontrivial irreducible representation of $S$, and $\iota(S)$ for the smallest index of a proper subgroup of $S$. Note that $\ell(S)<\iota(S)$, since the permutation representation of $S$ on the cosets of a subgroup of index $\iota(S)$ has a nontrivial constituent of degree less than $\iota(S)$. For $d\geq2$, write $t_d(S)$ for the largest $t$ for which $S^t$ is $d$-generated, and call it \emph{Hall's ceiling}. It can be computed using the following theorem of Hall.

\begin{theorem}[\cite{Hal36}]\label{thm:hall}
Let $S$ be a nonabelian finite simple group and let $d,t\geq1$. A $d$-tuple of elements of $S^t$ generates $S^t$ if and only if its $t$ coordinate projections are generating $d$-tuples of $S$ lying in pairwise distinct $\Aut(S)$-orbits.
\end{theorem}

\begin{remark} \label{rem:hall}
The above is a reformulation of two facts proved by Hall. Let $F$ be the free group on $x_1,\dots,x_d$. A $d$-tuple $Y=(y_1,\dots,y_d)$ of elements of a group defines the homomorphism $\varphi_Y$ from $F$ sending $x_l$ to $y_l$. The first fact is that two generating $d$-tuples $Y,Y'$ of $S$ lie in the same $\Aut(S)$-orbit if and only if $\ker\varphi_Y=\ker\varphi_{Y'}$ \cite[Th.~1.4]{Hal36}. So these orbits are in bijection with the normal subgroups $K\trianglelefteq F$ with $F/K\cong S$. The second fact is that for pairwise distinct such subgroups $K_1,\dots,K_t$, the natural map $F\to\prod_iF/K_i\cong S^t$ is surjective \cite[Sec.~1.6]{Hal36}. % Indeed, its image is a subdirect product of copies of $S$, hence a direct product of diagonal subgroups, and a diagonal linking two coordinates $i\neq j$ would force $K_i=K_j$.

To deduce \Cref{thm:hall}, let $\Sigma$ be a $d$-tuple of $S^t$, and let $Y^{(1)},\dots,Y^{(t)}\in S^d$ be its coordinate tuples. Composing $\varphi_\Sigma$ with the $i$th coordinate projection gives $\varphi_{Y^{(i)}}$, so $\ker\varphi_\Sigma=\bigcap_i\ker\varphi_{Y^{(i)}}$, and $\Sigma$ generates $S^t$ exactly when $\varphi_\Sigma$ is surjective. If the $Y^{(i)}$ generate $S$ and lie in pairwise distinct orbits, their kernels are pairwise distinct by the first fact, and $\varphi_\Sigma$ is surjective by the second. Conversely, if $\Sigma$ generates $S^t$, then every $Y^{(i)}$ generates $S$. If $Y^{(i)}$ and $Y^{(j)}$ with $i\neq j$ lay in the same orbit, then $\ker\varphi_{Y^{(i)}}=\ker\varphi_{Y^{(j)}}$ by the first fact, so $\ker\varphi_\Sigma$ is already the intersection of the other $t-1$ kernels. Then $S^t=F/\ker\varphi_\Sigma$ would embed in $S^{t-1}$, which is impossible.

Finally, an automorphism fixing a generating tuple is trivial, so $\Aut(S)$ acts freely on the generating $d$-tuples \cite[Sec.~1.3]{Hal36}. Writing $\phi_d(S)$ for the number of these tuples, there are $\phi_d(S)/\lvert\Aut(S)\rvert$ orbits, and \Cref{thm:hall} gives
\[
    t_d(S) = \frac{\phi_d(S)}{\lvert\Aut(S)\rvert} .
\]
\end{remark}

\begin{proposition}\label{prop:powers}
Let $S$ be a nonabelian finite simple group.
\begin{enumerate}
    \item If $k<\min\left\{\iota(S),\ \ell(S)^2\right\}$, then for every $t\geq1$ we have
    \[
        \ab_k(S^t) = 1
        \quad \text{and} \quad
        \Rep_k(S^t) = 1+t\left(\Rep_k(S)-1\right) .
    \]
    \item For every $d\geq2$ we have
    \[
        t_d(S) = (1-o(1))\frac{|S|^{d-1}}{\lvert\Out(S)\rvert} ,
    \]
    where $o(1)\to0$ as $|S|\to\infty$.
    \item If $t\leq\exp(C\ell(S))$ for some constant $C$, then $\Rep_k(S^t)\leq\exp(O_C(k))$ for all $k\geq1$.
\end{enumerate}
\end{proposition}

\begin{proof}
(1) Put $G=S^t$. Let $H<G$ be a proper subgroup of index $j$, and let $K$ be the kernel of the action of $G$ on the cosets of $H$. Since $S$ is nonabelian simple, the normal subgroups of $G$ are products of coordinate subgroups, so $K$ is one of these, and it is proper. Hence some coordinate copy of $S$ survives in $G/K$, and therefore embeds in $\Sym(G/H)\cong\Sym(j)$. So $j\geq \iota(S)>k$. Thus $G$ itself is the only subgroup of index at most $k$, and $\ab_k(G)=1$ since $G$ is perfect.

An irreducible representation of $G$ is an external tensor product of irreducible representations of the factors, so its degree is the product of theirs. If two of them are nontrivial, the degree is at least $\ell(S)^2>k$. So an irreducible representation of $G$ of degree at most $k$ is either trivial, or nontrivial in exactly one coordinate, where it is one of the $\Rep_k(S)-1$ nontrivial irreducible representations of $S$ of degree at most $k$.

(2) A random pair of elements generates $S$ with probability tending to $1$ as $|S|\to\infty$, by Dixon \cite{Dix69} for the alternating groups and by \cite{LS95} in general. A $d$-tuple generates as soon as its first two entries do, and so
\[
    (1-o(1))|S|^d \leq \phi_2(S)\cdot|S|^{d-2} \leq \phi_d(S) \leq |S|^d .
\]
Since $t_d(S)=\phi_d(S)/\lvert\Aut(S)\rvert$ and $\lvert\Aut(S)\rvert=|S|\cdot\lvert\Out(S)\rvert$, the claim follows.

(3) We may assume $k\geq\ell(S)$, since otherwise $\Rep_k(S^t)=1$. As in (1), an irreducible representation of $S^t$ of degree at most $k$ is nontrivial in at most $b=\lfloor\log k/\log\ell(S)\rfloor$ coordinates. Liebeck and Shalev \cite[Th.~1.1(i)]{LS05} showed that $\zeta_S(s)=\sum_{\chi\in\Irr(S)}\chi(1)^{-s}$ tends to $1$ as $|S|\to\infty$, for every $s>1$. So $\zeta_S(2)\leq C'$ for an absolute constant $C'$, and hence $\Rep_k(S)\leq \zeta_S(2) k^2\leq C'k^2$. Choosing the coordinates and the representations on them gives
\[
    \Rep_k(S^t) \leq \sum_{j=0}^b\binom tj\left(\Rep_k(S)-1\right)^j \leq \sum_{j=0}^b\left(C'k^2t\right)^j \leq (b+1)\left(C'k^2t\right)^b .
\]
The logarithm of the right-hand side is $O(\log^2k)+b\log t$, and
\[
    b\log t \leq C\ell(S)\frac{\log k}{\log\ell(S)} \leq Ck ,
\]
since $x/\log x$ is increasing for $x\geq e$ and $k\geq\ell(S)$.
\end{proof}

A version of the following example with the exponent $d-2$ in place of $d-1$ is \cite[Example~18]{Sab23}, which goes back to Kassabov and Nikolov \cite{KN06}.

\begin{example}\label{ex:altpower}
Let $d\geq2$. For $n\geq7$, put $S=\Alt(n)$, $G_n=S^{t_d(S)}$ and $k=n-1$. We claim that
\[
    \ab_k(G_n) = 1,
    \qquad
    \Rep_k(G_n) = 1+t_d(S) = (1-o(1))\frac{|S|^{d-1}}{2} .
\]
Here $\iota(S)=n$, and for $n\geq7$ the standard module is the only nontrivial irreducible $S$-module of degree at most $n-1$. So $\ell(S)=k$ and $\Rep_k(S)=2$, and since $k<\min\{n,k^2\}$, the first two claims follow from \Cref{prop:powers}(1). For the third, $\lvert\Out(S)\rvert=2$, so \Cref{prop:powers}(2) gives $t_d(S)=(1-o(1))|S|^{d-1}/2$. Since $|S|=(k+1)!/2$, this is $(1-o(1))((k+1)!)^{d-1}/2^d$.
\end{example}

To record how far $\Rep_k$ can outrun $\ab_k$ within a class $\mathcal{C}$ of finite groups, we use the \emph{gap function}
\[
    \Delta_{\mathcal{C}}(k) = \sup_{G\in\mathcal{C}}\frac{\Rep_k(G)}{\ab_k(G)} .
\]
Together with \Cref{ex:altpower}, \Cref{thm:main} shows that the gap function of the class $\mathcal{C}_d$ of all finite $d$-generated groups is
\[
    \Delta_{\mathcal{C}_d}(k) = \exp\left((d-1)k\log k+O(dk)\right) .
\]

\section{The number of quasiprimitive subgroups}\label{sec:qp}

In this section we bound the number of conjugacy classes of quasiprimitive subgroups,
\[
    \QP(m) = \#\left\{ \text{quasiprimitive } Q\leq\PGL_{m'}(\C) \mid m' \leq m \right\}\big/\text{conjugacy}.
\]
In \Cref{prop:reduction}, this number is multiplied by the prefactor $\exp((d-1)k\log k+O(dk))$. For the purposes of \Cref{thm:main} it therefore suffices to prove that $\log \QP(m)=O(m)$, so that $\QP(k)$ is absorbed by the error term.
We prove the stronger bound $\log \QP(m)=O(\log^4m)$, which is \Cref{thm:count}. We have found no bound on $\QP(m)$ in the literature, beyond the folklore fact that it is finite \cite[Rem.~2.61]{Cal25}.\footnote{Complete classifications of the finite primitive linear groups are available only in low degrees, see \cite[Sec.~5]{Cal25} and the survey \cite{DF17}. The only asymptotic count we are aware of is due to Rossmann \cite[Prop.~3.3]{Ros16}. He shows that for a number field $K$ and every $\delta>0$, the primitive finite \emph{nilpotent} subgroups of $\GL_m(K)$ form only $O(m^{1+\delta})$ conjugacy classes.}

As noted in the introduction, a soft count is not enough. A quasiprimitive $Q\leq\PGL_m(\C)$ has order at most $(m+1)!$ by \factref{J}, but there are $\exp(\Theta(m^3\log^3m))$ groups of such orders \cite[Th.~4.5 and Cor.~16.21]{BNV07}, far above the target $\exp(O(m))$.
The loss comes from treating $Q$ as an abstract group of bounded order. A typical group of order at most $(m+1)!$ is nowhere near a quasiprimitive subgroup of $\PGL_m(\C)$, and a good count has to use the structural restrictions imposed by quasiprimitivity.

\begin{remark}
The same question has been studied for permutation groups, where Pyber and Shalev \cite{PS97} and later Jaikin-Zapirain and Pyber \cite{JZP11} bounded the number of primitive subgroups of $\Sym(n)$ up to conjugacy. By the O'Nan--Scott theorem, the expensive ones are of affine type $\F_p^b\rtimes P_0$, so this is really a count of irreducible subgroups $P_0\leq\GL_b(\F_p)$. Jaikin-Zapirain and Pyber reduce it further to primitive linear groups over $\F_p$, and count these through their generalized Fitting subgroups. Our proof is modelled on theirs, with $\C$ in place of $\F_p$, and we compare the two arguments in detail in \Cref{sub:permutation}.
\end{remark}

Here is our argument in brief. Let $Q$ be the image of a primitive $\Gamma\leq\GL_m(\C)$. The data we record come from the generalized Fitting subgroup $R=F^*(\Gamma)$. By work of Collins \cite{Col08}, $R$ is a central product of extraspecial and quasisimple groups, up to scalars, and $\Gamma$ acts on it faithfully modulo scalars. This keeps the quotient $\Gamma/R$ small. Moreover, $V=\C^m$ decomposes as a tensor product $U\otimes W$ on which $R$ acts through the first factor alone, and $U$ is in turn a tensor product with one tensor factor for each factor of the central product $R$. We record $Q$ up to conjugacy through four pieces of data.
\begin{enumerate}
    \item \emph{The degrees.} These are the dimensions of the tensor factors of $U$. They are integers at least $2$ whose product divides $m$, so there are few possible lists.
    \item \emph{The group $\bar R$.} This is the image of $R$ in $Q$, recorded up to conjugacy through each factor of $R$ in its own tensor factor. For an extraspecial factor this is essentially free. For a quasisimple factor it is a genuine counting problem, namely counting the conjugacy classes of irreducible quasisimple subgroups of $\GL_t(\C)$, and here we rely on the classification.
    \item \emph{The group $Q_2$.} This is the image of $\Gamma$ on $W$. It is again quasiprimitive, of degree $\dim W$, which is a proper divisor of $m$. We count it by induction, so this is the recursive term.
    \item \emph{The group $Q$ itself.} The first three pieces trap $Q$ between two known groups. From below, $Q$ contains $\bar R$. From above, $Q$ lies in $\mathcal N\otimes Q_2$, where $\mathcal N$ is the normalizer of $\bar R$ in $\PGL(U)$. So $Q$ amounts to a subgroup of a small finite group, which we record through a few of its generators.
\end{enumerate}
The first two pieces are cheap, and the fourth costs $\exp(O(\log^3m))$. The recursion in the third piece runs down a chain of divisors of $m$ of length at most $\log_2m$, and this gives the fourth logarithm.
The outline follows \cite[Prop.~5.9]{JZP11}, except that we take all tensor factors at once and recurse only on $W$, for reasons explained in \Cref{sub:permutation}.

Here is how this section is organized. \Cref{sub:structure} states Collins's structure theory and fixes the data. \Cref{sub:recursion} counts $Q$ through the four pieces, treating the quasisimple factors as a black box for the moment, and \Cref{sub:components} then counts the quasisimple factors. \Cref{sub:proofcount} puts the two together and proves \Cref{thm:count,thm:main}. \Cref{rem:truth} discusses how far the count is from the truth, and \Cref{sub:permutation} compares it with its counterpart over finite fields.

\subsection{The structure of a primitive linear group}\label{sub:structure}

We now make the structure described above precise. We state what we use in the form given by Collins, fix the data, and record some basic facts about quasisimple linear groups.

\subsubsection{Components and quasicomponents}

Let $\Gamma\leq\GL_m(\C)$ be finite and primitive, and write $Z=Z(\Gamma)$. Following \cite[p.~763]{Col08}, the \emph{components} of $\Gamma$ are its subnormal quasisimple subgroups $E_1,\dots,E_s$, and its \emph{quasicomponents} are the subgroups $O_p(\Gamma)$ that are not cyclic, where $O_p(\Gamma)$ is the largest normal $p$-subgroup of $\Gamma$. The \emph{Fitting subgroup} $F(\Gamma)$ is the largest nilpotent normal subgroup of $\Gamma$, namely the product of the subgroups $O_p(\Gamma)$ over all primes $p$, and the \emph{layer} $E(\Gamma)$ is the subgroup generated by the components. These two subgroups commute elementwise, and the \emph{generalized Fitting subgroup} is their product $R=F^*(\Gamma)=F(\Gamma)E(\Gamma)$. We write $\bar R=R/Z$ for its image in $Q=\Gamma/Z$.

\begin{itemize}
\fact{C} \emph{The structure of a primitive linear group} \cite[Lems.~1, 2 and p.~763]{Col08}. Let $P_1,\dots,P_r$ be the quasicomponents, for distinct primes $p_1,\dots,p_r$. Each of them has the form $P_i=Z(P_i)T_i$ with $T_i$ extraspecial of order $p_i^{1+2n_i}$, of exponent $p_i$ when $p_i$ is odd. The generalized Fitting subgroup is the central product
\[
    R = Z\cdot P_1\cdots P_r\cdot E_1\cdots E_s ,
\]
whose factors commute elementwise. Moreover $C_\Gamma(R)=Z$, since $C_\Gamma(R)\leq R$ holds for any finite group, so $C_\Gamma(R)=Z(R)$ is an abelian normal subgroup of $\Gamma$, and hence central.
\end{itemize}

In words, $R$ is a central product of extraspecial and quasisimple groups, up to scalars, and nothing outside the scalars centralizes it. So $Q=\Gamma/Z$ acts faithfully on $R$ by conjugation.

\subsubsection{The tensor decomposition and its data}\label{subsub:data}

Throughout, $m\geq2$ and $\Gamma\leq\GL(V)=\GL_m(\C)$ is a finite primitive subgroup with image $Q=\Gamma/Z$ in $\PGL_m(\C)$, where $Z=Z(\Gamma)$. In particular, $\Gamma$ acts irreducibly on $V$. Since $m\geq2$, it is therefore nonabelian, and $Q$ is nontrivial. The groups $R=F^*(\Gamma)$ and $\bar R=R/Z\trianglelefteq Q$, the components $E_1,\dots,E_s$ and the quasicomponents $P_i=Z(P_i)T_i$ are as in \factref{C}. We need one more piece, a decomposition of $V$.

\begin{lemma}\label{lem:shape}
There is an isomorphism $V\cong U\otimes W$ such that the following hold.
\begin{enumerate}
\item $R$ acts as $\rho_U\otimes1$ and faithfully on $U$, and $U$ is the Kronecker product of faithful irreducible modules for the factors of the central product $R$. The quasicomponent $P_i$ contributes a factor of degree $p_i^{n_i}$ and the component $E_j$ a factor of degree $m_j\geq2$.
\item Write $u=\dim U$ and $w=\dim W$. Then $m=uw$ and $u=\prod_ip_i^{n_i}\cdot\prod_jm_j$. Hence $\sum_jm_j\leq\prod_jm_j\leq u$ and there are at most $\log_2u$ factors in total. Moreover $u\geq2$, and hence $w\leq m/2$.
\item Every $g\in\Gamma$ has the form $g=X_g\otimes Y_g$ with $X_g\in\GL(U)$ and $Y_g\in\GL(W)$, the two determined by $g$ up to reciprocal scalars. In particular $Q$ has well-defined images $Q_1\leq\PGL_u(\C)$ on $U$ and $Q_2\leq\PGL_w(\C)$ on $W$.
\end{enumerate}
\end{lemma}

\begin{proof}
By homogeneity \cite[Cor.~16.2.1]{Sup76}, we have $\Res^\Gamma_R V=U^{\oplus w}$ for an irreducible $R$-module $U$ of some degree $u$, and $m=uw$. In the resulting decomposition $V\cong U\otimes W$, the group $R$ acts as $\rho_U\otimes1$. So $R$ acts faithfully on $U$, since an element acting trivially on $U$ acts trivially on $V$. The factors of the central product $R$ commute elementwise by \factref{C}. Write $\rho_U(R)=HF$ with $H$ one factor and $F$ the product of the others. These are normal subgroups of $\rho_U(R)$ that commute elementwise, so $\Res^R_HU=L^{\oplus c}$ for an irreducible $H$-module $L$ by \cite[Lem.~16.4]{Sup76}. Thus $U\cong L\otimes M$ with $M=\C^c$ and $H$ acting as $\rho_L\otimes1$. Every element of $F$ commutes with $H$, so it lies in $\End_H(U)=\End_H(L)\otimes\End(M)$, which is $1\otimes\End(M)$ by Schur's lemma. So $F$ acts as $1\otimes\rho_M$ for a representation $\rho_M$ of $F$ on $M$. If $M'\leq M$ is an $F$-submodule, then $L\otimes M'$ is an $R$-submodule of $U$, so $M$ is irreducible. Both $L$ and $M$ are faithful, since an element of $H$ or $F$ acting trivially on its factor acts trivially on $U$. Applying the same argument to $F$ acting on $M$ and inducting on the number of factors, $U$ is the Kronecker product of faithful irreducible modules for the factors of $R$, of degrees multiplying to $u$. 
A quasicomponent $P_i=Z(P_i)T_i$ acts on its factor through $T_i$ modulo scalars, and a faithful irreducible representation of $T_i$ has degree $p_i^{n_i}$ \cite[p.~335]{Suz86}.
A component $E_j$ is nonabelian, so it contributes a factor of degree $m_j\geq2$. This proves (1).

The relations in (2) follow. The inequality $\sum_jm_j\leq\prod_jm_j$ holds because each $m_j\geq2$, and the bound on the number of factors holds because each factor has degree at least $2$. If $u$ were $1$, then $R$ would consist of scalars, and \factref{C} would give $\Gamma=C_\Gamma(R)=Z$, contradicting $Q\neq1$. So $u\geq2$ and $w\leq m/2$.

For (3), the image $\rho_U(R)$ spans $\End(U)\otimes1$ by Burnside's theorem, since $U$ is irreducible. Every $g\in\Gamma$ normalizes $R$, so conjugation by $g$, which is an algebra automorphism of $\End(V)$, maps the linear span of $R$ to itself. Thus it restricts to an automorphism of $\End(U)\otimes1$. By the Skolem--Noether theorem, it agrees there with conjugation by $X_g\otimes1$ for some $X_g\in\GL(U)$. Then $g(X_g\otimes1)^{-1}$ centralizes $\End(U)\otimes1$, and therefore lies in $1\otimes\GL(W)$ by Schur's lemma, as in the proof of (1). So $g=X_g\otimes Y_g$, and rescaling $X_g$ by $\lambda$ forces $Y_g$ to be rescaled by $\lambda^{-1}$. Hence the images $Q_1$ and $Q_2$ are well defined.
\end{proof}

Let us see how all of these groups fit together. For $X\in\GL(U)$ we write $\bar X$ for its image in $\PGL(U)$, and likewise on $W$. A scalar on $V$ can only be written as $X\otimes Y$ with both $X$ and $Y$ scalar, so the maps $g\mapsto\bar X_g$ and $g\mapsto\bar Y_g$ kill $Z$ and give surjections $Q\to Q_1$ and $Q\to Q_2$. The first of them is injective. Indeed, if $\bar X_g=1$, then $g$ acts on $V$ as $1\otimes Y$ for some $Y\in\GL(W)$, so it centralizes $R$, and $g\in C_\Gamma(R)=Z$ by \factref{C}. Hence $Q\cong Q_1$, and under this isomorphism $\bar R$ becomes the image of $\rho_U(R)$ in $\PGL(U)=\PGL_u(\C)$, a normal subgroup of $Q_1$. From now on we view $\bar R$ in $\PGL_u(\C)$ in this way, and write
\[
    \mathcal{N} = N_{\PGL_u(\C)}(\bar R)
\]
for its normalizer, so that $\bar R\trianglelefteq Q_1\leq\mathcal N$. The second map kills $\bar R$, since $R$ acts trivially on $W$, and so it induces a homomorphism $\psi\colon Q_1\to Q_2$ with $\bar R$ in its kernel. Thus $Q\leq\PGL(V)$ is the graph of $\psi$, namely the image of the elements $X\otimes Y$ with $\bar X\in Q_1$ and $\bar Y=\psi(\bar X)$. In words, $Q$ is abstractly the group $Q_1$ acting on $U$, and $\psi$ records how it acts on the multiplicity space $W$.
\Cref{fig:groups} summarizes the situation.

\begin{figure}[t]
\centering
\begin{tikzpicture}[>=stealth]
    \node (G) at (0,2.5) {$\Gamma$};
    \node (RR) at (-1.8,2.5) {$R$};
    \node (Q) at (0,1.4) {$Q$};
    \node (R) at (-1.8,1.4) {$\bar R$};
    \node (Q1) at (-1.8,0) {$Q_1$};
    \node (Q2) at (1.8,0) {$Q_2$};
    \node (N) at (-3.1,0) {$\mathcal N$};
    \node (U) at (-4.7,0) {$\PGL(U)$};
    \node (W) at (3.6,0) {$\PGL(W)$};
    \draw[->>] (G) -- (Q);
    \draw[->>] (RR) -- (R);
    \draw[{Hooks[right]}->] (RR) -- (G);
    \draw[{Hooks[right]}->] (R) -- (Q);
    \draw[{Hooks[right]}->] (R) -- (Q1);
    \draw[->] (Q) -- node[above left, font=\scriptsize, inner sep=1pt] {$\cong$} node[below right, font=\tiny, inner sep=0.5pt] {on $U$} (Q1);
    \draw[->>] (Q) -- node[below left, font=\tiny, inner sep=0.5pt] {on $W$} (Q2);
    \draw[->>] (Q1) -- node[below, font=\scriptsize] {$\psi$} (Q2);
    \draw[{Hooks[left]}->] (Q1) -- (N);
    \draw[{Hooks[left]}->] (N) -- (U);
    \draw[{Hooks[right]}->] (Q2) -- (W);
\end{tikzpicture}
\caption{The groups attached to a primitive $\Gamma\leq\GL(U\otimes W)$.}
\label{fig:groups}
\end{figure}
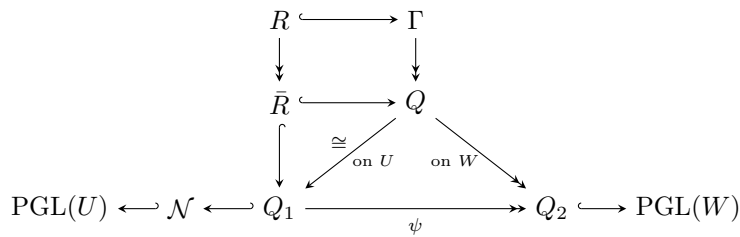

\begin{example}\label{ex:multiplicity}
Let $T=3^{1+2}$ be the extraspecial group of order $27$ and exponent $3$, acting faithfully and irreducibly on $U=\C^3$, and let $D$ be its normalizer in $\SL(U)$. The commutator map induces a nondegenerate alternating form on $T/Z(T)\cong\F_3^2$ with values in $Z(T)\cong\F_3$. Conjugation by $D$ preserves this form, and this gives a homomorphism $D\to\Sp_2(\F_3)=\SL_2(\F_3)$. This homomorphism is surjective with kernel $T$, so $D/T\cong\SL_2(\F_3)$. Composing with a faithful representation of $\SL_2(\F_3)$ on $W=\C^2$ gives a homomorphism $\varphi\colon D\to\GL(W)$. Let $\Gamma\leq\GL(U\otimes W)=\GL_6(\C)$ be the image of $D$ under $X\mapsto X\otimes\varphi(X)$. The following facts can be verified by a computation in \textsf{GAP}.
\begin{itemize}
    \item $\Gamma\cong D$ has order $648$, and it is a primitive subgroup of $\GL_6(\C)$ with centre $Z=Z(T)$ of order $3$.
    \item $R=T$ is the single quasicomponent, with $p_1=3$ and $n_1=1$, and there are no components.
    \item The decomposition of \Cref{lem:shape} is $V=U\otimes W$, with $u=3$ and $w=2$. So $\Res^\Gamma_RV=U^{\oplus2}$ is not irreducible.
    \item $Q_1=D/Z(T)$ is the Hessian group of order $216$ in $\PGL_3(\C)$, while $Q_2=\SL_2(\F_3)/\{\pm1\}\cong\Alt(4)$ is the tetrahedral group in $\PGL_2(\C)$.
    \item $\bar R=T/Z(T)\cong\F_3^2$, and its normalizer is $\mathcal N=Q_1$.
\end{itemize}
The homomorphism $\psi\colon Q_1\to Q_2$ sends the class of $X\in D$ to the image of $\varphi(X)$ in $\PGL(W)$. So it is the composite of the quotient map $D/Z(T)\to D/T\cong\SL_2(\F_3)$ with the quotient map $\SL_2(\F_3)\to\SL_2(\F_3)/\{\pm1\}$. Its kernel has order $18$ and contains $\bar R$ with index $2$, the extra coset coming from the central element $-1$ of $\SL_2(\F_3)$. So the kernel of $\psi$ can be strictly larger than $\bar R$. The three faithful representations of $\SL_2(\F_3)$ of degree $2$ differ by linear characters, so they give the same $\psi$ and hence the same $Q$. Thus $\bar R$ and $\mathcal N$ are determined by $T$ alone, while the embedding of $Q\cong Q_1$ into $\PGL(V)$ also depends on $\psi$.
\end{example}

\subsubsection{Quasisimple linear groups}

The components of $R$ will enter the count as quasisimple groups acting irreducibly in a degree of their own. We record the basic facts about their size that we will require.

\begin{lemma}\label{lem:basic}
Let $t\geq2$, and let $E\leq\GL_t(\C)$ be finite, quasisimple and irreducible. Put $S=E/Z(E)$. Then
\[
    |Z(E)|\leq t, \qquad |E|\leq\exp(O(t\log t)), \qquad \lvert\Out(S)\rvert=O(\log|E|) .
\]
\end{lemma}

\begin{proof}
Set $Z=Z(E)$. By Schur's lemma, $Z$ consists of scalars. Since $E$ is perfect, we have $E\leq\SL_t(\C)$, so every scalar in $E$ has determinant $1$, and therefore $|Z|\leq t$. By \factref{J}, the group $E$ has an abelian normal subgroup $A$ with $|E:A|\leq J(t)$. Since $S$ is simple, its abelian normal subgroup $AZ/Z$ is trivial, so $A\leq Z$, and hence
\[
    |E| = |Z|\cdot|S| \leq tJ(t) = \exp\left(O(t\log t)\right) .
\]

For the outer automorphisms, it suffices to show that $\lvert\Out(S)\rvert=O(\log|S|)$. For the alternating and sporadic groups, the outer automorphism group has order at most $4$. For a group of Lie type over $\F_q$ with $q=p^f$ and of untwisted rank $r$, the tables \cite[Tables~5.1.A, 5.1.B]{KL90} give $\lvert\Out(S)\rvert\leq6f(r+1)$, the three factors coming from the graph, field and diagonal automorphisms. On the other hand, $\log|S|=\Omega(r^2f)$.
\end{proof}

\subsection{Counting the four pieces}\label{sub:recursion}

We now count $Q$ through the four pieces of data described at the start of the section. Let us first fix some notation.

Write $\qp(m)$ for the number of conjugacy classes of quasiprimitive subgroups of $\PGL_m(\C)$ in degree exactly $m$, so that $\QP(m)=\sum_{m'\leq m}\qp(m')$. The count will proceed by recursion on $\qp$, and it will descend through divisors of $m$.

The components enter the recursion through a single quantity. Write $\qs(t)$ for the number of $\GL_t(\C)$-conjugacy classes of irreducible quasisimple subgroups of $\GL_t(\C)$, and set
\[
    \QS(m) = \max\left\{ \prod_j\qs(m_j)  \ \middle|\  m_j\geq2 \text{ integers with } \prod_jm_j\leq m  \right\} .
\]
This is the largest possible product of component counts in degree $m$, and it is the only input about the components that the recursion uses. We bound it in \Cref{sub:components}.

Finally, write $\mu(m)$ for the largest exponent in the prime factorization of $m$, and set
\[
    \nu(m) = \mu(m)\log m+\log m\log\log m .
\]
Since $\mu(m)\leq\log_2m$, with equality at $m=2^n$, we always have $\nu(m)=O(\log^2m)$, while $\nu(m)=O(\log m\log\log m)$ when $m$ is squarefree.

\subsubsection{The four pieces}

Here is the summary of the costs of the four pieces, which will be proved in the rest of the subsection.
\begin{enumerate}
    \item The degrees of the tensor factors cost $\exp(O(\log m\log\log m))$, by \Cref{lem:degrees}.
    \item The group $\bar R$ up to conjugacy costs $m^2\cdot \QS(m)$, by \Cref{lem:factors}. This is the only place where the components are paid for.
    \item The multiplicity group $Q_2$ costs $\qp(w)$, where $w$ is a proper divisor of $m$, by \Cref{lem:mult}. This is the recursive term.
    \item The group $Q$ itself costs $\exp(O(\nu(m)\log m))$ once the other three pieces are fixed, by \Cref{lem:group}. This is where the surviving powers of the logarithm come from.
\end{enumerate}
The first two pieces are counted directly. The last two rest on \Cref{lem:norm}, which says that a primitive group is not much larger than its generalized Fitting subgroup, and neither is the normalizer $\mathcal N$.

\subsubsection{The degrees}

\begin{lemma}\label{lem:degrees}
The degrees $p_i^{n_i}$ of quasicomponents and $m_j$ of components take at most $\exp(O(\log m\log\log m))$ values. They determine $u$ and $w=m/u$.
\end{lemma}

\begin{proof}
The factorization of $u$ in \Cref{lem:shape}(2) is a sequence of $f\leq\log_2m$ integers at least $2$ whose product divides $m$. Write $m=\prod_\ell\ell^{e_\ell}$ over distinct primes $\ell$. 
A sequence of $f$ positive integers whose product divides $m$ amounts to a distribution of the $e_\ell$ copies of each prime $\ell$ among $f$ slots and a remainder. 
So there are
\[
    \prod_\ell\binom{e_\ell+f}{f}\leq(f+1)^{\sum_\ell e_\ell}\leq(\log_2m+1)^{\log_2m}
\]
such sequences. 
Each of the $f$ factors must also be labelled as coming from a quasicomponent or from a component, and this gives $2^f\leq m$ labellings. With at most $\log_2m$ choices of $f$, there are in total at most $m\log_2m\cdot(\log_2m+1)^{\log_2m}=\exp(O(\log m\log\log m))$ choices. The degrees multiply to $u$.
\end{proof}

\subsubsection{The group \texorpdfstring{$\bar R$}{R}}

\begin{lemma}\label{lem:factors}
With the degrees $p_i^{n_i}$ and $m_j$ fixed, the number of $\PGL_u(\C)$-conjugacy classes available to $\bar R$ is at most $m^2\cdot\QS(m)$.
\end{lemma}

\begin{proof}
By \Cref{lem:shape}(1), we have $U=U_1\otimes\dots\otimes U_f$, where the $a$th factor $X_a$ of the central product $R$ acts on $U_a$ through a subgroup $L_a\leq\GL(U_a)$ and trivially on the other tensor factors. Since $R$ is generated by the factors together with the scalars $Z$, and Kronecker products multiply factorwise, the image of $R$ in $\PGL(U)$ is the image of the group of all $x_1\otimes\dots\otimes x_f$ with $x_a\in L_a$. For $g_a\in\GL(U_a)$, conjugation by $g_1\otimes\dots\otimes g_f$ replaces each $L_a$ by $g_aL_ag_a^{-1}$. So the class of $\bar R$ in $\PGL_u(\C)$ is determined by the classes of the images $\bar L_a$ of the $L_a$ in the groups $\PGL(U_a)$, and we count these for the two kinds of factors.

Let first $P_i$ be a quasicomponent. Since $Z(P_i)$ acts by scalars, $P_i$ and $T_i$ have the same image modulo scalars. The group $T_i$ is extraspecial of order $p_i^{1+2n_i}$, so it has at most two isomorphism types, and each has exactly $p_i-1$ faithful irreducible representations of degree $p_i^{n_i}$ \cite[p.~335]{Suz86}. There are at most $\log_2m$ factors and $\prod_ip_i^{n_i}\leq u\leq m$, so the quasicomponents contribute at most
\[
    \prod_i2(p_i-1) \leq 2^{\log_2m}\prod_ip_i^{n_i} \leq m^2 .
\]

Consider now a component $E_j$. It acts faithfully on its factor as an irreducible quasisimple subgroup of $\GL_{m_j}(\C)$, and there are $\qs(m_j)$ classes of these.

Multiplying gives at most $m^2\prod_j\qs(m_j)$, and this is at most $m^2\QS(m)$ by the definition of $\QS$, since $\prod_jm_j\leq u\leq m$ by \Cref{lem:shape}(2).
\end{proof}

\subsubsection{The group \texorpdfstring{$Q_2$}{Q2}}

We first bound the size of $Q$ and of $\mathcal N$ modulo $\bar R$. This is used both here and later for the group $Q$ itself.

\begin{lemma}\label{lem:norm}
The normalizer $\mathcal{N}$ is finite. Moreover, both $|\mathcal{N}/\bar R|$ and the index $|\Gamma:R|=|Q/\bar R|$ are at most $\exp(O(\nu(m)))$.
\end{lemma}

\begin{proof}
Write $\bar R=A\times\bar E$, where $A=\prod_i\F_{p_i}^{2n_i}$ is the image of the quasicomponents and $\bar E=\prod_jS_j$ is the product of the simple groups $S_j=E_j/Z(E_j)$. 
Here $A$ is abelian, while $\bar E$ is perfect and centreless. Hence $A=Z(\bar R)$ and $\bar E=[\bar R,\bar R]$ are both characteristic in $\bar R$, and so $\Out(\bar R)=\Aut(A)\times\Out(\bar E)$.
Both $\mathcal N$ and $Q$ act on $\bar R$ by conjugation, and we bound them in the same way. Let $K$ be either of them. Conjugation gives a homomorphism $K\to\Aut(\bar R)$ with kernel $C_K(\bar R)$. Since $\bar R\leq K$, composing it with the projection onto $\Out(\bar R)$ gives a homomorphism $K\to\Out(\bar R)$ with kernel $\bar R\,C_K(\bar R)$. Hence
\[
    |K/\bar R| \leq |C_K(\bar R)|\cdot\lvert\Out(\bar R)\rvert ,
\]
and it suffices to bound the two factors on the right. We first bound the centralizer and then $\Out(\bar R)$.

\emph{The centralizers.} Let $x\in\mathcal N$ centralize $\bar R$. Then $x$ lifts to some $g\in\GL(U)$ with $grg^{-1}=\lambda(r)r$ for all $r\in R$, where each $\lambda(r)$ is a scalar. Since scalars are central, $\lambda\colon R\to\C^\times$ is a homomorphism. It is trivial on the scalars, and hence factors through $\bar R$. Moreover, $\lambda$ determines $g$ up to an element of $C_{\GL(U)}(R)$, which consists of scalars by Schur's lemma, since $R$ is irreducible on $U$. So $\lambda$ determines $x$, and the centralizer of $\bar R$ in $\mathcal N$ has order at most $\lvert\Hom(\bar R,\C^\times)\rvert=|A|\leq m^2$, the components being perfect. The same argument applies to an element of $Q$ centralizing $\bar R$, lifted to $g\in\Gamma$, with $C_\Gamma(R)=Z$ from \factref{C} in place of Schur's lemma.

\emph{The automorphisms.} It remains to bound $\lvert\Out(\bar R)\rvert=\lvert\Aut(A)\rvert\cdot\lvert\Out(\bar E)\rvert$. By \Cref{lem:shape}(2), each $p_i^{n_i}$ divides $m$, and the primes $p_i$ are distinct. Hence $n_i\leq\mu(m)$ and $\sum_in_i\log p_i\leq\log m$, and so
\[
    \lvert\Aut(A)\rvert = \prod_i|\GL_{2n_i}(p_i)| \leq \exp\left(4\sum_in_i^2\log p_i\right) \leq \exp\left(4\mu(m)\log m\right) .
\]
For $\Out(\bar E)$, note that the simple factors $S_1,\dots,S_s$ are the minimal normal subgroups of $\bar E$, so every automorphism of $\bar E$ permutes them. This gives a homomorphism $\Aut(\bar E)\to\Sym(s)$, whose kernel consists of the automorphisms fixing each factor, namely $\prod_j\Aut(S_j)$. So this kernel has index at most $s!$, and it contains $\Inn(\bar E)=\prod_j\Inn(S_j)$. Hence
\[
    \lvert\Out(\bar E)\rvert \leq s!\cdot\prod_j\lvert\Out(S_j)\rvert .
\]
Since $s\leq\log_2m$, the factor $s!$ is at most $\exp(O(\log m\log\log m))$. Each $\Out(S_j)$ has order $O(\log|E_j|)=O(m_j\log m_j)$ by \Cref{lem:basic} in degree $m_j$, and since $\prod_jm_j\leq m$, the product is $m^{O(1)}$.

Multiplying the three bounds, we find that the orders of $\mathcal N/\bar R$ and of $Q/\bar R$ are both at most $\exp(O(\nu(m)))$. In particular $\mathcal N$ is finite, since $\bar R$ is. Finally, $|Q/\bar R|=|\Gamma:R|$ because $Z\leq R$.
\end{proof}

% The automorphisms of $A$ induced by $\Gamma$ or by $\mathcal N$ preserve the commutator pairing on each $T_i/Z(T_i)$, so they lie in $\prod_i\Sp_{2n_i}(p_i)$, as in Collins's description of $\Gamma/R$ \cite[Th.~5]{Col08}. We do not need this, since it only improves the constant in the exponent.

\begin{lemma}\label{lem:mult}
The group $Q_2$ is quasiprimitive, its degree $w$ is a proper divisor of $m$, and $|Q_2|\leq\exp(O(\nu(m)))$.
\end{lemma}

\begin{proof}
That $w$ divides $m$ and $w\leq m/2$ is \Cref{lem:shape}(2). 
By \Cref{lem:shape}(1), $R$ acts trivially on the multiplicity space $W$, so the map $Q\to Q_2$ kills $\bar R$. Hence $Q_2$ is a quotient of $Q/\bar R$, and $|Q/\bar R|\leq\exp(O(\nu(m)))$ by \Cref{lem:norm}. 
For quasiprimitivity, let $\Gamma_2\leq\SL_w(\C)$ be the full preimage of $Q_2$ under the surjection $\SL_w(\C)\to\PGL_w(\C)$. This is a finite group of order $w|Q_2|$ with image $Q_2$. Suppose that $\Gamma_2$ permutes the summands of a decomposition $W=W_1\oplus\dots\oplus W_b$ with $b\geq2$. Every $g\in\Gamma$ acts as $X_g\otimes Y_g$ by \Cref{lem:shape}(3), and $Y_g$ is a scalar multiple of an element of $\Gamma_2$. So $\Gamma$ would permute the subspaces $U\otimes W_1,\dots,U\otimes W_b$, contradicting the primitivity of $\Gamma$. So $\Gamma_2$ is primitive, and $Q_2$ is quasiprimitive.
\end{proof}

\subsubsection{The group \texorpdfstring{$Q$}{Q}}

It remains to count the possibilities for $Q$ once the other three pieces are fixed. We view $Q$ as a subgroup of $\mathcal{N}\otimes Q_2$, and we count such subgroups through their generators. For this, we rely on the following result of Lucchini, Menegazzo and Morigi.

\begin{lemma}[{\cite[Th.~B]{LMM01}}]\label{lem:gen}
There is an absolute constant $D$ such that every finite primitive subgroup of $\GL_m(\C)$ with $m\geq2$ is generated by at most $D\log m$ elements.
\end{lemma}

A quotient needs no more generators than the group, so the same bound holds for $Q$.

\begin{lemma}\label{lem:group}
Fix the degrees, a subgroup of $\PGL_u(\C)$ for $\bar R$ and a subgroup of $\PGL_w(\C)$ for $Q_2$. Then the groups $Q$ arising from some $\Gamma$ as in \Cref{subsub:data} with these data are at most $\exp(O(\nu(m)\log m))$ in number.
\end{lemma}

\begin{proof}
The Kronecker product $(X,Y)\mapsto X\otimes Y$ is a homomorphism $\GL(U)\times\GL(W)\to\GL(V)$, and $X\otimes Y$ is a scalar only if both $X$ and $Y$ are. So it induces an injective homomorphism
\[
    \kappa\colon\PGL(U)\times\PGL(W)\to\PGL(V),
    \qquad
    (\bar X,\bar Y)\mapsto\overline{X\otimes Y} .
\]
For subgroups $A\leq\PGL(U)$ and $B\leq\PGL(W)$ we write $A\otimes B=\kappa(A\times B)$, which is isomorphic to $A\times B$. As we saw after \Cref{lem:shape}, $Q$ is the image under $\kappa$ of the graph $\{(x,\psi(x))\mid x\in Q_1\}$ of $\psi$. This graph lies in $\mathcal N\times Q_2$, since $Q_1\leq\mathcal N$, and it contains $\bar R\times1$, since $\psi$ kills $\bar R$. Hence
\[
    \bar R\otimes1 \leq Q \leq \mathcal N\otimes Q_2 ,
\]
and $\mathcal N\otimes Q_2$ is finite by \Cref{lem:norm}. Therefore $Q$ is determined by the subgroup $\Lambda_0=Q/(\bar R\otimes1)$ of
\[
    \Lambda = (\mathcal N\otimes Q_2)/(\bar R\otimes1) \cong \mathcal N/\bar R\times Q_2 ,
\]
and $|\Lambda|=|\mathcal{N}/\bar R|\cdot|Q_2|\leq\exp(O(\nu(m)))$ by \Cref{lem:norm,lem:mult}. As a quotient of $Q$, the group $\Lambda_0$ is generated by $D'=\lceil D\log m\rceil$ elements by \Cref{lem:gen}. A subgroup generated by $D'$ elements is determined by a $D'$-tuple of its elements, so there are at most $|\Lambda|^{D'}$ of them. So at most
\[
    \exp\left(O(\nu(m))\right)^{D\log m+1} = \exp\left(O(\nu(m)\log m)\right)
\]
choices of $Q$ remain.
\end{proof}

\subsubsection{The recursion}

The recursion now follows by multiplying the four counts.

\begin{lemma}\label{lem:recursion}
For all $m\geq2$ we have
\[
    \log \qp(m) \leq \max_{w\mid m,\ w<m}\log \qp(w) + \log \QS(m) + O\left(\nu(m)\log m\right) .
\]
\end{lemma}

\begin{proof}
Let $\Gamma$ be as in \Cref{subsub:data}, so that $Q$ is an arbitrary quasiprimitive subgroup of $\PGL_m(\C)$. By \Cref{lem:degrees}, the degrees take at most $\exp(O(\log m\log\log m))$ values, and they fix $w=m/u$, which is a proper divisor of $m$ by \Cref{lem:mult}. Then \Cref{lem:factors} leaves at most $m^2\QS(m)$ choices of $\bar R$ up to conjugacy, and \Cref{lem:mult} leaves at most $\qp(w)$ choices of $Q_2$. We fix one of each. Since $Q$ is only counted up to conjugacy, we may conjugate it by elements $X\otimes Y$ with $X\in\GL(U)$ and $Y\in\GL(W)$, which conjugate $\bar R$ by $X$ and $Q_2$ by $Y$ independently. So we may assume that $\bar R$ and $Q_2$ are the fixed representatives. Finally, \Cref{lem:group} leaves at most $\exp(O(\nu(m)\log m))$ choices of $Q$. Every quasiprimitive $Q\leq\PGL_m(\C)$ is counted in this way, up to conjugacy. Multiplying the four counts gives the recursion, since $\log m\log\log m$ and $\log m$ are both $O(\nu(m)\log m)$.
\end{proof}

\subsection{The quasisimple factors}\label{sub:components}

It remains to bound $\QS(m)$. Recall that $\qs(t)$ is the number of conjugacy classes of irreducible quasisimple subgroups of $\GL_t(\C)$, and write $\mathcal{S}(t)$ for the set of isomorphism types of the simple groups $E/Z(E)$ for such subgroups $E$. Every $S\in\mathcal{S}(t)$ thus has a faithful projective representation of degree $t$. The degrees of these representations for finite simple groups are well understood by the work of Seitz and Zalesskii \cite{SZ93}. Using these results, we shall bound the field size and the rank of a group of Lie type in $\mathcal{S}(t)$, and the number of groups in $\mathcal S(t)$. This will allow us to deduce that $\qs(t)=t^{O(1)}$.

\subsubsection{Groups of Lie type}

We use the standard notation $S={}^{\epsilon}X_r(q)$ for a simple group of Lie type, so that for instance ${}^2A_r(q)=\PSU_{r+1}(q)$ and ${}^2B_2(q)$ is a Suzuki group. We call $q$ its field size, and we call $r$, the rank of the untwisted root system $X_r$, its \emph{untwisted rank}. The minimal degrees of \cite{SZ93}, the outer automorphism groups of \cite{KL90} and the bounds on the number of conjugacy classes in \cite{FG12} are all stated in terms of these two parameters.

\begin{lemma}\label{lem:liedata}
Let $t\geq2$ and let $S\in\mathcal{S}(t)$ be of Lie type over $\F_q$ of untwisted rank $r$. Then $q=O(t)$ if $r=1$ and $q^{r-1}=O(t)$ if $r\geq2$, and in both cases
\[
    r \leq \log_2t+O(1),
    \qquad
    r\log q = O(\log t)
    \qquad \text{and} \qquad
    \log|S| = O(\log^2t) .
\]
\end{lemma}

\begin{proof}
Write $\hat\ell(S)$ for the smallest degree of a nontrivial irreducible projective representation of $S$ over $\C$, so that $\hat\ell(S)\leq\ell(S)$. We have $\hat\ell(S)\leq t$. Seitz and Zalesskii \cite[Table~I]{SZ93}, building on Landazuri and Seitz \cite{LS74}, give a lower bound for $\hat\ell(S)$ in terms of $q$ and the type of $S$ apart from a finite list of exceptions, which we absorb into the constants here. Inspecting the entries of the table, we find an absolute constant $c>0$ such that
\[
    \hat\ell(S)\geq cq \quad \text{if } r=1
    \qquad \text{and} \qquad
    \hat\ell(S)\geq cq^{r-1} \quad \text{if } r\geq2 .
\]
For $r=1$ this is the bound $(q-1)/2$ for $\PSL_2(q)$ with $q\neq9$. Since $\hat\ell(S)\leq t$, this gives $q=O(t)$ if $r=1$, and hence $r\log q=O(\log t)$. If $r\geq2$, it gives $q^{r-1}=O(t)$. Since $q\geq2$, it follows that $2^{r-1}=O(t)$, that is, $r\leq\log_2t+O(1)$. Moreover $r\leq2(r-1)$, and so $r\log q\leq2\log\left(q^{r-1}\right)=O(\log t)$. Finally, $|S|\leq q^{O(r^2)}$, and hence $\log|S|=O\left(r\cdot(r\log q)\right)=O(\log^2t)$.
\end{proof}

The last bound is what sets the groups of Lie type apart from the alternating groups. For $t\geq4$, the group $\Alt(t+1)$ lies in $\mathcal S(t)$ and has order $\exp(\Theta(t\log t))$. We return to this in \Cref{sec:altfree}.

\begin{lemma}\label{lem:simplecount}
For all $t\geq2$ we have $|\mathcal{S}(t)|=O(t)$. Moreover, the Schur multiplier $M(S)=H^2(S,\C^\times)$ of every $S\in\mathcal S(t)$ has order $O(\log t)$.
\end{lemma}

\begin{proof}
Let $S\in\mathcal{S}(t)$, say $S=E/Z(E)$ with $E\leq\GL_t(\C)$ quasisimple and irreducible. If $S=\Alt(n)$ with $n\geq5$ then $n!/2=|S|\leq|E|\leq\exp(O(t\log t))$ by \Cref{lem:basic}, so that $n=O(t)$ and there are $O(t)$ such $S$. The sporadic groups contribute $O(1)$. Otherwise $S$ is of Lie type over $\F_q$ of untwisted rank $r$, and \Cref{lem:liedata} gives $q=O(t)$ when $r=1$ and $q^{r-1}\leq c_0t$ when $r\geq2$, for some constant $c_0$, together with $r\leq\log_2t+O(1)$. Such an $S$ is determined by its family, its rank and $q$, and there are boundedly many families of each rank, so the number of these is at most a constant times
\[
    O(t) + \sum_{r=2}^{O(\log t)}\#\left\{q \ \middle| \ q^{r-1}\leq c_0t\right\}
    \leq O(t)+\sum_{r=2}^{O(\log t)}(c_0t)^{1/(r-1)}
    = O(t) .
\]
Here the first term counts the groups of rank one, while the sum is dominated by the term counting groups of rank two.

For the multipliers, we use \cite[Th.~5.1.4]{KL90}. The multiplier of an alternating group has order at most $6$, and that of a sporadic group is $O(1)$. For a group of Lie type of untwisted rank $r$, the multiplier has order at most $4(r+1)$, apart from a finite list of exceptions of order at most $48$. So $|M(S)|=O(r)=O(\log t)$ by \Cref{lem:liedata}.
\end{proof}

\subsubsection{Counting through the covering groups}

We now count the irreducible quasisimple subgroups $E\leq\GL_t(\C)$ up to conjugacy. The idea is to push each of them into the universal covering group of its simple quotient, where it becomes a representation of a single group and can be counted there. For a finite simple group $S$, we write $\widehat S$ for its universal covering group, and $\Irr_t(\widehat S)$ for the set of irreducible representations of $\widehat S$ of degree $t$ up to equivalence.

\begin{lemma}\label{lem:inflate}
For all $t\geq2$ we have
\[
    \qs(t) \leq \sum_{S\in\mathcal{S}(t)}|\Irr_t(\widehat S)| .
\]
\end{lemma}

\begin{proof}
Let $E\leq\GL_t(\C)$ be quasisimple and irreducible, and put $S=E/Z(E) \in \mathcal S(t)$. Since $E$ is a perfect central extension of $S$, it is a quotient of $\widehat S$. Composing the quotient map with the inclusion $E\leq\GL_t(\C)$ gives an irreducible representation of $\widehat S$ of degree $t$ whose image is $E$. Equivalent representations have conjugate images, so the conjugacy class of $E$ is determined by $S$ and the equivalence class of this representation. Summing over $S$ gives the claim.
\end{proof}

It now remains to bound $|\Irr_t(\widehat S)|$ for each $S\in\mathcal S(t)$. A first bound is $|\Irr_t(\widehat S)|\leq|\Irr(\widehat S)|$, the number of conjugacy classes of $\widehat S$. For the groups of Lie type this bound is good enough, as the proof below shows. For the alternating groups it is not. Indeed, for $t\geq4$ the group $\Alt(t+1)$ lies in $\mathcal S(t)$ through its standard representation of degree $t$. On the other hand, $|\Irr(\Alt(t+1))|$ is roughly half the number of partitions of $t+1$, so it grows like $\exp(\Theta(\sqrt t))$, which is far from polynomial in $t$. For the alternating groups we therefore need a count of the representations of small degree alone, which we extract from Guralnick, Larsen and Tiep \cite{GLT12}.

\begin{lemma}\label{lem:repcount}
Let $t\geq2$ and let $S\in\mathcal{S}(t)$. Then $|\Irr_t(\widehat S)|=t^{O(1)}$.
\end{lemma}

\begin{proof}
Suppose first that $S=\Alt(n)$ with $n\geq5$. By \cite[Th.~5.3]{GLT12}, the universal covering group $\widehat S$ has $|\Irr_t(\widehat S)|\leq t^{2.5}$. In all other cases we use the bound $|\Irr_t(\widehat S)|\leq|\Irr(\widehat S)|$. Every conjugacy class of $\widehat S$ lies over a class of $S$, and the fibre over one class consists of at most $|Z(\widehat S)|=|M(S)|$ classes. Hence
\[
    |\Irr_t(\widehat S)| \leq |\Irr(\widehat S)| \leq |M(S)|\cdot|\Irr(S)| .
\]
For $S$ sporadic, both factors are $O(1)$. For $S$ of Lie type of untwisted rank $r$ over $\F_q$, we have $|M(S)|=O(\log t)$ by \Cref{lem:simplecount}. Moreover, $|\Irr(S)|\leq15.2\cdot q^r$ by Fulman and Guralnick \cite[Cor.~1.2]{FG12}, and
\[
    q^r = \exp(r\log q) = \exp\left(O(\log t)\right) = t^{O(1)}
\]
by \Cref{lem:liedata}. So $|\Irr_t(\widehat S)|=t^{O(1)}$ in every case.
\end{proof}

By combining the preceding three lemmas, we obtain the desired bound on $\qs(t)$.

\begin{proposition}\label{prop:lambda}
There is an absolute constant $C$ such that $\qs(t)\leq t^C$ for all $t\geq2$. Consequently $\QS(m)\leq m^C$ for all $m\geq2$.
\end{proposition}

% \begin{proof}
% By \Cref{lem:inflate,lem:simplecount,lem:repcount}, we have $\qs(t)\leq|\mathcal S(t)|\cdot t^{O(1)}=O(t)\cdot t^{O(1)}$, which is at most $t^C$ for a suitable constant $C$. For integers $m_j\geq2$ with $\prod_jm_j\leq m$, it follows that $\prod_j\qs(m_j)\leq\left(\prod_jm_j\right)^C\leq m^C$.
% \end{proof}

\subsection{Unrolling the recursion}\label{sub:proofcount}

We now unroll the recursion of \Cref{lem:recursion}, using \Cref{prop:lambda} at each level.

\begin{proof}[Proof of \Cref{thm:count}]
Fix $m$ and unroll \Cref{lem:recursion} down to $\qp(1)=1$, choosing at each step a divisor that realizes the maximum. This produces a chain $m=w_0>w_1>\dots>w_h=1$ in which each $w_{i+1}$ divides $w_i$, and is therefore at most $w_i/2$, so $h\leq\log_2m$. Every $w_i$ divides $m$, so $\mu(w_i)\leq\mu(m)$ and $\log w_i\leq\log m$, and hence $\nu(w_i)\leq\nu(m)$. By \Cref{prop:lambda}, the $h$ steps give
\[
    \log\qp(m) \leq \sum_{i=0}^{h-1}\left(C\log w_i+O\left(\nu(w_i)\log w_i\right)\right)
    = O\left(\nu(m)\log^2m\right) .
\]
Unwinding $\nu$, this shows that the quasiprimitive subgroups of $\PGL_m(\C)$ form at most $\exp\left(O\left((\mu(m)+\log\log m)\log^3m\right)\right)$ conjugacy classes. Summing over the degrees $m'\leq m$ costs a factor $m$ and replaces $\nu(m')$ by $\max_{m'\leq m}\nu(m')=O(\log^2m)$, which gives $\log \QP(m)=O(\log^4m)$.
\end{proof}

With the count of quasiprimitive subgroups in hand, we can return to \Cref{prop:reduction}, and \Cref{thm:main} follows by absorbing the factor $\QP(k)=\exp(O(\log^4k))$ into the error term.

\begin{remark}\label{rem:truth}
We do not expect \Cref{thm:count} to be sharp. Its four powers of $\log m$ all come from the recursion, one from the number of levels, one from the generation bound of \Cref{lem:gen}, and two from $\Aut(A)$ in \Cref{lem:norm}. The bound on $|\mathcal N/\bar R|$ behind the last two is sharp in order of magnitude. For $m=2^n$ with a single quasicomponent $2^{1+2n}$, the quotient $\mathcal N/\bar R\cong\Sp_{2n}(2)$ has order $\exp(\Omega(\log^2 m))$. So any improvement has to count the relevant subgroups of $\Sp_{2n_i}(p_i)$ directly rather than through their generators. 
% The components are not the obstacle. By Liebeck and Shalev \cite[Cor.~1.2]{LS07}, for almost all $t$ the only irreducible representation of degree $t$ of a quasisimple group is the standard representation of $\Alt(t+1)$, so $\qs(t)=1$ for almost all $t$. 
% In the other direction, the following example shows that $\log\qp(m)=\Omega(\log^2m/\log\log m)$ for infinitely many $m$.
\end{remark}

\subsection{Comparison with permutation groups and finite fields}\label{sub:permutation}

We now compare our argument with that of Jaikin-Zapirain and Pyber \cite{JZP11}. Their count comes from permutation groups. Pyber and Shalev \cite{PS97} counted the primitive subgroups of $\Sym(n)$ up to conjugacy, and the expensive ones turn out to be the affine groups $\F_p^b\rtimes\Gamma$ with $n=p^b$ and $\Gamma\leq\GL_b(\F_p)$ irreducible. This turns the problem into one about linear groups over $\F_p$. Jaikin-Zapirain and Pyber reduce it further to primitive linear groups, and show that the $d$-generated primitive subgroups of $\GL_b(\F_p)$ fall into $n^{O(d)}$ conjugacy classes \cite[Prop.~5.9]{JZP11}.

This bound is measured in the size $n=p^b$ of the space rather than in the dimension $b$, and this is the natural scale for the problem. By \cite[Prop.~1.7]{PS97}, already for $p=3$ there are $\exp(\Omega(b^2/\log b))$ irreducible subgroups of $\GL_b(\F_p)$ up to conjugacy, for infinitely many $b$. Measured in $b$, the logarithm of this count is far from polylogarithmic. Measured in $n$, it is $\Omega(\log^2n/\log\log n)$, which is polylogarithmic in $n$, just as the logarithm of our bound is polylogarithmic in $m$. Because $n$ is exponential in the dimension, a power of $n$ is also a generous target, and this is what the differences below come down to. Over $\C$ the situation is different. By Jordan's theorem \factref{J}, a quasiprimitive subgroup of $\PGL_m(\C)$ has order at most $J(m)$, so there are finitely many of them in each degree, and it makes sense to bound their number in terms of $m$ alone.

To compare the two counts, we transcribe their setting into our notation. Let $\Gamma\leq\GL(V)$ be a primitive subgroup with $V=\F_p^b$, so that $|V|=n$, and let $R=F^*(\Gamma)$ be its generalized Fitting subgroup. (Jaikin-Zapirain and Pyber write $P$ for $\Gamma$ and $W$ for $V$.) We suppress two technical points. The tensor products below are taken over the field $F=Z(\End_R(V))$ rather than over $\F_p$, and $\Gamma$ may act on them $F$-semilinearly. Neither affects the comparison.

Both counts record the group through the same four pieces of data, and the main difference lies in how the tensor factors are found. Each step of their induction splits off one tensor factor, and the structure of $R$ is used only at the end. Each step of our recursion removes the whole $R$-part at once, and recurses only on the multiplicity space, which $R$ does not see. We now go through the pieces and explain which steps carry over and which do not. \Cref{tab:steps} summarizes this.

\begin{table}[ht]
\centering
\small
\begin{tabular}{@{}llll@{}}
\toprule
 & over $\F_p$ & over $\C$ & transfers \\
\midrule
structure of $R$ & \cite[Lem.~5.2]{JZP11}, at the end & \factref{C}, at the start & yes \\
(1) the degrees & one split per step & all at once & changed \\
(2) the group $\bar R$ & crude, powers of $n$ & \Cref{prop:lambda} & no \\
(3) the group $Q_2$ & induct on larger factor & recurse on $W$ & changed \\
(4) the group $Q$ & generators in $N/R$ & generators in $\Lambda$ & two adjustments \\
\bottomrule
\end{tabular}
\bigskip
\caption{The count over $\F_p$ following \cite{JZP11}, and over $\C$.}
\label{tab:steps}
\end{table}

\emph{Pieces (1) and (3), the degrees and the group $Q_2$.} These are the pieces that handle the tensor decompositions, and this is where the two arguments differ most. The structure is the same in both settings. Over $\F_p$ too, $V$ is the tensor product of an irreducible $R$-module and a multiplicity space \cite[Lem.~5.1]{JZP11}, and the irreducible module splits over the factors of $R$. What differs is how the count uses this.

Over $\F_p$, the tensor decompositions are handled by induction on the dimension, in Family~2 of the proof of \cite[Prop.~5.9]{JZP11}. Suppose that $|\Gamma|$ is at most a fixed power of $n$ (the few larger primitive groups are identified by name \cite[Prop.~5.7]{JZP11}), and that $\Gamma$ preserves a decomposition $V=V_1\otimes V_2$ with $2\leq\dim V_1\leq\dim V_2$. Then $\Gamma$ lies in a tensor product $X\otimes Y$ of primitive groups $X\leq\GL(V_1)$ and $Y\leq\GL(V_2)$. They count $Y$ by induction and $X$ crudely, and then locate $\Gamma$ inside $X\otimes Y$ through $d$ generators, which is cheap because $X\otimes Y$ has order $n^{O(1)}$. So piece (1) is the choice of $\dim V_1$ and $\dim V_2$, and piece (3) is the inductive count of $Y$. The structure lemma \cite[Lem.~5.2]{JZP11} is needed only at the bottom, for a group that preserves no tensor decomposition at all, which is their Family~3.

Over $\C$, the same induction fails at the last step. Our target $\exp(O(\log^4m))$ only pays for choosing $O(\log m)$ elements from a group of order $\exp(O(\log^2m))$, as in \Cref{lem:group}. But $X\otimes Y$ can be far larger, for instance when $X$ contains $\Alt(a+1)$ acting through its standard representation of degree $a=\dim V_1$, which has order $\exp(\Theta(a\log a))$. So we locate $Q$ only modulo its large normal subgroup $\bar R$, above which everything is small by \Cref{lem:norm}. This requires knowing $\bar R$ in full, and then all of its tensor factors come with it. That is why \Cref{lem:shape} takes all tensor factors of $U$ at once, and why piece (1) chooses all of their degrees at once, in \Cref{lem:degrees}. Piece (2) then records $\bar R$, and piece (4) records $Q$ inside the small group $\Lambda$.

The only thing left to recurse on is the multiplicity space $W$. Its group $Q_2$ is a quotient of $Q/\bar R$, and it is quasiprimitive of degree $w\leq m/2$ by \Cref{lem:mult}. So the recursion runs through $Q_2$ alone.

\emph{Piece (2).} This is the piece that does not transfer, again because of the targets. Over $\F_p$, once the large groups of \cite[Prop.~5.7]{JZP11} are set aside, the factors of $R$ can be counted crudely through their orders, within $n^{O(1)}$. Over $\C$, a quasisimple linear group of degree $t$ can have order $\exp(\Theta(t\log t))$, so a count through orders is far too expensive, and we go through the covering groups instead, as in \Cref{sub:components}.
%\footnote{A count of the same kind appears for maximal subgroups over finite fields. Guralnick, Larsen and Tiep \cite[Th.~6.3]{GLT12}, sharpening Liebeck, Martin and Shalev \cite[Th.~1.3]{LMS05}, bound the number of conjugacy classes of almost simple irreducible maximal subgroups of a classical group of rank $r$ over $\F_q$ by $O(r^6+r\log\log q)$. Their treatment of the alternating groups goes through the projective representations of $\Alt(n)$ of small degree, just as \Cref{lem:repcount} does.} 
The extraspecial factors cause no such problem, and cost $m^2$ in \Cref{lem:factors}.

\emph{Piece (4).} Over $\F_p$, once $R$ is fixed up to conjugacy, the quotient $\Gamma/R$ is a $d$-generated subgroup of $N_{\GL(V)}(R)/R$, a group of order $n^{O(1)}$. Over $\C$, the group $\Lambda$ of \Cref{lem:group} takes its place, with two adjustments. First, $\mathcal N$ has to be shown to be finite, which is automatic over a finite field, and this is \Cref{lem:norm}. Second, the generation bound of \cite{LMM00} used in \cite{JZP11} is for permutation groups, and we replace it by the bound $O(\log m)$ for primitive subgroups of $\GL_m(\C)$ \cite[Th.~B]{LMM01}, which is \Cref{lem:gen}.

\section{Groups without large alternating sections}\label{sec:altfree}

The prefactor in \Cref{thm:main} is $\exp((d-1)k\log k+O(dk))$, and its factor $\log k$ has a clear source. In the proof of \Cref{prop:reduction}, everything costs $\exp(O(dk))$ except two factorials, Jordan's bound $|Q|\leq J(m)$ and the count $j(j!)^{d-1}$ of subgroups of index $j$. Both bound the image of a homomorphism, onto $Q$ in the first case and into $\Sym(j)$ through the action on cosets in the second, and in both cases the worst image is a symmetric group. These images are sections of $G$. 
So if $G$ has no large alternating sections, we can hope to replace both factorials by exponentials, and thereby remove the logarithm.
This is what we do in this section. For the subgroups, we rely on the following input of Lubotzky and Segal.

\begin{itemize}
\fact{G} \emph{Exponential subgroup growth.} Let $c\geq4$ and $d\geq2$, and let $\mathcal{A}_c$ be the class of finite groups with no section isomorphic to $\Alt(c+1)$. By \cite[Th.~3.3, Cor.~3.4]{LS03}, the free pro-$\mathcal{A}_c$ group on $d$ generators has at most $2^{(\sigma+o(1))n}$ open subgroups of index at most $n$, where $\sigma=(d-1)\log_2(c!)/(c-1)$. Every $d$-generated group in $\mathcal{A}_c$ is a quotient of it, and so satisfies the same bound.
\end{itemize}

For the quasiprimitive quotients, we prove the counterpart in \Cref{sub:altfree}. Combining the structure theory of \Cref{sub:structure} with the linear version of the theorem of Babai, Cameron and P\'alfy, we show in \Cref{prop:qaltfree} that $|Q|\leq m^{O_c(1)}$. With these two bounds in place, the proof of \Cref{prop:reduction} goes through as before and gives \Cref{thm:altfree}. Finally, \Cref{sub:altfreegap} shows that \Cref{thm:altfree} is sharp, with groups in this class whose gap function grows exponentially, and \Cref{sub:monomialsharp} uses the same construction to show that \Cref{thm:monomial} is sharp.

\subsection{Removing the logarithm}\label{sub:altfree}

The subgroups are taken care of by \factref{G}, so it remains to bound the quasiprimitive quotients. The bound we need is the linear counterpart of a theorem of Babai, Cameron and P\'alfy \cite{BCP82}, see also \cite[Cor.~IV]{PS97}, which says that a primitive permutation group without sections isomorphic to $\Alt(c+1)$ has order polynomial in its degree.

\begin{proposition}\label{prop:qaltfree}
Let $c\geq4$ and let $Q\leq\PGL_m(\C)$ be a quasiprimitive subgroup that is $\Alt(c+1)$-free. Then $|Q|\leq m^{O_c(1)}$.
\end{proposition}

\begin{proof}
Let $\Gamma\leq\GL_m(\C)$ be a finite primitive subgroup with image $Q$. For $m=1$ the group $Q$ is trivial, so we may assume that $m\geq2$, and then $\Gamma$ is as in \Cref{subsub:data}. Every section of $Q$ is $\Alt(c+1)$-free, and hence lies in the class $\mathcal B_{c'}$ of \cite[p.~352, Prop.~2]{LS03} for some $c'$ depending only on $c$. Recall from the proof of \Cref{lem:norm} that $|Q/\bar R|\leq|C_Q(\bar R)|\cdot|Q^{\Out}|$, where $Q^{\Out}$ is the image of $Q$ in $\Out(\bar R)=\Aut(A)\times\Out(\bar E)$, and that $|C_Q(\bar R)|\leq|A|\leq m^2$. We bound $|\bar R|$ and the two projections of $Q^{\Out}$ to $\Aut(A)$ and $\Out(\bar E)$.

\emph{The components.} The simple quotients $S_j=E_j/Z(E_j)$ of the components of $\Gamma$ are sections of $Q$, and so they lie in $\mathcal B_{c'}$. By the definition of this class \cite[p.~352]{LS03}, and since a classical group of untwisted rank $r$ has degree at least $r+1$, each $S_j$ is sporadic, alternating of degree at most $c$, or of Lie type of untwisted rank $r=O_c(1)$. In the first two cases $|S_j|=O_c(1)$, which is at most $m_j^{O_c(1)}$ since $m_j\geq2$. In the last case, $E_j$ acts faithfully and irreducibly on a tensor factor of degree $m_j$ by \Cref{lem:shape}(1), so $S_j\in\mathcal S(m_j)$, and \Cref{lem:liedata} gives $r\log q=O(\log m_j)$. Hence $|S_j|\leq q^{O(r^2)}=m_j^{O(r)}=m_j^{O_c(1)}$. Since $\prod_jm_j\leq m$ and $|A|\leq m^2$, we get $|\bar R|=|A|\cdot\prod_j|S_j|\leq m^{O_c(1)}$.

\emph{The action on $\bar E$.} The image of $Q^{\Out}$ in $\Out(\bar E)$ maps to $\Sym(s)$, permuting the factors $S_j$, and the kernel has order at most $\prod_j\lvert\Out(S_j)\rvert=m^{O(1)}$ as in the proof of \Cref{lem:norm}. The image in $\Sym(s)$ is an $\Alt(c+1)$-free permutation group of degree $s\leq\log_2m$, so its order is at most $(c!)^{(s-1)/(c-1)}\leq m^{\log_2(c!)}$ by \cite[Th.~16.4.8]{LS03}.

\emph{The action on $A$.} Write $A_i=\F_{p_i}^{2n_i}$ for the image of $P_i$ in $\bar R$, so that $A=\prod_iA_i$, and let $H_i$ be the image of $Q$ in $\GL(A_i)$. The group $Z(P_i)$ is an abelian normal subgroup of $\Gamma$, hence central, so conjugation by $\Gamma$ preserves the nondegenerate alternating form on $A_i$ given by commutators in $P_i$. No nonzero $H_i$-invariant subspace of $A_i$ is totally isotropic, since its preimage in $P_i$ would be an abelian normal subgroup of $\Gamma$ that is not central, contradicting primitivity. This makes $A_i$ a completely reducible $H_i$-module. Indeed, if $U\leq A_i$ is a minimal invariant subspace, then $U\cap U^\perp$ is invariant and totally isotropic, hence zero. So $A_i=U\oplus U^\perp$ with both summands invariant, the form is nondegenerate on $U^\perp$, and we may repeat the argument inside $U^\perp$.
Thus $H_i$ is a completely reducible subgroup of $\GL_{2n_i}(p_i)$ in the class $\mathcal B_{c'}$, and the theorem of Babai, Cameron and P\'alfy \cite{BCP82}, in the form \cite[Th.~16.4.5]{LS03}, gives $|H_i|\leq|A_i|^{O_c(1)}$. Since the primes $p_i$ are distinct, $\Aut(A)=\prod_i\GL(A_i)$, so the image of $Q^{\Out}$ in $\Aut(A)$ is a subgroup of $\prod_iH_i$, and its order is at most $|A|^{O_c(1)}\leq m^{O_c(1)}$.

Multiplying the four bounds gives $|Q|=|Q/\bar R|\cdot|\bar R|\leq m^{O_c(1)}$.
\end{proof}

We now combine \factref{G} with \Cref{prop:qaltfree} to get rid of the logarithm in the class of $\Alt(c+1)$-free groups.

\begin{proof}[Proof of \Cref{thm:altfree}]
We run the proof of \Cref{prop:reduction} with two changes. First, since $G$ is $\Alt(c+1)$-free and $d$-generated, \factref{G} bounds the number $\sum_ja_j\leq s_k(G)$ of subgroups of index at most $k$ by $\exp(O_{c,d}(k))$, in place of $a_j\leq j(j!)^{d-1}$. Second, every $Q$ that arises is a section of $G$, hence $\Alt(c+1)$-free, so $|Q|\leq m^{O_c(1)}$ by \Cref{prop:qaltfree}. After the division by $|Q|$ from \Cref{lem:orbits}, the surjection $\varphi$ then costs
\[
    |Q|^{d(H)-1} \leq \exp\left(O_c(\log m)\cdot O\left(k(d-1)/m\right)\right) = \exp\left(O_{c,d}(k)\right).
\]
The remaining factors, namely $\QP(k)$, the at most $k^2$ pairs $(j,m)$ and $\ab_j(G)\leq\ab_k(G)$, are as in \Cref{prop:reduction}, and multiplying gives the claim.
\end{proof}

\subsection{The gap function in this class}\label{sub:altfreegap}

Write $\mathcal{A}_{c,d}$ for the class of finite $d$-generated $\Alt(c+1)$-free groups, so that \Cref{thm:altfree} reads $\Delta_{\mathcal{A}_{c,d}}(k)\leq\exp(O_{c,d}(k))$. We now show that this bound is attained for infinitely many $k$, up to the constant in the exponent.

\subsubsection{Powers of simple groups}

In \Cref{ex:altpower} we saw that the general upper bound for the gap function of the class of \emph{all} finite $d$-generated groups was attained by a large power of $\Alt(k+1)$. 
The small representations of a power $S^t$ of a simple group are those that are nontrivial in a single coordinate, and its abelian sections of small index are trivial, by \Cref{prop:powers}(1). Taking $t=t_d(S)$ as large as $d$-generation allows, the gap is thus about $|S|^{d-1}$. For $S=\Alt(k+1)$ this is $(k!)^{d-1}$ up to a factor $\exp(O(dk))$.

The same construction is much weaker in $\mathcal A_{c,d}$, because an $\Alt(c+1)$-free simple group is small compared to the degrees of its representations. The argument for the components in the proof of \Cref{prop:qaltfree} shows that such a group has $|S|=\ell(S)^{O_c(1)}$. So powers of simple groups in $\mathcal A_{c,d}$ only give polynomial lower bounds for the gap function. 
Here is a concrete example. 

\begin{example}\label{ex:psl2power}
Let $p>11$ be a prime with $p\equiv\pm2\pmod5$, and put $S=\PSL_2(p)$ and $k=p-1$. The order of $S$ is prime to $5$. By the character table \cite[Table~5.4]{Bon11}, we have $\ell(S)\in\{(p-1)/2,(p+1)/2\}$, and $S$ has two irreducible representations of degree $\ell(S)$ and $\lfloor(p-1)/4\rfloor$ of degree $p-1$, so at least $k/4$ nontrivial ones of degree at most $k$. By Galois's theorem \cite[Satz~II.8.28]{Hup67}, we have $\iota(S)=p+1$. By \Cref{prop:powers}, the power $S^{t_d(S)}$ has $\ab_k=1$ and $\Rep_k\geq t_d(S)\cdot k/4$. With $\lvert\Out(S)\rvert=2$ and $|S|\geq k^3/2$, this gives
\[
    \Delta_{\mathcal{A}_{4,d}}(k) = \Omega_d\left(k^{3d-2}\right)
    \quad \text{for all } k=p-1 \text{ with } p \text{ as above} .
\]
We return to these groups in \Cref{sub:notsufficient}.
\end{example}

The right way to find extremal examples is to look back at the proof of \Cref{thm:altfree}. There, a representation induced from a primitive one of degree $m$ on a subgroup of index $j$ costs $\exp(O_{c,d}(j))$ for the subgroup by \factref{G}, and $\exp(O_{c,d}(k\log m/m))$ for the surjection onto $Q$. Both costs are largest when $m$ is bounded and $j$ is about $k$. An extremal example should thus have its small representations induced from subgroups of index about $k$, with a primitive part of bounded degree. In light of this, we keep the idea of the construction above (many different quotients each carrying a representation of degree $k$), and change the group that is repeated. In place of a simple group we take a group $G$ whose representation of degree $k$ is \emph{imprimitive}, induced from a subgroup of index about $k$ with a primitive part of bounded degree. Such a group can be exponentially large in $k$, in comparison with the polynomial size of primitive groups by \Cref{prop:qaltfree}.
The model is a wreath product $S\wr T$ with $S$ fixed and $T$ transitive of degree $r$, which we construct in \Cref{subsub:wreath}. We first extend \Cref{prop:powers} to arbitrary groups $G$, and then apply it to these wreath products.

\subsubsection{Free subdirect powers}\label{subsub:freepowers}

Let $G$ be an arbitrary finite group with $d \geq d(G)$. In order to extend \Cref{prop:powers} to $G$,
the abelianization growth of the power must stay small, and in place of the condition $k<\iota(S)$ we use that a group without abelian composition factors has $\ab_k\leq k^{\alpha}$ for an absolute constant $\alpha$ \cite[Th.~3]{Sab23}. Apart from this, the group should have many different quotients isomorphic to $G$. 

Let us consider the freest construction of this sort. Start with a free group $F$ on $d$ generators. Write $\Epi(F,G)$ for the set of surjections $F\twoheadrightarrow G$, and $\mathcal K$ for the set of their kernels, namely the normal subgroups $K\trianglelefteq F$ with $F/K\cong G$. Each $K\in\mathcal K$ is the kernel of exactly $\lvert\Aut(G)\rvert$ surjections, so $|\mathcal K| = \lvert\Epi(F,G)\rvert / \lvert\Aut(G)\rvert$.
The \emph{free subdirect power} of $G$ is the group\footnote{In Hall's paper \cite[Sec.~1.6]{Hal36}, this is the group $F/G(F)$.} 
\[
\Pi_d(G)= \frac{F}{\bigcap_{K\in\mathcal K} K}. 
\]
It is a subdirect power of $G$, every $d$-generated subdirect power of $G$ is a quotient of it, and every $K\in\mathcal K$ gives a different quotient of it isomorphic to $G$. For a nonabelian simple group $S$, the quotient maps embed $\Pi_d(S)$ in $S^{\mathcal K}$, and the images of the generators of $F$ form a $d$-tuple whose coordinate tuples lie in pairwise distinct $\Aut(S)$-orbits, since their kernels are distinct. So $\Pi_d(S)=S^{\mathcal K}$ by \Cref{thm:hall}, and since $|\mathcal K|=t_d(S)$, this is precisely the power $S^{t_d(S)}$ of \Cref{prop:powers}.

\begin{proposition}\label{prop:freepower}
Let $G$ be a finite group without abelian composition factors, let $d\geq d(G)$, and suppose that $G$ has a faithful irreducible representation of degree $k$. 
Let $\alpha$ be the absolute constant of \cite[Th.~3]{Sab23}.
Then $\Pi_d(G)$ is a finite $d$-generated group, every simple section of $\Pi_d(G)$ is a section of $G$, and
\[
    \frac{\Rep_k(\Pi_d(G))}{\ab_k(\Pi_d(G))} \geq \frac{\lvert\Epi(F,G)\rvert}{k^{\alpha}\,\lvert\Aut(G)\rvert}.
\]
\end{proposition}

\begin{proof}
Put $\Pi=\Pi_d(G)$. The set $\mathcal K=\{K_1,\dots,K_s\}$ is finite, and $\Pi$ embeds in $\prod_iF/K_i\cong G^s$, so $\Pi$ is finite. Put $N_j=(K_1\cap\dots\cap K_j)/\bigcap\mathcal K$. Then
\[
    N_{j-1}/N_j \cong (K_1\cap\dots\cap K_{j-1})K_j/K_j ,
\]
which is a normal subgroup of $F/K_j\cong G$. So $\Pi$ has a normal series whose factors are normal subgroups of $G$, and every composition factor of $\Pi$ is a composition factor of $G$. In particular, $\Pi$ has no abelian composition factors, and $\ab_k(\Pi)\leq k^\alpha$ by \cite[Th.~3]{Sab23}. Intersecting a composition series of $\Pi$ with a subgroup shows that every simple section of $\Pi$ is a section of a composition factor of $\Pi$, and hence a section of $G$.

For the representations, fix $K\in\mathcal K$ and compose $\Pi\to F/K\cong G$ with the faithful irreducible representation of $G$. This gives an irreducible representation of $\Pi$ of degree $k$ with kernel $K/\bigcap\mathcal K$. Different $K$ give different kernels, and hence nonisomorphic representations, so $\Rep_k(\Pi)\geq|\mathcal K|$.
\end{proof}

Note that $\lvert\Epi(F,G)\rvert=P_d(G)|G|^d$, where $P_d(G)$ is the probability that $d$ random elements generate $G$. If $Z(G) = 1$, then $\lvert\Aut(G)\rvert=|G|\cdot\lvert\Out(G)\rvert$. The bound is then
\begin{equation}\label{eq:freepower}
    \frac{\Rep_k(\Pi_d(G))}{\ab_k(\Pi_d(G))} \geq \frac{P_d(G)\,|G|^{d-1}}{k^{\alpha}\,\lvert\Out(G)\rvert} .
\end{equation}
So a building block $G$ gives an exponential gap as soon as it is exponentially large compared to the degree of its faithful representation, while $P_d(G)$ and $\Out(G)$ should not spoil it. Following the discussion above, we take a building block whose faithful representation is imprimitive.

\subsubsection{Iterated wreath products}\label{subsub:wreath}

Imprimitive linear groups are subgroups of wreath products. If $G\leq\GL_k(\C)$ permutes $r$ blocks of dimension $\ell=k/r$, then $G$ embeds in $\GL_\ell(\C)\wr\Sym(r)$. We therefore take $W=S\wr_X T$, where $S$ is a nonabelian simple group with a nontrivial irreducible representation $\theta$ of degree $\ell$, and $T\leq\Sym(X)$ is transitive of degree $r=|X|$. Then $W$ acts on the direct sum $V=\bigoplus_{y\in X}V_y$ of $r$ copies of $\C^\ell$, with the base group $S^X$ acting on $V_y$ through $\theta$ applied to the $y$-coordinate, and with $T$ permuting the copies as it permutes $X$. This representation of degree $k=\ell r$ is irreducible\footnote{Indeed, the $V_y$ are irreducible and pairwise non-isomorphic as modules for the base group, so every invariant subspace is a sum of some of the $V_y$, and $T$ permutes them transitively.} and faithful.\footnote{An element acting trivially maps every $V_y$ to itself, so its component in $T$ is trivial, and then its component in the base group is trivial since $\theta$ is faithful.} On the other hand,
\[
    |W| \geq |S|^r = \exp\left(\frac{\log|S|}{\ell}\cdot k\right)
\]
is exponential in $k$ when $\ell, S$ are fixed. The base group is a power of a nonabelian simple group, so it adds no abelian composition factors.
It remains to choose the top group $T$. It has to be a transitive group of unbounded degree, $\Alt(c+1)$-free and without abelian composition factors, which are the same conditions that we impose on $W$. The simplest choice is therefore to iterate the construction, and take for $T$ the previous building block, and thus form an iterated wreath product.
%\footnote{The bound $(c!)^{(n-1)/(c-1)}$ on the order of an $\Alt(c+1)$-free subgroup of $\Sym(n)$, which sets the constant in \factref{G}, is attained for $n=c^h$ by the iterated wreath product $\Sym(c)\wr\dots\wr\Sym(c)$.}

We now make this precise. Fix a prime $p\geq7$ with $p\equiv\pm2\pmod5$, put $S=\PSL_2(p)$, and let $\ell=\ell(S)\leq(p+1)/2$. Then $5\nmid|S|$. Let $X_1=\mathbf P^1(\F_p)$ with its action of $S$, and put $W_1=S$. For $n\geq2$, let
\[
    W_n = S\wr_{X_{n-1}}W_{n-1} = B_n\rtimes W_{n-1} ,
    \qquad B_n = S^{X_{n-1}} ,
\]
be the permutational wreath product, acting on $X_n=X_1\times X_{n-1}$ by $(x,y)^{(b,w)}=(x^{b_y},y^w)$. This action is faithful and transitive (by induction on $n$). Write $r_n=|X_n|=(p+1)^n$. All composition factors of $W_n$ are isomorphic to $S$, and
\[
    |W_n| = |S|^{1+r_1+\dots+r_{n-1}} = |S|^{(r_n-1)/p} .
\]
Moreover, $W_n=S\wr_{X_{n-1}}W_{n-1}$ with $W_{n-1}$ transitive on $X_{n-1}$, so the representation constructed above, with $\theta$ of degree $\ell$, is a faithful irreducible representation of $W_n$ of degree $\ell r_{n-1}$.

The group $W_n$ is generated by two random elements with probability bounded away from $0$, uniformly in $n$. Bhattacharjee \cite{Bha94} proved this for iterated wreath products of alternating groups in their natural actions, and Quick \cite[Th.~A]{Qui06} extended it to iterated wreath products of arbitrary nonabelian simple groups with faithful transitive actions. For our groups, his theorem gives $P_2(W_n)\geq(1-\eta_p)P_2(S)$ for all $n$, where $\eta_p\to0$ as $p\to\infty$. Since also $P_2(S)\to1$, we may fix $p$ large enough so that
\[
    P_d(W_n) \geq P_2(W_n) \geq \frac12
    \quad \text{for all } n\geq1 \text{ and } d\geq2 .
\]
We assume this in \Cref{prop:altfreegap}, where it makes every $W_n$ $2$-generated. \Cref{lem:wreath} and \Cref{sub:notsufficient} hold for every prime $p\geq7$ with $p\equiv\pm2\pmod5$. What remains are two elementary facts.

\begin{lemma}\label{lem:wreath}
Let $n\geq2$.
\begin{enumerate}
    \item $B_n$ is the unique minimal normal subgroup of $W_n$, and $C_{W_n}(B_n)=1$.
    \item $\lvert\Out(W_n)\rvert\leq2^{(r_n-1)/p}(p+1)^{n^2}$.
\end{enumerate}
\end{lemma}

\begin{proof}
(1) The base group $B_n$ is a minimal normal subgroup, since a nontrivial normal subgroup of $W_n$ inside $B_n$ contains a coordinate subgroup,\footnote{Let $N\leq B_n$ be such a subgroup, and let $b\in N$ have a nontrivial coordinate $b_y$. Pick $s\in S$ not commuting with $b_y$, and let $c\in B_n$ be $s$ in the coordinate $y$ and trivial elsewhere. Then $[b,c]\in N$ is nontrivial and lies in the coordinate subgroup $S_y$ at $y$. So $N\cap S_y$ is a nontrivial normal subgroup of $S_y\cong S$, and hence $S_y\leq N$.} and $W_{n-1}$ permutes the coordinate subgroups transitively. An element $(b,w)$ centralizing $B_n$ fixes every coordinate, so $w=1$ since $W_{n-1}$ is faithful, and then $b\in Z(B_n)=1$. A second minimal normal subgroup would intersect $B_n$ trivially and hence centralize it, so there is none.

(2) Write $S_y$ for the coordinate subgroup of $B_n$ at $y\in X_{n-1}$, fix a point $y_0\in X_{n-1}$, and let $\alpha\in\Aut(W_n)$. By (1), $\alpha$ maps $B_n$ to itself, so it induces an automorphism $\bar\alpha$ of $W_n/B_n=W_{n-1}$. It also permutes the minimal normal subgroups of $B_n$, which are the coordinate subgroups, say $\alpha(S_y)=S_{\sigma(y)}$ for a permutation $\sigma$ of $X_{n-1}$. We claim that $\alpha$ is determined by three pieces of data, namely $\bar\alpha$, the point $\sigma(y_0)$, and the $r_{n-1}$ isomorphisms $\alpha|_{S_y}\colon S_y\to S_{\sigma(y)}$. This gives
\[
    \lvert\Aut(W_n)\rvert \leq \lvert\Aut(W_{n-1})\rvert\cdot r_{n-1}\cdot\lvert\Aut(S)\rvert^{r_{n-1}} .
\]
First, $\alpha$ is determined by its restriction to $B_n$, since $B_n$ is self-centralizing,\footnote{If $\alpha,\alpha'\in\Aut(W_n)$ agree on $B_n$, then $\beta=\alpha^{-1}\alpha'$ is trivial on $B_n$, so $wbw^{-1}=\beta(wbw^{-1})=\beta(w)b\beta(w)^{-1}$ for all $w\in W_n$ and $b\in B_n$. Thus $\beta(w)^{-1}w$ centralizes $B_n$, and hence $\beta(w)=w$.} and this restriction is determined by $\sigma$ and the isomorphisms $\alpha|_{S_y}$. Second, $\sigma$ is determined by $\bar\alpha$ and $\sigma(y_0)$. Indeed, an element $w\in W_n$ with image $\bar w\in W_{n-1}$ satisfies $w^{-1}S_yw=S_{y^{\bar w}}$. Applying $\alpha$ to this identity, and using that the image of $\alpha(w)$ in $W_{n-1}$ is $\bar\alpha(\bar w)$, gives $\sigma(y^{\bar w})=\sigma(y)^{\bar\alpha(\bar w)}$. Taking $y=y_0$, this determines $\sigma$ on all of $X_{n-1}$, since $W_{n-1}$ is transitive.

Finally, $W_n$ and $W_{n-1}$ have trivial centre, by (1) and since $S$ is simple, so $\lvert\Aut(W_n)\rvert=|W_n|\cdot\lvert\Out(W_n)\rvert$ and likewise for $W_{n-1}$. Since $|W_n|=|S|^{r_{n-1}}|W_{n-1}|$ and $\lvert\Aut(S)\rvert=2|S|$, the display gives $\lvert\Out(W_n)\rvert\leq2^{r_{n-1}}r_{n-1}\lvert\Out(W_{n-1})\rvert$. Unrolling this from $\lvert\Out(W_1)\rvert=2$ gives $\lvert\Out(W_n)\rvert\leq2^{1+r_1+\dots+r_{n-1}}(p+1)^{1+2+\dots+(n-1)}$, which is the claim, since $1+r_1+\dots+r_{n-1}=(r_n-1)/p$.
\end{proof}

\begin{proposition}\label{prop:altfreegap}
Let $c\geq4$ and $d\geq2$, and let $S=\PSL_2(p)$, $\ell=\ell(S)$, $r_n$ and $W_n$ be as above. Then
\[
    \Delta_{\mathcal{A}_{c,d}}(k) \geq \exp\left((1-o(1))\gamma k\right)
    \quad \text{with} \quad
    \gamma = \frac{2\left((d-1)\log|S|-\log2\right)}{p+1} > 0 ,
\]
for all $k$ of the form $\ell r_{n-1}$ with $n\geq2$.
\end{proposition}

\begin{proof}
Fix $n\geq2$ and put $k=\ell r_{n-1}$. We apply \Cref{prop:freepower} to $G=W_n$, which is $2$-generated and has a faithful irreducible representation of degree $k$. The order of $W_n$ is a power of $|S|$, and hence prime to $5$, so $W_n$ has no section isomorphic to $\Alt(5)$. By \Cref{prop:freepower}, neither does $\Pi_d(W_n)$, which therefore lies in $\mathcal A_{4,d}\subseteq\mathcal A_{c,d}$.
%\footnote{For these groups one can take $\alpha=1$ in \Cref{prop:freepower}, by an elementary argument. Dickson's list of the subgroups of $\PSL_2(p)$ \cite[Sec.~260]{Dic01} shows that $|U|\cdot|U:U'|\leq|S|$ for every $U\leq S$. Induction on $e$ extends this to $|U:U'|\leq|S^e:U|$ for every $U\leq S^e$, and induction along a chief series, as in the proof of \cite[Lemma~26]{Sab23}, gives $|H:H'|\leq|L:H|$ for every subgroup $H$ of a finite group $L$ whose composition factors are all isomorphic to $S$.} 
Since $Z(W_n)=1$ by \Cref{lem:wreath}(1), the bound \eqref{eq:freepower} applies. Using $P_d(W_n)\geq1/2$, the order of $W_n$ and \Cref{lem:wreath}(2), we get
\[
    \log\frac{\Rep_k(\Pi_d(W_n))}{\ab_k(\Pi_d(W_n))} \geq \left((d-1)\log|S|-\log2\right)\frac{r_n-1}{p}-n^2\log(p+1)-\log2-\alpha\log k .
\]
Finally $(r_n-1)/p\geq r_{n-1}=k/\ell\geq2k/(p+1)$, and $n=O(\log k)$ since $k=\ell(p+1)^{n-1}\geq2^{n-1}$, so the remaining terms are $O(\log^2k)=o(k)$.
\end{proof}

Together with \Cref{thm:altfree}, this shows that
\[
    \Delta_{\mathcal{A}_{c,d}}(k)=\exp\left(\Theta_{c,d}(k)\right)
\]
for infinitely many $k$.

\subsection{Monomial representations}\label{sub:monomialsharp}

The wreath product construction of \Cref{sub:altfreegap} also shows that the monomial count of \Cref{thm:monomial} is sharp, up to the constant in the exponent. For this we leave the $\Alt(c+1)$-free class, and take the top group to be an alternating group in its natural action. The representation of \Cref{subsub:wreath} is then monomial as soon as the representation $\theta$ of the base is. So we need a nonabelian simple group with a monomial irreducible representation of small degree, and $\Alt(5)$ has one of degree $5$.

\begin{example}\label{ex:monomialsharp}
Let $d\geq2$, let $m$ be large, and put $k=5m$. Let $\theta$ be the irreducible representation of $\Alt(5)$ of degree $5$. It is monomial, induced from a nontrivial linear character of $\Alt(4)$, so its space has a basis of $5$ lines permuted by $\Alt(5)$. Now let $W=\Alt(5)\wr_{[m]}\Alt(m)$, with $\Alt(m)$ acting naturally on $[m]=\{1,\dots,m\}$. By \Cref{subsub:wreath}, $\theta$ gives a faithful irreducible representation of $W$ of degree $k=5m$ on $m$ copies of the space of $\theta$. It is monomial, since the $5m$ lines in these copies are permuted by $W$. We claim that the free subdirect power $P=\Pi_d(W)$, which is a finite $d$-generated group, satisfies
\[
    \frac{\Rep_k^{\mon}(P)}{\ab_k(P)} \geq \exp\left(\left(\frac15-o(1)\right)(d-1)k\log k\right) .
\]

We use the bound \eqref{eq:freepower} for $G=W$, which has no abelian composition factors. It holds with $\Rep_k^{\mon}$ in place of $\Rep_k$, since the representations of $\Pi_d(W)$ that it counts are pulled back from the representation above along the quotient maps $\Pi_d(W)\to W$, and so they are monomial. By Quick's theorem \cite[Th.~A]{Qui06}, applied to $W$ with top group $\Alt(m)$ acting naturally on $[m]$ and base group $\Alt(5)$, and since $P_2(\Alt(m))\to1$, we have $P_d(W)\geq P_2(W)\geq1/2$ for all large $m$. In particular, $W$ is $d$-generated. The proof of \Cref{lem:wreath} also applies to $W$, since $\Alt(m)$ is transitive on $[m]$. It gives $Z(W)=1$ and
\[
    \lvert\Aut(W)\rvert \leq \lvert\Aut(\Alt(m))\rvert\cdot m\cdot\lvert\Aut(\Alt(5))\rvert^m = m!\cdot m\cdot120^m ,
\]
and since $|W|=60^m\cdot m!/2$, it follows that $\lvert\Out(W)\rvert\leq2m\cdot2^m$. So \eqref{eq:freepower} gives
\[
    \log\frac{\Rep_k^{\mon}(P)}{\ab_k(P)} \geq (d-1)\log|W|-O(m)-\alpha\log k = (d-1)m\log m-O(dm) .
\]
With $m=k/5$, this is $(1/5-o(1))(d-1)k\log k$, as claimed.
\end{example}

\section{Expanding sequences}\label{sec:reduction}

We now return to expanding sequences, which motivated this work through the questions of Wigderson and Sabatini recalled in the introduction. The previous sections bound the ratio $\Rep_k/\ab_k$ for arbitrary $d$-generated groups, and expansion plays no role in them. To bound $\Rep_k$ along expanding sequences, it therefore remains to bound $\ab_k$. This is a theorem of Lubotzky and Weiss, and it is the only property of expansion that we use in this section. Its proof is short and natural, so we include it in \Cref{sub:LW}, to make clear exactly what is used. In \Cref{sub:bounds} we deduce \Cref{cor:quasilinear} and \Cref{cor:sharp}, and show with wreath products that the resulting bounds are sharp up to the logarithm. Finally, \Cref{sub:notsufficient} shows that none of these growth conditions suffices for $d$-expansion.

We use the Cayley graphs, spectral gaps and classes $\mathcal{C}_{d,\epsilon}$ of \Cref{sub:expanding}. Let $s\in G$ act on the functions on $G$ by $(sf)(g)=f(gs)$, so that the normalized Laplacian of $\Cay(G,\Sigma)$ is $L=1-(2d)^{-1}\sum_i(s_i+s_i^{-1})$. Since each $s_i$ acts unitarily, expanding the squares gives $\langle Lf,f\rangle=(2d)^{-1}\sum_i\lVert f-s_if\rVert^2$. So the spectral gap is the minimum of the \emph{Rayleigh quotient}
\[
    \frac{\langle Lf,f\rangle}{\langle f,f\rangle} = \frac{1}{2d}\sum_{i=1}^d\frac{\lVert f-s_if\rVert^2}{\lVert f\rVert^2}
\]
over the nonzero functions $f$ on $G$ orthogonal to the constants, and every such $f$ bounds it from above. The \emph{expansion constant} of an expanding sequence is a uniform lower bound $\epsilon$ on the spectral gaps.

\subsection{Abelianization growth}\label{sub:LW}

Let $G$ be a finite group with a generating tuple $\Sigma$. The \emph{Kazhdan constant} of $(G,\Sigma)$ is
\[
    \kappa(G,\Sigma) = \inf_{\rho,\,v}\ \max_{s\in\Sigma}\|\rho(s)v-v\| ,
\]
where $\rho$ runs over the unitary representations of $G$ without nonzero invariant vectors, and $v$ over the unit vectors in the space of $\rho$. In words, every unit vector of such a representation is moved by at least $\kappa(G,\Sigma)$ by some generator. Write $\Sigma=(s_1,\dots,s_d)$, and let $\rho$ be a unitary representation of $G$ without nonzero invariant vectors. Then $\rho$ is a sum of nontrivial irreducible representations, each of which occurs in the functions on $G$ orthogonal to the constants. So every unit vector $v$ of $\rho$ has Rayleigh quotient
\[
    \frac{1}{2d}\sum_{l=1}^d \lVert\rho(s_l)v-v\rVert^2,
\]
and the spectral gap of $\Cay(G,\Sigma)$ is the infimum of these Rayleigh quotients over all such $\rho$ and $v$. In particular, if $\Cay(G,\Sigma)$ has spectral gap at least $\epsilon$, then $\kappa(G,\Sigma)\geq\sqrt{2\epsilon}\geq\epsilon$. So every member of $\mathcal C_{d,\epsilon}$ has a generating set of $d$ elements with Kazhdan constant at least $\epsilon$. Conversely, a lower bound on the Kazhdan constant gives a lower bound on the spectral gap, depending only on $d$ \cite[Th.~4.3.2]{Lub94}.

The same definition applies to a finitely generated infinite group $\Gamma$ with a finite generating tuple $\Sigma$, with $\rho$ running over the unitary representations of $\Gamma$ on Hilbert spaces without nonzero invariant vectors. The group $\Gamma$ has \emph{Kazhdan's property (T)} if $\kappa(\Gamma,\Sigma)>0$. A representation of a finite quotient $G$ of $\Gamma$ is also a representation of $\Gamma$, so the image of $\Sigma$ in $G$ has Kazhdan constant at least $\kappa(\Gamma,\Sigma)$. Hence the finite quotients of a group with property (T), with the images of a fixed generating tuple, form a family of expanders \cite[Prop.~3.3.1]{Lub94}. This observation of Margulis is the source of several expanding sequences in the later sections of this paper.

The theorem of Lubotzky and Weiss bounds the abelian sections of a finite group $G$ in terms of $\kappa(G,\Sigma)$. We include its short proof, slightly adapted from \cite[Props.~3.9, 3.10]{LW93}.

\begin{proposition}[{Lubotzky--Weiss \cite[Th.~3.6]{LW93}}]\label{prop:LW}
Let $G$ be a finite group generated by a set $\Sigma$ of $d$ elements, and put $\kappa=\kappa(G,\Sigma)$ and $C=\lceil2\pi/\kappa\rceil+1$. Then every subgroup $H\leq G$ satisfies
\[
    |H:H'| \leq C^{d|G:H|} .
\]
In particular, $\ab_k(G)\leq C^{dk}$ for all $k\geq1$.
\end{proposition}

\begin{proof}
Let $j=|G:H|$, and let $T$ be a set of representatives of the left cosets of $H$. For a linear character $\chi\in\widehat H$, the induced representation $\rho_\chi=\Ind_H^G\chi$ has an orthonormal basis $(e_t)_{t\in T}$ on which $g\in G$ acts by
\[
    \rho_\chi(g)e_t = \chi(h)\,e_{t'} , \qquad \text{where } gt=t'h \text{ with } t'\in T \text{ and } h\in H .
\]
So each $\rho_\chi(s)$ with $s\in\Sigma$ is the permutation matrix $P_s$ of $s$ on the cosets, which is independent of $\chi$, times a diagonal matrix whose entries are values $\chi(h)$ at elements $h\in H$ that depend only on $s$ and $t$. These $dj$ entries form a point of a torus of dimension $dj$, which determines $\chi$, and we divide this torus into $C^{dj}$ boxes of side $2\pi/C<\kappa$. If two distinct characters $\chi_1,\chi_2$ lie in the same box, then $\chi=\chi_1\overline{\chi_2}$ is nontrivial, and each $\rho_\chi(s)$ is $P_s$ times a diagonal matrix $D_s$ with entries within $\kappa$ of $1$. Let $f$ be the unit vector with all coordinates equal to $j^{-1/2}$. It is fixed by every $P_s$, so $\rho_\chi(s)f-f=P_s(D_s-1)f$, and all coordinates of $(D_s-1)f$ have absolute value less than $\kappa j^{-1/2}$. Hence $\|\rho_\chi(s)f-f\|<\kappa$ for every $s\in\Sigma$. But $\rho_\chi$ has no nonzero invariant vectors, by Frobenius reciprocity, and this contradicts the definition of $\kappa$. Hence the $|H:H'|$ linear characters of $H$ lie in distinct boxes, and $|H:H'|\leq C^{dj}$.
\end{proof}

Expansion enters this proof only through the Kazhdan constant of $G$ on the representations $\Ind_H^G\chi$, with $\chi$ a nontrivial linear character of a subgroup $H$ of index at most $k$, that is, on monomial representations of degree at most $k$. With $\kappa\geq\epsilon$, we obtain that every $G\in\mathcal{C}_{d,\epsilon}$ satisfies
\begin{equation}
    \label{eq:LW}
    \log \ab_k(G) \leq dk\log\left(\lceil2\pi/\epsilon\rceil+1\right) = O_\epsilon(dk).
\end{equation}

\subsection{Representation growth}\label{sub:bounds}

We now deduce \Cref{cor:quasilinear} and \Cref{cor:sharp}. The upper bounds follow directly from \Cref{prop:LW} together with \Cref{thm:main,thm:altfree}. For the lower bounds, we recall an expanding sequence with large abelian sections, due to Lubotzky and \.Zuk \cite{LZ03}.

\subsubsection{Upper bounds}

\begin{proof}[Proof of \Cref{cor:quasilinear}]
Let $G\in\mathcal{C}_{d,\epsilon}$. Multiplying the bound of \Cref{thm:main} with \eqref{eq:LW} gives
\[
    \log\Rep_k(G) \leq (d-1)k\log k+O(dk)+dk\log\left(\lceil2\pi/\epsilon\rceil+1\right) ,
\]
which is the claim. In particular, the expansion constant only enters through the term $O_\epsilon(dk)$.
\end{proof}

\begin{remark}\label{rem:methods}
The bound $\Rep_k(G)\leq\exp(O_\epsilon(dk^2))$ of \cite{HRV93} for $G\in\mathcal{C}_{d,\epsilon}$ is thus far from best possible. 
Their result is stated for the finite quotients of a fixed group with property (T), but the proof only uses the Kazhdan constant and the number of generators. It is a packing argument similar to the one in the proof of \Cref{prop:LW}, see \cite[Lem.~10 and proof of Prop.~IV]{HRV93} and \cite[proof of Prop.~2.8]{Pis17}. Here is how it goes in brief, in a slightly simplified form. A representation $\rho$ of degree $k$ is recorded by the tuple $(\rho(s))_{s\in\Sigma}$ of unitary matrices, a point of $U(k)^d$, which has dimension $dk^2$. If $\rho_1$ and $\rho_2$ are inequivalent irreducible representations of degree $k$, then $\rho_1\otimes\overline{\rho_2}$ has no nonzero invariant vectors, by Schur's lemma. It acts on $\mathrm{M}_k(\C)$ by $X\mapsto\rho_1(g)X\rho_2(g)^{-1}$, and applying the Kazhdan constant to the identity matrix shows that $\|\rho_1(s)-\rho_2(s)\|\geq\kappa\sqrt k$ for some $s\in\Sigma$, in the Hilbert--Schmidt norm. In other words, the tuples of pairwise inequivalent representations are at distance at least $\kappa$ from each other in the metric $\operatorname{dist}(A,B) = \max_{s\in\Sigma}k^{-1/2}\|A_s-B_s\|$ on $U(k)^d$. Balls of radius $\kappa/2$ around them are disjoint, and comparing volumes bounds their number by $\exp(O_\kappa(dk^2))$. 
The lack of sharpness in this argument comes from the fact that almost none of these points in $U(k)^d$ are representations of expanders. We reuse this argument in a slightly more general form in \Cref{lem:hrv}.
\end{remark}

\subsubsection{An expanding sequence with large abelian and representation growth}

The following groups are wreath products $C_5\wr T$ with an abelian base group, the counterpart of the wreath products $S\wr_XT$ of \Cref{sub:altfreegap} with $C_5$ in place of $S$ and $T$ acting regularly. Unlike the groups $W_n$ there, they have large abelian sections of small index, and these produce their many small representations.

\begin{example}\label{ex:lower}
Let $(p_n)_{n\geq1}$ be the primes $p\geq7$ congruent to $2$ modulo $5$, so that $5$ divides none of the orders $|\PSL_2(p_n)|=p_n(p_n^2-1)/2$. Put $T_n=\PSL_2(p_n)$, which is an expanding sequence with the standard two generators by Selberg's theorem \cite[Th.~4.4.2]{Lub94}, and let
\[
    G_n = C_5\wr T_n = \F_5[T_n]\rtimes T_n
\]
be the wreath product with $T_n$ acting on itself by right multiplication. (These are the wreath products of \Cref{ex:sharpindex} with $p=5$ and $T_n$ in place of $C_k$.) By \cite[Th.~5.13(1)]{LZ03}, the sequence $(G_n)_{n\geq1}$ is expanding, so it is $d_0$-expanding for some $d_0$. Each $G_n$ is generated by three elements, because the regular module is cyclic, but these three generators do not give expanders \cite[Th.~5.13(2)]{LZ03}. We claim that for every large $k$ some $G_n$ satisfies
\[
    \log \ab_k(G_n) \geq (1-o(1))k\log5
    \quad \text{and} \quad
    \log \Rep_k(G_n) \geq (1-o(1))k\log5 ,
\]
and that the groups $G_n$ are $\Alt(5)$-free.

Write $k_n=|T_n|$. The subgroup $\F_5[T_n]$ is abelian of index $k_n$ and order $5^{k_n}$, so
\[
    \ab_{k_n}(G_n) \geq 5^{k_n}
    \quad \text{and, by \eqref{eq:lower},} \quad
    \Rep_{k_n}(G_n) \geq \frac{5^{k_n}}{k_n} .
\]
Both growth functions are nondecreasing in $k$. By the prime number theorem in arithmetic progressions, the primes congruent to $2$ modulo $5$ satisfy $p_{n+1}/p_n\to1$. Since $k_n=p_n(p_n^2-1)/2$, also $k_{n+1}/k_n\to1$. For $k_n\leq k<k_{n+1}$ we therefore have $\ab_k(G_n)\geq5^{k_n}$ and $\Rep_k(G_n)\geq5^{k_n}/k_n$ with $k_n=(1-o(1))k$, which gives the claimed bounds for all large $k$.

Finally, every simple section of $G_n$ is a section of one of its composition factors, as in the proof of \Cref{prop:freepower}. These are $C_5$ and $T_n$, and since $|T_n|$ is prime to $5$, neither has a section isomorphic to $\Alt(5)$.
\end{example}

The example shows three things. First, as remarked by Lubotzky and \.Zuk \cite{LZ03} and by Sabatini \cite[Sec.~2.3]{Sab25}, the exponential bound of \Cref{prop:LW} cannot be improved to a polynomial one. Together with \Cref{prop:LW}, it gives
\[
    \sup_{G\in\mathcal{C}_{d,\epsilon}}\log\ab_k(G) = \Theta_{d,\epsilon}(k)
\]
for every $d\geq d_0$ and every $\epsilon$ small enough for $\mathcal{C}_{d,\epsilon}$ to contain the sequence of \Cref{ex:lower}. Second, $\sup_{G\in\mathcal{C}_{d,\epsilon}}\log\Rep_k(G)\geq(1-o(1))k\log5$ for all large $k$, so \Cref{cor:quasilinear} is sharp up to the logarithm, which answers \Cref{q:wigderson} up to the same logarithm. Third, since the groups are $\Alt(5)$-free, they give the lower bound in \Cref{cor:sharp}. In the $\Alt(c+1)$-free class the logarithm disappears, and the resulting bound $\log\Rep_k=O(k)$ is sharp up to the constant.

\begin{proof}[Proof of \Cref{cor:sharp}]
For the upper bound, multiply the bound of \Cref{thm:altfree} with \eqref{eq:LW}. For the lower bound, use \Cref{ex:lower}.
\end{proof}

\subsubsection{An example with very large abelianization growth}

Along an expanding sequence, \Cref{prop:LW} gives $\ab_k\leq\exp(O(k))$, and \Cref{cor:quasilinear} gives $\Rep_k\leq\exp(O(k\log k))$. For $d$-generated groups, the first condition automatically implies the second by \Cref{thm:main}. The converse fails for general sequences, as the following modification of a construction from \cite{Sab23} shows.

\begin{example}\label{ex:window}
Let $n\geq5$, let $S=\Alt(n)$, and let $p$ be a prime with $n<p\leq2n$. Since $p\nmid|S|$, the group algebra $\F_p[S]$, on which $S$ acts by left multiplication, splits as $\F_p\oplus I$, where $I$ is the augmentation submodule. Put $G_n=I\rtimes S$. This is analogous to the construction of Sabatini \cite[Ex.~9, Prop.~23]{Sab23} with the regular action of $S$ in place of its action on $n$ points. As a direct summand of the cyclic module $\F_p[S]$, the module $I$ is cyclic, so $G_n$ is $3$-generated. We claim that
\[
    \log\ab_k(G_n)=O(k\log k) \quad \text{for all } k\geq2 \text{ and } n ,
    \qquad \text{while} \qquad
    \ab_{|S|}(G_n)\geq n^{|S|-1} .
\]
So $\Rep_k(G_n)\leq\exp(O(k\log k))$ by \Cref{thm:main}, but no bound $\ab_k(G_n)\leq \exp (O(k))$ holds uniformly in $n$.

In order to see this, let $H\leq G_n$ have index at most $k$, put $W=H\cap I$, and let $\bar H=H/W$ be the image of $H$ in $S$. Since $W$ is abelian and normal in $H$, the group $H$ acts on $W$ through $\bar H$, and $[W,H]=[W,\bar H]$ lies in $W\cap H'$. Moreover $H'W/W=\bar H'$. Hence
\[
    |H:H'| = |H:H'W|\cdot|H'W:H'| = |\bar H:\bar H'|\cdot|W:W\cap H'| \leq |\bar H:\bar H'|\cdot|W_{\bar H}| ,
\]
where $W_{\bar H}=W/[W,\bar H]$ are the coinvariants. Since $p\nmid|\bar H|$, averaging over $\bar H$ gives a decomposition $W=W^{\bar H}\oplus[W,\bar H]$, so $\dim W_{\bar H}=\dim W^{\bar H}\leq\dim I^{\bar H}$. The $\bar H$-invariants of $\F_p[S]$ are the functions on $S$ that are constant on the cosets $\bar Ht$, so $\dim I^{\bar H}=|S:\bar H|-1$. If $\bar H=S$, this gives $|H:H'|\leq|S:S'|=1$. Otherwise $k\geq|G_n:H|\geq|S:\bar H|\geq n$, and $|\bar H:\bar H'|\leq|S:\bar H|^\alpha$ by \cite[Th.~3]{Sab23}, so
\[
    |H:H'| \leq k^\alpha p^{k-1} \leq k^\alpha(2k)^{k-1} = \exp\left(O(k\log k)\right) .
\]
For the lower bound, the abelian subgroup $I$ has index $|S|$ and order $p^{|S|-1}\geq n^{|S|-1}$.
\end{example}

\subsection{Slow representation growth does not imply \texorpdfstring{$d$}{d}-expansion}\label{sub:notsufficient}

We now show that neither $\ab_k\leq\exp(O(k))$ nor the stronger bound $\Rep_k\leq\exp(O(k))$ implies that a sequence of $d$-generated groups is $d$-expanding. The examples are free subdirect powers from \Cref{subsub:freepowers}.

We bound spectral gaps through Schreier graphs. Let a finite group $G$ act on a finite set $\Omega$ with $|\Omega| \geq 2$ on the right, and let $\Sigma=(s_1,\dots,s_d)$ be a tuple of elements of $G$. The \emph{Schreier graph} $\Sch(\Omega,\Sigma)$ has vertex set $\Omega$ and, for each $i$, an edge from $\omega$ to $\omega^{s_i}$ and one from $\omega$ to $\omega^{s_i^{-1}}$. Its normalized Laplacian and spectral gap are defined as for Cayley graphs, with $G$ acting on the functions on $\Omega$ by $(g\cdot w)(\omega)=w(\omega^g)$. If $G$ is transitive on $\Omega$, then the spectral gap of $\Cay(G,\Sigma)$ is at most that of $\Sch(\Omega,\Sigma)$. Indeed, fix $\omega_0\in\Omega$, and send a function $w$ on $\Omega$ to the function $f(g)=w(\omega_0^g)$ on $G$. Every point of $\Omega$ is $\omega_0^g$ for exactly $|G|/|\Omega|$ elements $g$, so this multiplies all norms by the same factor and preserves orthogonality to the constants. It also maps $s\cdot w$ to $sf$, so it preserves Rayleigh quotients.

We will use the following basic property of the free subdirect power $\Pi_d(G)$ of a finite group $G$.

\begin{lemma}\label{lem:freepowergap}
Let $G$ be a finite group and let $d\geq d(G)$. For every generating $d$-tuple $\Sigma$ of $\Pi_d(G)$ and every generating $d$-tuple $Y$ of $G$, there is a surjection $\Pi_d(G)\to G$ that maps $\Sigma$ to $Y$. Consequently, the spectral gap of $\Cay(\Pi_d(G),\Sigma)$ is at most that of $\Cay(G,Y)$.
\end{lemma}

\begin{proof}
Write $\Pi_d(G)=F/\bigcap\mathcal K$ as in \Cref{subsub:freepowers}, and let $\varphi\colon F\to\Pi_d(G)$ be the surjection that maps the free generators to $\Sigma$. Composing $\varphi$ with the quotient maps $\Pi_d(G)\to F/K$ for $K\in\mathcal K$ gives $|\mathcal K|$ surjections from $F$ onto copies of $G$, with pairwise distinct kernels. In particular, one of them is the kernel $K$ of the surjection $\psi\colon F\to G$ that maps the free generators to $Y$. Composing the quotient map $\Pi_d(G)\to F/K$ with the isomorphism $F/K\to G$ induced by $\psi$ gives a surjection that maps $\Sigma$ to $Y$.
For the spectral gap, let $\Pi_d(G)$ act on $G$ by right multiplication through this surjection. This action is transitive, and its Schreier graph with respect to $\Sigma$ is $\Cay(G,Y)$, since $\Sigma$ maps to $Y$ entry by entry.
\end{proof}

Call a sequence of finite groups \emph{uniformly $d$-expanding} if there is $\epsilon>0$ such that the Cayley graphs with respect to all their generating $d$-tuples have spectral gap at least $\epsilon$. 
By \Cref{lem:freepowergap}, if $(\Pi_d(G_n))$ is $d$-expanding, then $(G_n)$ is uniformly $d$-expanding.\footnote{We do not know whether the converse holds. For instance, the groups $\PSL_2(p)$ are uniformly expanding for a set of primes of density one \cite{BG10}, but we do not know whether the groups $\Pi_d(\PSL_2(p))=\PSL_2(p)^{t_d(\PSL_2(p))}$ of \Cref{ex:psl2power} are $d$-expanding along these primes. Even powers with $p\sim|\PSL_2(p)|^{1/3}$ factors are open \cite[Question~4.5]{Sab25}.}
So a single sequence of generating $d$-tuples of $G_n$ with spectral gap tending to $0$ prevents $(\Pi_d(G_n))$ from being $d$-expanding.

\begin{example} \label{ex:altpower-notexp}
Consider the alternating groups. By \Cref{subsub:freepowers}, the free subdirect power is $\Pi_d(\Alt(n))=\Alt(n)^{t_d(\Alt(n))}$. However, this group is not $d$-expanding, since $\Alt(n)$ has generating $d$-tuples with spectral gap tending to $0$ as $n \to \infty$, for instance $Y=((1\,2\,3),c,1,\dots,1)$, where $c$ is the cycle $(1\,2\,\cdots\,n)$ or $(2\,3\,\cdots\,n)$, according to whether $n$ is odd or even, followed by $d-2$ identities. This follows by the argument of \cite[Ex.~4.3.3C]{Lub94}, which treats $\Sym(n)$ with a transposition and an $n$-cycle. The small representations of $\Pi_d(\Alt(n))$, however, grow as fast as $\exp(\Theta(k\log k))$ with $k = n - 1$ by \Cref{ex:altpower}. Since $\Pi_d(\Alt(n))$ has no abelian composition factors, this answers the versions of \cite[Questions~1.1 and 4.4]{Sab25} with a fixed number $d$ of generators in the negative.
\end{example}

To get examples with $\Rep_k\leq \exp ( O(k) )$, we reuse the iterated wreath products $W_n$ of \Cref{subsub:wreath} as building blocks. Their non-expanding generating tuples come from an amenable group acting on a rooted tree, as we explain below.

\subsubsection{Iterated wreath products and bounded automatic automorphisms}

The automorphisms of the rooted tree $X^*$ over a finite set $X$ are the permutations of the finite words over $X$ that preserve prefixes. Such an automorphism $g$ permutes the first letters, and at every word $w$ it induces an automorphism $g_w$ of the subtree below $w$, which we identify with $X^*$ and call the \emph{section} of $g$ at $w$. If $S\leq\Sym(X)$, the automorphisms all of whose sections permute the first letters through $S$ act on the words of length $n$ through the $n$-fold iterated permutational wreath product $S\wr\cdots\wr S$. For $S=\PSL_2(p)$ acting on $X=\mathbf P^1(\F_p)$, this is precisely the group $W_n$ of \Cref{subsub:wreath}.

Segal \cite{Seg01} found finitely generated groups of tree automorphisms that map onto all these wreath products at once. Let $\sigma_1,\sigma_2$ generate $S$, and fix distinct points $x_1,x_2\in X$. Let $a_i$ apply $\sigma_i$ to the first letter of a word and leave the other letters unchanged. Let $b_i$ fix the first letter, with section $a_i$ at $x_1$, section $b_i$ at $x_2$, and trivial sections elsewhere. In other words, $b_i$ applies $\sigma_i$ to the letter that follows a prefix $x_2^jx_1$ with $j\geq0$, and fixes the words without such a prefix. The group $\Gamma=\langle a_1,a_2,b_1,b_2\rangle$ is the group $A$ of \cite[Lem.~1]{Seg01} for the constant sequence $S_n=S$, with $x_1$ and $x_2$ the points $1$ and $l_0$ there. If $S$ is perfect and transitive, and distinct points of $X$ have distinct stabilizers in $S$, then the stabilizer $\St_\Gamma(1)$ of the first level is the direct product of copies of $\Gamma$, one acting on the subtree below each letter \cite[Lem.~1]{Seg01}. The rooted elements $\langle a_1,a_2\rangle\cong S$ meet $\St_\Gamma(1)$ trivially and map onto the action of $\Gamma$ on the first letters, so $\Gamma=\St_\Gamma(1)\rtimes\langle a_1,a_2\rangle$. Conjugation by a rooted element permutes the subtrees below the letters as it permutes $X$, so $\Gamma\cong\Gamma\wr_XS$, and by induction the image of $\Gamma$ on the words of length $n$ is the whole iterated wreath product.

The generators of $\Gamma$ are of a very restricted kind. An automorphism is \emph{automatic} if it has finitely many distinct sections, and \emph{bounded} if the number of words $w$ of length $n$ with $g_w\neq1$ is bounded independently of $n$. The sections of $a_i$ at nonempty words are trivial. The sections of $b_i$ are $b_i$ at the words $x_2^j$, $a_i$ at the words $x_2^jx_1$, and $1$ elsewhere, so for every $n\geq1$ exactly two words of length $n$ have a nontrivial section. Bartholdi, Kaimanovich and Nekrashevych proved that the group of all bounded automatic automorphisms is amenable \cite[Th.~1.2]{BKN10}, and they list Segal's groups among its subgroups \cite[Ex.~1.7]{BKN10}. Together, this gives the following.

\begin{theorem}[\cite{Seg01, BKN10}]\label{thm:segal}
Let $S\leq\Sym(X)$ be a perfect transitive group generated by two elements, in which distinct points have distinct stabilizers. Then $\Gamma=\langle a_1,a_2,b_1,b_2\rangle$ is amenable, and its image in $\Sym(X^n)$ is the $n$-fold iterated wreath product of $S$, for every $n\geq1$.
\end{theorem}

We deduce that the groups $\Pi_d(W_n)$ are not $d$-expanding.

\begin{proposition}\label{prop:notsufficient}
Let $d\geq4$, and let $W_n$ be the iterated wreath products of \Cref{subsub:wreath}. Then the groups $\Pi_d(W_n)$ are $d$-generated and satisfy
\[
    \ab_k\leq k^\alpha
    \qquad \text{and} \qquad
    \Rep_k\leq\exp(O_d(k))
    \qquad \text{for all } k ,
\]
but they are not $d$-expanding.
\end{proposition}

\begin{proof}
By the proof of \Cref{prop:freepower}, every composition factor of $\Pi_d(W_n)$ is a composition factor of $W_n$, and hence isomorphic to $\PSL_2(p)$, whose order is prime to $5$, so $\Pi_d(W_n)$ has no abelian composition factors and is $\Alt(5)$-free. Hence $\ab_k\leq k^\alpha$ by \cite[Th.~3]{Sab23}, and $\Rep_k\leq\exp(O_d(k))$ by \Cref{thm:altfree}. For the expansion, let $Y_n$ be the $d$-tuple consisting of the images in $W_n$ of the four generators of \Cref{thm:segal}, followed by $d-4$ identities. The group $S=\PSL_2(p)$ is perfect and generated by two elements, and it acts $2$-transitively on $\mathbf P^1(\F_p)$, so distinct points have distinct stabilizers. By \Cref{thm:segal}, the groups $W_n$ are therefore quotients of a single amenable group, and $Y_n$ is the image of a fixed generating tuple. By \cite[Th.~3.1]{LW93}, the spectral gaps of $\Cay(W_n,Y_n)$ tend to $0$. By \Cref{lem:freepowergap}, so do those of $\Cay(\Pi_d(W_n),\Sigma)$ for every generating $d$-tuple $\Sigma$ of $\Pi_d(W_n)$.
\end{proof}

Like the powers of alternating groups in \Cref{ex:altpower-notexp}, these groups answer the versions of \cite[Questions~1.1 and 4.4]{Sab25} with a fixed number $d$ of generators in the negative. We do not know whether they are $d'$-expanding for some $d'>d$, so the original questions remain open for them.

\section{Expanding powers of simple groups}\label{sec:powers}

In this section we ask whether the remaining logarithm in \Cref{cor:quasilinear} is really there, and study the expanding powers of simple groups that could produce it.

\subsection{Can the logarithm really be present?}\label{sub:log}

For expanding sequences, \Cref{cor:quasilinear} determines an upper bound on the representation growth that is tight up to a single logarithm. Write $\Delta_{d,\epsilon}$ for the gap function of the class $\mathcal{C}_{d,\epsilon}$. Every $G\in\mathcal{C}_{d,\epsilon}$ satisfies $\Rep_k(G)\leq\Delta_{d,\epsilon}(k)\cdot\ab_k(G)$, and $\log\ab_k(G)=O_{d,\epsilon}(k)$ by \eqref{eq:LW}. Conversely, since $\ab_k(G)\geq1$, we have $\Delta_{d,\epsilon}(k)\leq\sup_{G\in\mathcal{C}_{d,\epsilon}}\Rep_k(G)$. So the logarithm is absent, that is, $\sup_{G\in\mathcal{C}_{d,\epsilon}}\log\Rep_k(G)=O_{d,\epsilon}(k)$, if and only if
\[
    \Delta_{d,\epsilon}(k) = \exp\left(O_{d,\epsilon}(k)\right) .
\]
We do not know whether this holds.

Among all $d$-generated groups, the logarithm can be present. It is attained by the powers $\Alt(n)^{t_d(\Alt(n))}$ at Hall's ceiling in \Cref{ex:altpower}, but these are not $d$-expanding by \Cref{ex:altpower-notexp}. By \Cref{thm:altfree}, the logarithm can only come from large alternating sections, so the natural candidates among expanding sequences are again powers of alternating groups, but with fewer factors. The question is therefore how many factors an expanding power of $\Alt(n)$ can have.

At $k=n-1$, a power $\Alt(n)^{f(n)}$ has $\ab_k=1$ and $\Rep_k=1+f(n)$ by \Cref{prop:powers}(1), and if it is $d$-generated, then $f(n)$ is at most Hall's ceiling $t_d(\Alt(n))=\exp((1+o(1))(d-1)n\log n)$ by \Cref{prop:powers}(2). The logarithm would be present if there were an expanding sequence of such powers with $f(n)\geq|\Alt(n)|^\delta$ for a fixed $\delta>0$. Indeed, along such a sequence $\log\Delta_{d,\epsilon}(k)\geq\log f(k+1)=\Omega(k\log k)$, and together with \Cref{cor:quasilinear} this would give $\sup_{G\in\mathcal{C}_{d,\epsilon}}\log\Rep_k(G)=\Theta(k\log k)$ for infinitely many $k$. In other words, the bound of \Cref{thm:main} would be attained along an expanding sequence.

Expanding sequences of powers of alternating groups do exist. The largest multiplicities known are due to Bartholdi and Kassabov. For $e\geq2$, they construct a group with property (T), generated by $e+4$ elements, that maps onto $\Alt(n)^{f(n)}$ with $n=p^{(e+2)j}-1$ and $f(n)=p^{j^{e+1}}$, for a fixed prime $p$ and all $j\geq1$ \cite[Cor.~14]{BK25}. The images of its generators make these powers an expanding sequence \cite[Prop.~3.3.1]{Lub94}, and $\log f(n)$ is of order $(\log n)^{e+1}$. So $f(n)$ is superpolynomial in $n$ but still quasipolynomial, far below Hall's ceiling $\exp((1+o(1))(d-1)n\log n)$. In particular $\log f(n)=O(n)$, so these powers satisfy $\Rep_k\leq\exp(O(k))$ by \Cref{prop:powers}(3), and hence do not exhibit the logarithm.

At the other end, there is no expansion at Hall's ceiling, as we saw in \Cref{ex:altpower-notexp}. In between these values, we do not know for which $f(n)$ expansion occurs. For the purposes of this paper, the most relevant question is whether it can occur with $\log f(n)$ at least a fixed fraction of $\log t_d(\Alt(n))$. This is \Cref{q:alt} from the introduction.

% A multiplicity $f(n)=\exp(n^\alpha)$ with $\alpha<1$ would already give $\Delta_{d,\epsilon}(k)\geq\exp(k^\alpha)$ for infinitely many $k$, far above the gaps known along expanding sequences, but still short of the logarithm.

We give two partial answers. In \Cref{subsub:random} we show that random generating $d$-tuples with $d\geq3$ do not give expanders as soon as the number of factors is superpolynomial in $n$, so a positive answer to \Cref{q:alt} needs special generators. In \Cref{sub:proofceiling} we show that expansion forces the number of factors to stay below Hall's ceiling by a fixed fraction on a logarithmic scale. Namely, if the powers in \Cref{q:alt} are $d$-expanding with expansion constant $\epsilon$, then $\delta\leq d-1-\epsilon/160$.
In \Cref{subsub:lietype} we show that the analogous question has a positive answer for groups of Lie type of rank at least two.

\subsection{Random generators}\label{subsub:random}

A positive answer to \Cref{q:alt} cannot come from random generators. The reason is that $\Alt(n)$ has generating tuples with small spectral gap that are not too rare, and a power with many factors contains such a tuple in one of its coordinates.

\subsubsection{Leaky tuples}\label{subsub:leaky}

We first construct the generating tuples of $\Alt(n)$ with small spectral gap. To show that they generate, we use the following criterion.

\begin{lemma}\label{lem:altgen}
Let $[n]=A\sqcup B$ with $2\leq|A|<|B|$ and $|B|\geq5$. Let $y_1,y_2\in\Alt(n)$ preserve $A$, write $y_i=(u_i,v_i)$ with $u_i\in\Sym(A)$ and $v_i\in\Sym(B)$, and let $z\in\Alt(n)$ map exactly one point of $A$ into $B$. If $u_1$ is a cycle of length $|A|$ and $\langle v_1,v_2\rangle\geq\Alt(B)$, then $\langle y_1,y_2,z\rangle=\Alt(n)$.
\end{lemma}

\begin{proof}
Let $G=\langle y_1,y_2,z\rangle$, and let $N\leq\Sym(B)$ consist of the $v$ with $(1,v)\in\langle y_1,y_2\rangle$. By Goursat's lemma, $N$ is normal in $\langle v_1,v_2\rangle$, with quotient isomorphic to a quotient of $\langle u_1,u_2\rangle\leq\Sym(A)$, which has order at most $|A|!$. As $2\leq|A|<|B|$, this is smaller than $|B|!/2=|\Alt(B)|$, so $N\neq1$. As $\langle v_1,v_2\rangle$ is $\Alt(B)$ or $\Sym(B)$, whose nontrivial normal subgroups all contain $\Alt(B)$, we get $N\geq\Alt(B)$. Thus $G$ contains $\Alt(B)$, acting trivially on $A$.

Fix distinct $x,y\in B$. Since $\Alt(n)$ is generated by the $3$-cycles $(x\,y\,p)$ with $p\notin\{x,y\}$, it suffices to find these in $G$. For $p\in B$ they lie in $\Alt(B)$. As $z$ maps exactly one point of $A$ into $B$, it maps $B$ onto $(B\setminus\{c\})\cup\{a\}$ for some $c\in B$ and $a\in A$. So $G$ contains $\Alt(B)^z=\Alt(B^z)$, and hence a $3$-cycle $(x'\,y'\,a)$ with $x',y'\in B\setminus\{c\}$. Conjugating it by powers of $y_1$, which preserves $B$ and acts on $A$ as the cycle $u_1$, and then by $\Alt(B)$, which is $2$-transitive on $B$, we find $(x\,y\,p)$ in $G$ for every $p\in A$.
\end{proof}

\begin{lemma}\label{lem:leaky}
Let $d\geq3$ and $h\geq2$. For all large $n$, depending on $h$, a uniformly random generating $d$-tuple of $\Alt(n)$ has a Cayley graph with spectral gap at most $2/(dh)$ with probability at least $n^{-dh}$.
\end{lemma}

\begin{proof}
Write $S=\Alt(n)$, put $A=\{1,\dots,h\}$ and $B=[n]\setminus A$, and let $H=S\cap(\Sym(A)\times\Sym(B))$ be the stabilizer of $A$ in $S$. Call a tuple $Y=(y_1,\dots,y_d)\in S^d$ \emph{leaky} if $y_1,\dots,y_{d-1}\in H$ and $y_d$ maps exactly one point of $A$ into $B$. So a leaky tuple preserves $A$, except that $y_d$ lets a single point leak out. We show that leaky tuples have small spectral gap, and that a uniformly random tuple is leaky and generating with probability at least $n^{-dh}$.

\emph{Leaky tuples have small spectral gap.} Let $Y$ be leaky, and take $w=\mathbf 1_A-h/n$, which is orthogonal to the constants and has $\lVert w\rVert^2=h(n-h)/n$. It is invariant under $y_1,\dots,y_{d-1}$. Since $y_d$ maps exactly one point of $A$ out of $A$, the sets $A$ and $A' = A^{y_d^{-1}}$ differ in one point on each side, so $w-y_d\cdot w=\mathbf 1_A-\mathbf 1_{A'}$ has squared norm $2$. The Rayleigh quotient of $w$ on $\Sch([n],Y)$ is therefore
\[
    \frac{1}{2d}\sum_{i=1}^d\frac{\lVert w-y_i\cdot w\rVert^2}{\lVert w\rVert^2}
    = \frac{1}{2d}\cdot\frac{2}{h(n-h)/n}
    = \frac{n}{dh(n-h)} \leq \frac{2}{dh}
\]
for $n\geq2h$, and as $S$ is transitive on $[n]$, this bounds the spectral gap $\epsilon(Y)$ of $\Cay(S,Y)$.

\emph{Leaky generating tuples are not too rare.} Let $Y$ be a uniformly random element of $S^d$, and consider the following three events:
\begin{description}
    \item[$E_1$] $y_1,\dots,y_{d-1}\in H$, and 
    $y_d$ maps exactly one point of $A$ into $B$,
    \item[$E_2$] $u_1$ is a $h$-cycle, where $y_i=(u_i,v_i)$ with $u_i\in\Sym(A), v_i\in\Sym(B)$ for $i=1,2$,
    \item[$E_3$] $\langle v_1,v_2\rangle\geq\Alt(B)$.
\end{description}
Note that $H$ has index $\binom nh$ in $S$, and the image of $A$ under $y_d$ is a uniformly random $h$-subset of $[n]$, so
\[
\PP(E_1)= h(n-h) \binom{n}{h}^{-d}.
\]
Given $E_1$, the permutations $u_1,u_2$ are independent and uniform in $\Sym(A)$, so $\PP(E_2\mid E_1)=1/h$. Given also $u_1,u_2$, the $v_i$ are independent, with $v_i$ uniform among the permutations of $B$ with the same sign as $u_i$, since $y_i$ is even. By Dixon's theorem \cite{Dix69}, two independent uniform elements of $\Sym(B)$ generate a group containing $\Alt(B)$ with probability $1-o(1)$, and prescribing their signs, an event of probability $1/4$, multiplies the probability of failure by at most $4$. So $\PP(E_3\mid E_1,E_2)=1-o(1)$. On all three events, $Y$ is leaky, and for large $n$ it generates $S$ by \Cref{lem:altgen} with $z=y_d$. Hence
\[
    \PP\left(\epsilon(Y)\leq\frac{2}{dh} \middle| \langle Y\rangle=S\right)
    \geq \PP(Y\text{ leaky}, \langle Y\rangle=S)
    \geq h(n-h) \binom{n}{h}^{-d}\cdot\frac1h\cdot(1-o(1)) ,
\]
which is at least $n^{-dh}$ for large $n$.
\end{proof}

\begin{remark}
Leaky tuples in the proof above are a special instance of a phenomenon that Friedman studied in his proof of Alon's second eigenvalue conjecture \cite{Fri08}. He studies random regular graphs that are the Schreier graphs $\Sch([n],Y)$ of tuples $Y$ of independent uniform permutations of $[n]$. A \emph{tangle} is a fixed finite graph whose edges are labelled by $1,\dots,d$, and it occurs in $\Sch([n],Y)$ if it embeds there preserving the labels \cite[Def.~4.1]{Fri08}. Friedman proves that a tangle with $v$ vertices and $e\geq v$ edges occurs with probability of order $n^{-(e-v)}$ \cite[Th.~4.7]{Fri08}.

For a leaky tuple, the labelled graph that $\Sch([n],Y)$ induces on $A$ is a tangle with $h$ vertices and $dh-1$ undirected edges, namely $h$ edges for each of $y_1,\dots,y_{d-1}$ and $h-1$ for $y_d$. Friedman's count shows that such a tangle occurs with probability of order $n^{-(d-1)h+1}$. This is larger than our bound $n^{-dh}$, because his count includes occurrences on every $h$-subset of $[n]$, while we fix $A=\{1,\dots,h\}$. Our proof, however, still needs a restriction to even permutations, and it needs to show that the tuple generates $\Alt(n)$, which is not included in Friedman's count.
\end{remark}

\subsubsection{Random generators of powers}\label{subsub:randompowers}

A power with many factors contains a tuple with small spectral gap in one of its coordinates, and this is enough to spoil its spectral gap.

\begin{proposition}\label{prop:random}
Let $d\geq3$ and $h\geq2$. Let $\Sigma$ be a uniformly random generating $d$-tuple of $\Alt(n)^t$, where $t=t(n)$ satisfies $n^{dh+1}\leq t\leq t_d(\Alt(n))$. Then, with probability tending to $1$ as $n\to\infty$, the spectral gap of $\Cay(\Alt(n)^t,\Sigma)$ is at most $2/(dh)$.
\end{proposition}

\begin{proof}
Write $S=\Alt(n)$, and call a generating $d$-tuple of $S$ \emph{slow} if its Cayley graph has spectral gap at most $2/(dh)$. By \Cref{lem:leaky}, a uniformly random generating $d$-tuple of $S$ is slow with probability $q_n\geq n^{-dh}$ for large $n$. If some coordinate tuple of $\Sigma$ is slow, then so is $\Sigma$, since $S^t$ acts transitively on $S$ by right multiplication through that coordinate, with Schreier graph the Cayley graph of the coordinate tuple. It thus remains to show that some coordinate tuple of $\Sigma$ is slow with probability tending to $1$. Being slow is invariant under $\Aut(S)$, which acts freely on the $\phi_d(S)=t_d(S)\lvert\Aut(S)\rvert$ generating $d$-tuples of $S$ by \Cref{rem:hall}. So a proportion $q_n$ of the $t_d(S)$ orbits consist of slow tuples. By \Cref{thm:hall}, every sequence of $t$ distinct orbits comes from exactly $\lvert\Aut(S)\rvert^t$ generating $d$-tuples of $S^t$, one for each choice of a tuple in each orbit. So the orbits of the coordinate tuples $Y^{(1)},\dots,Y^{(t)}$ of $\Sigma$ form a uniformly random sequence of $t$ distinct orbits. Hence the probability that no $Y^{(i)}$ is slow is
\[
    \prod_{i=0}^{t-1}\frac{(1-q_n)t_d(S)-i}{t_d(S)-i} \leq (1-q_n)^t \leq\exp(-tn^{-dh}) \leq\exp(-n) ,
\]
since each factor is at most $1-q_n$, and $q_n\geq n^{-dh}, t\geq n^{dh+1}$.
\end{proof}

Applying this for every fixed $h$, we see that for $d\geq3$, random generating $d$-tuples of $\Alt(n)^{f(n)}$ are not expanders as soon as $f(n)$ is superpolynomial in $n$, that is, $\log f(n)/\log n\to\infty$. More precisely, for every $\epsilon>0$, the probability that the Cayley graph of a uniformly random generating $d$-tuple of $\Alt(n)^{f(n)}$ has spectral gap at least $\epsilon$ tends to $0$ as $n \to \infty$. Indeed, given $\epsilon$, choose $h$ with $2/(dh)<\epsilon$. Since $f(n)$ is superpolynomial, it exceeds $n^{dh+1}$ for all large $n$, and \Cref{prop:random} applies with this $h$. 
% In particular, however the tuples for different $n$ are chosen jointly, the resulting sequence of Cayley graphs is almost surely not a family of expanders. 
We record two consequences.

\begin{example}\label{ex:ceilingrandom}
Let $d\geq2$, and consider the groups $\Pi_d(\Alt(n))=\Alt(n)^{t_d(\Alt(n))}$ of \Cref{ex:altpower}, which sit at Hall's ceiling for $d$ generators. By \Cref{ex:altpower-notexp}, they are not $d$-expanding. One may hope that more generators help, since $\Pi_d(\Alt(n))$ is $d'$-generated for every $d'\geq d$, and Hall's ceiling for $d'$ generators is far higher. Random generators do not help. For every $d'>d$, the multiplicity $t_d(\Alt(n))$ is superpolynomial in $n$, so by \Cref{prop:random}, uniformly random generating $d'$-tuples of $\Pi_d(\Alt(n))$ are not expanders either. We do not know whether $(\Pi_d(\Alt(n)))$ is $d'$-expanding for some $d'>d$.
\end{example}

\begin{example}\label{ex:bkrandom}
Consider the expanding powers $\Alt(n)^{f(n)}$ of Bartholdi and Kassabov from \Cref{sub:log}. These are $d$-expanding for every $d\geq e+4$. Their multiplicities $f(n)$ are superpolynomial in $n$. By \Cref{prop:random}, uniformly random generating $d$-tuples of the same groups are not expanders. So these powers form a family of expanders with respect to suitable generating $d$-tuples, but not with respect to random generating $d$-tuples.

This gives an interesting perspective on the question of Lubotzky and Weiss whether expansion is a property of the groups alone \cite[Problem~1.1]{LW93}. The answer is known to be negative, as first shown by Alon, Lubotzky and Wigderson \cite{ALW01}. The alternating groups give a simple example. They are expanding with respect to carefully constructed generating sets of bounded size \cite{Kas07}, but not with respect to a $3$-cycle and a long cycle, as in \Cref{ex:altpower-notexp}. For $\Alt(n)$, the non-expanding generating sets are special and highly structured, and whether random generating pairs give expanders is a well-known open question \cite[Problem~2.28]{Lub12}. For the powers of $\Alt(n)$ above, the roles are reversed. The expanding generating tuples are special, coming from a group with property (T), while random generating tuples do not expand.
\end{example}

\begin{remark}\label{rem:beta}
The proof of \Cref{prop:random} uses only one property of the multiplicity $t$, and it applies to powers of any nonabelian finite simple group $S$. For $\epsilon>0$, let $\beta_{d,\epsilon}(S)$ be the probability that a uniformly random generating $d$-tuple of $S$ has a Cayley graph with spectral gap at most $\epsilon$. If $t\leq t_d(S)$ and $t\,\beta_{d,\epsilon}(S)\to\infty$, then a uniformly random generating $d$-tuple of $S^t$ has spectral gap at most $\epsilon$ with probability tending to $1$. For $S=\Alt(n)$ and $\epsilon=2/(dh)$ with $h\geq2$, \Cref{lem:leaky} gives $\beta_{d,\epsilon}(S)\geq n^{-dh}=n^{-2/\epsilon}$ for large $n$. This is $|S|^{-o(1)}$ for fixed $\epsilon$, so every multiplicity that is superpolynomial in $n$ satisfies this condition. For groups $S = G(q)$ of Lie type of bounded rank, $\beta_{d,\epsilon}(S)$ is much smaller, namely $\beta_{d,\epsilon}(S)\leq C|S|^{-\delta}$ for suitable $\epsilon,\delta,C>0$ depending only on the rank of $G$ \cite{BGGT15}. Moreover, for $S=G(p)$ with $p$ prime, Becker and Breuillard show that $\beta_{d,\epsilon}(S)=0$ for suitable $\epsilon$ for all $p$ outside a small set of exceptional primes \cite{BB26}. So for most primes, no multiplicity $t$ satisfies this condition.
\end{remark}

\subsection{Expansion near Hall's ceiling}\label{sub:proofceiling}

We show in this subsection that expansion forces the multiplicity of a power of $\Alt(n)$ to stay well below Hall's ceiling, which is $\exp((1+o(1))(d-1)n\log n)$ for $d$ generators. This is \Cref{thm:ceiling} from the introduction. Its proof gives slightly more. Let $n\geq1000$ and let $r$ be an integer with $40\leq r\leq n/2$. Then for every $d\geq2$, every generating $d$-tuple of $\Alt(n)^t$ with $t\geq3\,t_d(\Alt(n))\,n^{-r/40}$ has spectral gap at most $2r/n$. In particular, within a polynomial factor of the ceiling the spectral gap tends to $0$.

We now prove \Cref{thm:ceiling}. The coordinate tuples of a generating $d$-tuple of $\Alt(n)^t$ are $t$ pairwise non-conjugate generating $d$-tuples of $\Alt(n)$ by \Cref{thm:hall}, since $\Aut(\Alt(n)) = \Sym(n)$ for $n \geq 7$. The proof has two steps.
\begin{enumerate}
    \item When $t$ is large, two of the $t$ generating $d$-tuples of $\Alt(n)$ must be nearly conjugate, in a sense made precise below.
    \item Two nearly conjugate coordinates force a small spectral gap.
\end{enumerate}

\subsubsection{Nearly conjugate coordinates}\label{subsub:intertwiner}

We begin with the second step. Recall from \Cref{rem:methods} that a spectral gap keeps inequivalent irreducible representations apart. This is a lemma of de la Harpe, Robertson and Valette \cite[Lem.~10]{HRV93}, following Wassermann \cite[Cor.~2]{Was91}. We use it in the following form, with the spectral gap in place of the Kazhdan constant, and for representations without a common irreducible constituent.

\begin{lemma}\label{lem:hrv}
Let $G$ be a finite group, let $\Sigma=(s_1,\dots,s_d)$ be a generating tuple of $G$, and let $\theta,\theta'$ be unitary representations of $G$ on the same space of dimension $m$, without a common irreducible constituent. Then the spectral gap of $\Cay(G,\Sigma)$ is at most
\[
    \frac{1}{2dm}\sum_{l=1}^d\lVert\theta(s_l)-\theta'(s_l)\rVert_{\mathrm{HS}}^2 .
\]
\end{lemma}

\begin{proof}
The group $G$ acts unitarily on the linear maps of the space, with the Hilbert--Schmidt inner product, by $g\cdot M=\theta(g)M\theta'(g)^{-1}$. Its invariant vectors are the intertwiners from $\theta'$ to $\theta$, and there are none, by Schur's lemma. The identity $I$ has $\lVert I\rVert_{\mathrm{HS}}^2=m$ and $\lVert s\cdot I-I\rVert_{\mathrm{HS}}=\lVert\theta(s)-\theta'(s)\rVert_{\mathrm{HS}}$. So its Rayleigh quotient is the displayed quantity, and it bounds the spectral gap.
\end{proof}

We will apply this when $\theta$ and $\theta'$ come from a single representation of a group $H$, pulled back along two homomorphisms $G\to H$ that agree on all generators but one. The bound is then a character ratio.

\begin{lemma}\label{lem:charratio}
Let $G$ be a finite group, let $\Sigma=(s_1,\dots,s_d)$ be a generating tuple of $G$, and let $\rho,\rho'\colon G\to H$ be homomorphisms to a finite group $H$. Suppose that there is $\tau\in H$ with
\[
    \rho'(s_1) = \rho(s_1)\tau
    \qquad \text{and} \qquad
    \rho'(s_l) = \rho(s_l)
    \quad \text{for } l\geq2.
\]
Let $\psi$ be a unitary representation of $H$ with character $\chi$, such that $\psi\circ\rho$ and $\psi\circ\rho'$ have no common irreducible constituent. Then the spectral gap of $\Cay(G,\Sigma)$ is at most
\[
    \frac1d\left(1-\frac{\operatorname{Re}\chi(\tau)}{\chi(1)}\right) .
\]
\end{lemma}

\begin{proof}
Put $\theta=\psi\circ\rho$ and $\theta'=\psi\circ\rho'$. They agree on $s_l$ for $l\geq2$, and since $\psi(\rho(s_1))$ and $\psi(\tau)$ are unitary,
\[
    \lVert\theta(s_1)-\theta'(s_1)\rVert_{\mathrm{HS}}^2 = \lVert\psi(\rho(s_1))\left(I-\psi(\tau)\right)\rVert_{\mathrm{HS}}^2 = \lVert I-\psi(\tau)\rVert_{\mathrm{HS}}^2 = 2\chi(1)-2\operatorname{Re}\chi(\tau) .
\]
\Cref{lem:hrv} with $m=\chi(1)$ gives the claim.
\end{proof}

So the spectral gap is small as soon as $\operatorname{Re}\chi(\tau)/\chi(1)$ is close to $1$. We will apply this with $H = \Sym(n)$ and the standard representation, where this ratio is governed by the support of $\tau$. Our permutations act on the right, so that $p^{\beta\gamma}=(p^\beta)^\gamma$, and we write $\supp\sigma$ for the set of points moved by $\sigma\in\Sym(n)$.

\begin{lemma}\label{lem:intertwiner}
Let $G$ be a finite group, let $\Sigma=(s_1,\dots,s_d)$ be a generating tuple of $G$, and let $\rho,\rho'\colon G\to\Sym(n)$ be homomorphisms such that $G$ acts transitively on $[n]^2$ by $(p,q)^g=(p^{\rho(g)},q^{\rho'(g)})$. Suppose that there is $\tau\in\Sym(n)$ with
\[
    \rho'(s_1) = \rho(s_1)\tau
    \qquad \text{and} \qquad
    \rho'(s_l) = \rho(s_l)
    \quad \text{for } l\geq2.
\]
Then the spectral gap of $\Cay(G,\Sigma)$ is at most $\lvert\supp\tau\rvert/(d(n-1))$.
\end{lemma}

\begin{proof}
Let $\psi$ be the representation of $\Sym(n)$ on the vectors in $\C^n$ with coordinate sum $0$. Its character is $\chi(\sigma)=n-1-\lvert\supp\sigma\rvert$. The permutation representations of $G$ on $\C^n$ through $\rho$ and $\rho'$ have characters $\chi_\pi=1+\chi\circ\rho$ and $\chi_{\pi'}=1+\chi\circ\rho'$. A pair $(p,q)$ is fixed by $g$ exactly when $p$ and $q$ both are, so the permutation character of $[n]^2$ is $\chi_\pi\chi_{\pi'}$. By Burnside's lemma, the number of orbits of $G$ on $[n]^2$ is $\langle\chi_\pi\chi_{\pi'},1\rangle=\langle\chi_\pi,\chi_{\pi'}\rangle$. As $G$ is transitive on $[n]^2$, this is $1$, so $\psi\circ\rho$ and $\psi\circ\rho'$ have no common irreducible constituent. \Cref{lem:charratio} now bounds the spectral gap by $d^{-1}\left(1-(n-1-\lvert\supp\tau\rvert)/(n-1)\right)=\lvert\supp\tau\rvert/(d(n-1))$.
\end{proof}

\subsubsection{Many coordinates are nearly conjugate}\label{subsub:nearconj}

We write $g^c=c^{-1}gc$ for conjugation, and for a tuple $Z=(z_1,\dots,z_d)$ of permutations we write $Z^c=(z_1^c,\dots,z_d^c)$ and $Z\cdot\sigma=(z_1\sigma,z_2,\dots,z_d)$, the perturbation of the first entry of $Z$ by $\sigma$. For $s\geq0$, two tuples $Z,Z'$ are \emph{$s$-nearly conjugate} if $Z'^c=Z\cdot\tau$ for some $c\in\Sym(n)$ and some $\tau$ with $\lvert\supp\tau\rvert\leq s$. By \Cref{lem:intertwiner}, if two coordinates of a generating $d$-tuple of a power of $\Alt(n)$ are $s$-nearly conjugate, then its spectral gap is at most $s/(d(n-1))$.

We now explain the strategy for the first step, which is to find two nearly conjugate coordinates. We fix a parameter $r$, which is small compared to $n$, and a set $\mathcal P$ of permutations $\sigma$ with $\lvert\supp\sigma\rvert\leq r$, such that $\lvert\supp\sigma^{-1}\sigma'\rvert\geq r/2$ for all distinct $\sigma,\sigma'\in\mathcal P$. We perturb every tuple $Z\in\Alt(n)^d$ by each element of $\mathcal P$. If $\mathcal Y$ is a large set of pairwise non-conjugate generating tuples, then a counting argument gives some $Z$ with two perturbations $Z\cdot\sigma$ and $Z\cdot\sigma'$ that are conjugate to tuples $Y,Y'\in\mathcal Y$. Since $Z\cdot\sigma'=(Z\cdot\sigma)\cdot\sigma^{-1}\sigma'$, the tuples $Y$ and $Y'$ are nearly conjugate, and $r/2\leq\lvert\supp\sigma^{-1}\sigma'\rvert\leq2r$. If $Y$ and $Y'$ are distinct, the upper bound and \Cref{lem:intertwiner} give a spectral gap of at most $2r/(d(n-1))$, which is small. They are the same tuple only if $Z\cdot\sigma$ is conjugate to its own perturbation by $\sigma^{-1}\sigma'$, and the lower bound $r/2$ makes this rare. We first record the three ingredients of this argument, namely an estimate on conjugacy classes, a supply of such perturbations, and a bound on the tuples that are conjugate to a perturbation of themselves. Then we carry out the counting in \Cref{lem:nearconj}.

\begin{lemma}\label{lem:classes}
Let $n\geq120$ and $s>0$. The sum of $|\mathcal K|^{-1}$ over the conjugacy classes $\mathcal K$ of $\Sym(n)$ whose elements $\sigma$ have $\lvert\supp\sigma\rvert\geq s$ is at most $2(2e/\sqrt n)^{s}$.
\end{lemma}

\begin{proof}
Let $\mathcal K$ be a class consisting of the permutations with support of size $j$ and with $b_i$ cycles of length $i$ for each $i\geq2$. Then $|\mathcal K|=\binom nj j!/z$, where
\[
    z = \prod_i i^{b_i}b_i! \leq \prod_i(ib_i)^{b_i} \leq j^{\sum_ib_i} \leq j^{j/2} ,
\]
since $\sum_iib_i=j$, and hence $\sum_ib_i\leq j/2$. As $j!\geq(j/e)^j$ and $\binom nj\geq(n/j)^j$, this gives
\[
    |\mathcal K| \geq \left(\frac{n}{e\sqrt j}\right)^j \geq \left(\frac{\sqrt n}{e}\right)^j .
\]
A class of permutations with support of size $j$ is determined by its cycle type on the support, which is a partition of $j$, so there are at most $2^j$ such classes. Hence the sum in question is at most $\sum_{j\geq s}(2e/\sqrt n)^j$, and this is at most $2(2e/\sqrt n)^{s}$ since $2e/\sqrt n\leq1/2$ for $n\geq120$.
\end{proof}

\begin{lemma}\label{lem:code}
Let $n\geq 80$ and let $r$ be an integer with $3\leq r\leq n/2$. Then there is a set $\mathcal P\subseteq\Alt(n)$ of size at least $n^{r/40}$, such that $\lvert\supp\sigma\rvert\leq r$ for all $\sigma\in\mathcal P$, and $\lvert\supp\sigma^{-1}\sigma'\rvert\geq r/2$ for all distinct $\sigma,\sigma'\in\mathcal P$.
\end{lemma}

\begin{proof}
Put $M=\lceil r/2\rceil$, and let $\mathcal P\subseteq\Alt(n)$ be a maximal set of permutations $\sigma$ with $\lvert\supp\sigma\rvert\leq r$, such that $\lvert\supp\sigma^{-1}\sigma'\rvert\geq M$ for all distinct $\sigma,\sigma'\in\mathcal P$. By maximality, every even permutation $\pi$ with $\lvert\supp\pi\rvert=r$ is of the form $\sigma\eta$ with $\sigma\in\mathcal P$ and $\lvert\supp\eta\rvert<M$. Indeed, if $\pi\in\mathcal P$, take $\sigma=\pi$ and $\eta=1$. Otherwise $\mathcal P\cup\{\pi\}$ violates the defining condition, so there is $\sigma\in\mathcal P$ such that $\eta=\sigma^{-1}\pi$ has $\lvert\supp\eta\rvert<M$. Now, there are at least $\binom nr(r-1)!/2\geq(n/2)^r/(2r)$ even permutations with support of size $r$, for instance the $r$-cycles if $r$ is odd, and the products of a transposition and a disjoint $(r-2)$-cycle if $r$ is even. On the other hand, for each $j$ at most $\binom nj j!\leq n^j$ permutations have support of size $j$, so summing over $j<M$, there are at most $Mn^{M-1}$ permutations with support of size less than $M$. This gives
\[
    |\mathcal P| \geq \frac{(n/2)^r}{2rMn^{M-1}} \geq \frac{n^{(r+1)/2}}{2r^2\cdot2^r} \geq \left(\frac{\sqrt n}{8}\right)^r \geq n^{r/40} ,
\]
where the third step uses $2r^2\leq4^r$, and the last one uses $n^{19/40}\geq8$, which holds for $n\geq80$.
\end{proof}

\begin{lemma}\label{lem:selfconj}
Let $d\geq2, n\geq120, s\geq1$, and let $\eta\in\Alt(n)$ with $\lvert\supp\eta\rvert\geq s$. Then the number of $Z\in\Alt(n)^d$ such that $Z\cdot\eta$ is conjugate to $Z$ under $\Sym(n)$ is at most $8(2e/\sqrt n)^{s/2}|\Alt(n)|^d$.
\end{lemma}

\begin{proof}
Write $S=\Alt(n)$, and suppose that $Z\cdot\eta=Z^a$ with $Z=(z_1,\dots,z_d)$ and $a\in\Sym(n)$, where $a\neq1$ since $\eta\neq1$. Then $a$ centralizes $z_2,\dots,z_d$, which leaves at most $|C_S(a)|$ choices for each of them. Moreover $z_1^a=z_1\eta$ says that $(a^{-1})^{z_1}=\eta a^{-1}$, which leaves at most $|C_S(a)|$ choices for $z_1$. This identity also shows that $\lvert\supp\eta a^{-1}\rvert=\lvert\supp a\rvert$, and since $\eta=(\eta a^{-1})a$, we get $\lvert\supp\eta\rvert\leq2\lvert\supp a\rvert$. So $\lvert\supp a\rvert\geq s/2$, and the number of $Z$ in question is at most
\[
    \sum_a|C_S(a)|^d \leq |S|^{d-2}\sum_{a}|C_S(a)|^2 ,
\]
where $a$ runs over the permutations with $\lvert\supp a\rvert\geq s/2$, and we use $|C_S(a)|\leq|S|$ for all but two factors.\footnote{This crude estimate suffices for our purposes, and it keeps the bound independent of $d$. Keeping all the centralizers would only improve the constant $1/160$ in \Cref{thm:ceiling}.} Grouping the elements $a$ by their conjugacy classes $\mathcal K$ in $\Sym(n)$, and using $|C_S(a)|\leq n!/|\mathcal K|$, we bound this by
\[
    |S|^{d-2}\sum_{\mathcal K}|\mathcal K|\left(\frac{n!}{|\mathcal K|}\right)^2 = 4|S|^d\sum_{\mathcal K}|\mathcal K|^{-1} \leq 8\left(\frac{2e}{\sqrt n}\right)^{s/2}|S|^d ,
\]
where the last step is \Cref{lem:classes}.
\end{proof}

\begin{lemma}\label{lem:nearconj}
Let $d\geq2$ and $n\geq1000$, and let $r$ be an integer with $40\leq r\leq n/2$. Then every set of at least $3\,t_d(\Alt(n))\,n^{-r/40}$ pairwise non-conjugate generating $d$-tuples of $\Alt(n)$ contains two distinct tuples that are $2r$-nearly conjugate.
\end{lemma}

\begin{proof}
Write $S=\Alt(n)$, and let $\mathcal Y$ be a set of $t\geq3\,t_d(S)\,n^{-r/40}$ pairwise non-conjugate generating $d$-tuples of $S$. Let $\mathcal P$ be a subset of size $K=\lceil n^{r/40}\rceil$ of the set in \Cref{lem:code}, so that $n^{r/40}\leq K\leq2n^{r/40}$. Let $\overline{\mathcal Y}$ be the union of the conjugacy classes of the tuples in $\mathcal Y$. Recall from \Cref{rem:hall} that $\Aut(S)=\Sym(n)$ acts freely on the generating $d$-tuples of $S$, so each of these $t$ classes has $n!$ elements, and that $t_d(S)=\phi_d(S)/n!$. Moreover, $\phi_d(S)\geq\phi_2(S)|S|^{d-2}$, as a $d$-tuple generates as soon as its first two entries do. A random pair of elements of $S$ generates $S$ with probability greater than $1-2.468/n\geq2/3$ \cite[Cor.~1.3]{MRD15}, so $\phi_2(S) \geq (2/3)|S|^2$. This gives
\[
    K|\overline{\mathcal Y}| = Kt\cdot n! \geq 3 t_d(S) \cdot n! \geq 2|S|^d .
\]
For $Z\in S^d$, let $N(Z)$ be the number of $\sigma\in\mathcal P$ with $Z\cdot\sigma\in\overline{\mathcal Y}$. By double counting, we have $\sum_ZN(Z)=K|\overline{\mathcal Y}|$. Thus the number of triples $(Z,\sigma,\sigma')$ with $\sigma\neq\sigma'$ and $Z\cdot\sigma,Z\cdot\sigma'\in\overline{\mathcal Y}$ is
\[
    \sum_ZN(Z)\left(N(Z)-1\right) 
    \geq \sum_Z N(Z) - \sum_Z 1
    = K|\overline{\mathcal Y}|-|S|^d \geq |S|^d .
\]
Call such a triple \emph{collapsing} if $Z\cdot\sigma$ and $Z\cdot\sigma'$ are conjugate. Put $Z'=Z\cdot\sigma$ and $\eta=\sigma^{-1}\sigma'$, so that $Z\cdot\sigma'=Z'\cdot\eta$, and the triple is collapsing if $Z'$ and $Z'\cdot\eta$ are conjugate. Note that each triple is determined by $\sigma$, $\eta$ and $Z'$, and $\lvert\supp\eta\rvert\geq r/2$. So by \Cref{lem:selfconj}, with at most $K$ choices for $\sigma$ and then at most $K$ for $\eta$, the number of collapsing triples is at most
\[
    8K^2\left(\frac{2e}{\sqrt n}\right)^{r/4}|S|^d \leq 32(2e)^{r/4}n^{-3r/40}|S|^d,
\]
where we have used $K\leq2n^{r/40}$. This is less than $|S|^d$. Indeed, $(2e)^{1/4}n^{-3/40}<1$ for $n\geq1000$, so the coefficient $32(2e)^{r/4}n^{-3r/40}$ decreases in $r$ and in $n$, and at $r=40$ and $n=1000$ it is $32(2e)^{10}10^{-9}<3/4$. So some triple is not collapsing. For it, $\lvert\supp\eta\rvert\leq2r$, and $Z'$ and $Z'\cdot\eta$ lie in $\overline{\mathcal Y}$ but are not conjugate. So $Z'=Y^a$ and $Z'\cdot\eta=Y'^{a'}$ for two distinct tuples $Y,
Y'\in\mathcal Y$ and some $a,a'\in\Sym(n)$. Thus
\[
Y'^{a' a^{-1}} = (Z'\cdot\eta)^{a^{-1}} = Y\cdot\eta^{a^{-1}},
\]
and hence $Y'$ is conjugate to $Y\cdot\eta^{a^{-1}}$ with $\lvert\supp \eta^{a^{-1}} \rvert=\lvert\supp\eta\rvert$. So $Y$ and $Y'$ are $2r$-nearly conjugate.
\end{proof}

\subsubsection{Conclusion}\label{subsub:proofceiling}

\begin{proof}[Proof of \Cref{thm:ceiling}]
Let $n\geq1000$ and let $40\leq r\leq n/2$. We show that if $t\geq3\,t_d(\Alt(n))\,n^{-r/40}$, then every generating $d$-tuple $\Sigma$ of $\Alt(n)^t$ has spectral gap at most $2r/n$. By \Cref{thm:hall}, the coordinate tuples $Y^{(1)},\dots,Y^{(t)}$ of $\Sigma$ are pairwise non-conjugate generating $d$-tuples of $\Alt(n)$. So \Cref{lem:nearconj} gives $i\neq j$ such that $Y^{(i)}$ and $Y^{(j)}$ are $2r$-nearly conjugate, say $(Y^{(j)})^c=Y^{(i)}\cdot\tau$ with $\lvert\supp\tau\rvert\leq2r$. Let $\rho$ be the projection of $\Alt(n)^t$ onto the coordinate $i$, and $\rho'$ the projection onto the coordinate $j$ followed by conjugation by $c$. Then $\rho'(s_l)=\rho(s_l)$ for $l\geq2$ and $\rho'(s_1)=\rho(s_1)\tau$. Since $i\neq j$ and $\Alt(n)$ is transitive, $\Alt(n)^t$ acts transitively on $[n]^2$ through $\rho$ and $\rho'$. So by \Cref{lem:intertwiner}, the spectral gap of $\Cay(\Alt(n)^t,\Sigma)$ is at most $2r/(d(n-1))\leq2r/n$.

Now let $\epsilon$ and $f(n)$ be as in the theorem. We take $r=\lfloor\epsilon n/4\rfloor$, which lies between $40$ and $n/2$ for $n\geq\max\{1000,160/\epsilon\}$. Then $2r/n<\epsilon$, so
\[
    f(n) < 3t_d(\Alt(n))\,n^{-r/40} \leq 3t_d(\Alt(n))\,n^{1/40}\exp\left(-\frac{\epsilon}{160}n\log n\right) .
\]
Since $\log t_d(\Alt(n))=(1+o(1))(d-1)n\log n$, the claim follows.
\end{proof}

\subsection{Powers of groups of Lie type}\label{subsub:lietype}

The analogue of \Cref{q:alt} for nonabelian finite simple groups $S$ in place of $\Alt(n)$ asks for an expanding sequence of powers $S^{f(S)}$ with $f(S)\geq|S|^\delta$ for a fixed $\delta>0$. For the groups $\PSL_3(\F_p)$ and $\SL_{3n}(\F_2)$, this question has a positive answer, by the following two examples. Both come from property (T).

\begin{example}\label{ex:sl3}
Fix $m\geq1$, let $p_1<p_2<\cdots$ be the primes, and put
\[
    G_n = \PSL_3(\F_{p_n})^{p_n^m} .
\]
Write $p=p_n$. Identify $(\F_p)^{p^m}$ with the ring of all functions $\F_p^m\to\F_p$, with pointwise operations, and consider the evaluation homomorphism
\[
    \Z[x_1,\dots,x_m] \to (\F_p)^{p^m} ,
    \qquad
    f \mapsto \left(a\mapsto \bar f(a)\right) ,
\]
where $\bar f$ is the reduction of $f$ modulo $p$, evaluated at the point $a=(a_1,\dots,a_m)\in\F_p^m$. It is surjective, since every function $\F_p^m\to\F_p$ is a polynomial, the indicator function of a point $a$ being $\prod_i(1-(x_i-a_i)^{p-1})$. For a commutative ring $R$, write $\EL_3(R)$ for the subgroup of $\GL_3(R)$ generated by the elementary matrices $e_{ij}(r)=1+rE_{ij}$ with $i\neq j$ and $r\in R$, where $E_{ij}$ is the matrix unit. The evaluation homomorphism maps $\EL_3(\Z[x_1,\dots,x_m])$ onto $\EL_3((\F_p)^{p^m}) = \SL_3(\F_p)^{p^m}$, since $\SL_3(\F_p)$ is generated by elementary matrices. The latter group maps onto $G_n$. The group $\EL_3(\Z[x_1,\dots,x_m])$ has property (T) \cite[Th.~1.1]{EJ10}, and it is generated by the $d=6(m+1)$ elementary matrices $e_{ij}(1)$ and $e_{ij}(x_l)$ \cite[(6.1)]{EJ10}. So the groups $G_n$ form a $d$-expanding sequence.

Write $S=\PSL_3(\F_p)$. Since $|S|=p^{8+o(1)}$, the number of factors and Hall's ceiling for $d$ generators, given by \Cref{prop:powers}(2), are
\[
    p^m = |S|^{(1+o(1))m/8}
    \qquad \text{and} \qquad
    t_{d}(S) = |S|^{(1+o(1))(6m+5)} .
\]
So the number of factors is a fixed positive power of $|S|$, a fraction of about $1/48$ of the ceiling on a logarithmic scale for large $m$. The number of factors is $p^m=\exp(O(\ell(S)))$, so $\Rep_k(G_n)\leq\exp(O(k))$ for all $k$ by \Cref{prop:powers}(3).

The same construction works with $\PSL_r(\F_p)$ in place of $\PSL_3(\F_p)$, for any fixed $r\geq3$, since $\EL_r(\Z[x_1,\dots,x_m])$ also has property (T).
\end{example}

\begin{example}\label{ex:kassabov}
In \cite[Rem.~2.5]{Pis17}, Pisier records the following example of Kassabov,
\[
    G_n = \SL_{3n}(\F_2)^{2^{n^2}} .
\]
As $n$ varies, these groups are three-generated and form an expanding sequence. This follows by the same procedure as in the previous example, now evaluating the free associative ring $\Z\langle x,y,z\rangle$ at suitable triples of matrices in $\mathrm{M}_n(\F_2)$ instead of polynomials at points, and using $\EL_3(\mathrm{M}_n(\F_2))=\SL_{3n}(\F_2)$. So the groups $G_n$ are quotients of $\EL_3(\Z\langle x,y,z\rangle)$, which has property (T) \cite[Th.~1.1]{EJ10} and is generated by three elements \cite[Rem.~2.5]{Pis17}. Since $|\SL_{3n}(\F_2)|=2^{9n^2+O(1)}$, the number of factors is about the ninth root of the order of a factor. Put $S=\SL_{3n}(\F_2)$ and $k=\ell(S)$, which is of order $2^{3n}$ by \cite{LS74}. Since $\ell(S)<\iota(S)$, \Cref{prop:powers}(1) gives $\ab_k(G_n)=1$ and $\Rep_k(G_n)\geq2^{n^2}=\exp(\Omega(\log^2k))$. So along this expanding sequence the gap function $\Delta_{d,\epsilon}(k)$ is superpolynomial in $k$. On the other hand, $2^{n^2}=\exp(O(\ell(S)))$, so $\Rep_k(G_n)\leq\exp(O(k))$ for all $k$ by \Cref{prop:powers}(3).
\end{example}

We do not know whether a multiplicity that is a fixed positive power of $|S|$ can be reached along an expanding sequence with $S=\PSL_2(p)$, see \Cref{ex:psl2power} and \cite[Question~4.5]{Sab25}. Our two arguments against expansion of large powers of alternating groups both fail for $\PSL_2(p)$. The argument of \Cref{thm:ceiling} turns a pair of coordinates that are conjugate up to a perturbation $\tau$ into a test vector whose Rayleigh quotient is $d^{-1}(1-\operatorname{Re}\chi(\tau)/\chi(1))$, for any nontrivial irreducible character $\chi$, by \Cref{lem:charratio}. In $\Alt(n)$ the standard character makes this small when $\tau$ has small support. In $\PSL_2(p)$ there is no analogue of small support, since every nontrivial element $\tau$ has $|\chi(\tau)|/\chi(1)=O(p^{-1/2})$ for every nontrivial $\chi$ \cite[Table~5.4]{Bon11}. So the Rayleigh quotient is $(1+O(p^{-1/2}))/d$ whatever $\tau$ is, and it does not tend to $0$. The argument of \Cref{prop:random} needs generating tuples with small spectral gap, and for most primes there are none, see \Cref{rem:beta}. So it might well be that powers of $\PSL_2(p)$ are $d$-expanding up to a constant fraction of Hall's ceiling on a logarithmic scale, as in \Cref{ex:sl3,ex:kassabov}, or even at the ceiling itself.

\raggedbottom % the bibliography entries do not stretch
\bibliographystyle{alpha}
\bibliography{repgrowth}

\end{document}